\documentclass[12]{article}
\usepackage[titletoc]{appendix}
\usepackage{fullpage}
\usepackage[utf8]{inputenc}

\usepackage{titlesec} 
\titleformat{\subsection}[runin]
{\normalfont\large\bfseries}{\thesubsection}{0.5em}{}

\usepackage[utf8]{inputenc}
\usepackage{setspace}
\usepackage{tocloft}

\usepackage{graphicx}
\usepackage{amscd,amsmath,amsthm,amssymb}
\usepackage{verbatim}
\usepackage[all]{xy}
\usepackage{color}
\usepackage[colorlinks]{hyperref}
\usepackage[paper=a4paper, margin=2cm]{geometry}
\usepackage{subcaption}
\usepackage{diagbox}

\usepackage{booktabs}

\usepackage{tikz, float} \usetikzlibrary {positioning}
\usetikzlibrary{calc}

\usepackage{makeidx}         
\usepackage{graphicx}        
\usepackage{multicol}        
\usepackage[bottom]{footmisc}

\usepackage{array}

\usepackage{newtxmath}      
\hypersetup{
     colorlinks = true,
     linkcolor = blue,
     anchorcolor = blue,
     citecolor = teal,
     filecolor = blue,
     urlcolor = black
     }

\makeindex       
\title{Matroid stratifications of hypergraph varieties, their realization spaces, and discrete conditional independence models
}

\author{
 Oliver Clarke, Kevin Grace, Fatemeh Mohammadi, and Harshit J Motwani
 }

\date{}
\usepackage[ruled,vlined]{algorithm2e}

\usepackage{amsmath}
\usepackage{amssymb}
\usepackage{amsthm}  
\usepackage{pgf,tikz}
\usetikzlibrary{arrows}
\usepackage{comment}
\usepackage{algpseudocode}
\usepackage{algpascal}
\usepackage{mathrsfs}
\usepackage{hyperref}
\usepackage{multirow}

\tikzstyle{vertex}=[circle, draw, inner sep=0pt, minimum size=6pt]

\newcommand{\CC}{\mathbb{C}}

\newcommand{\KK}{\mathbb{K}}

\newcommand{\MX}{\mathcal{X}}
\newcommand{\MY}{\mathcal{Y}}

\newcommand{\MH}{\mathcal{H}}

\newcommand{\MC}{\mathcal{C}}
\newcommand{\MD}{\mathcal{D}}
\newcommand{\ML}{\mathcal{L}}
\newcommand{\MM}{\mathcal{M}}
\newcommand{\MS}{\mathcal{S}}
\newcommand{\MT}{\mathcal{T}}
\newcommand{\MP}{\mathcal{P}}

\newcommand{\mL}{\mathcal{L}}
\newcommand{\mP}{\mathcal{P}}
\newcommand{\mI}{\mathcal{I}}

\newcommand{\rk}{\textrm{Rank}}
\newcommand{\spn}{\textrm{Conv}}

\newcommand{\Hom}{\textrm{Hom}}
\newcommand{\comb}{\mathrm{comb}}

\newcommand{\ind}[2]{{#1}\perp\!\!\!\perp{#2}}

\DeclareMathOperator{\Gr}{Gr}
\DeclareMathOperator{\rank}{rank}

\newtheorem{theorem}{Theorem}[section]
\newtheorem*{theorem*}{Theorem}
\newtheorem{lemma}[theorem]{Lemma}
\newtheorem{corollary}[theorem]{Corollary}
\newtheorem{proposition}[theorem]{Proposition}
\newtheorem*{proposition*}{Proposition}
\newtheorem{problem}[theorem]{Problem}
\newtheorem{conjecture}[theorem]{Conjecture}
\newtheorem{question}[theorem]{Question}
\theoremstyle{remark}
\newtheorem{remark}[theorem]{Remark}

\theoremstyle{definition}
\newtheorem{example}[theorem]{Example}

\newtheorem{definition}[theorem]{Definition}
\newtheorem*{notation}{Notation}

\begin{document}

\maketitle

\begin{abstract}

We study varieties associated to hypergraphs from the point of view of projective geometry and matroid theory. We describe their decompositions into matroid varieties, which may be reducible and can have arbitrary singularities by the Mnëv--Sturmfels universality theorem. 
We focus on various families of hypergraph varieties for which we explicitly compute an irredundant irreducible decomposition. Our main
findings in this direction are threefold: (1) we describe minimal matroids of such hypergraphs; (2) we prove
that the varieties of these matroids are irreducible and their union is the hypergraph variety; and (3) we
show that every such matroid is realizable over real numbers.
As corollaries, we give conceptual decompositions of various, previously-studied, varieties associated
with graphs, hypergraphs, and adjacent minors of generic matrices. In particular, our decomposition strategy gives immediate
matroid interpretations of the irreducible components of multiple families of varieties associated to conditional independence (CI) models in statistical theory
and unravels their symmetric structures which hugely simplifies the computations.
\end{abstract}

{\hypersetup{linkcolor=black}\setcounter{tocdepth}{1}\setlength\cftbeforesecskip{1.1pt}{\tableofcontents}}

\section{Introduction}\label{sec:intro}

This work develops new connections between 
conditional independence (CI) models 
in statistics \cite{Studeny05:Probabilistic_CI_structures,DrtonSturmfelsSullivant09:Algebraic_Statistics,Sullivant}, projective geometry \cite{richter2011perspectives,lee2013mnev}, the theory of matroids \cite{Oxley, piff1970vector} and their realization spaces 
\cite{mnev1985manifolds, mnev1988universality, sturmfels1989matroid}, and determinantal varieties in commutative algebra \cite{bruns2003determinantal}. 
In particular, we study a family of varieties whose defining equations are indexed by the edges of some hypergraphs,  
and show that their irreducible decompositions lead to well-structured configurations of points and lines.  
Moreover, we apply our results to classes of \textit{conditional independence} varieties which naturally arise in algebraic statistics. 
We also analyze well-known results from the literature \cite{herzog2010binomial, Rauh, HSS, pfister2019primary, clarke2020conditional} and show, in these cases, that the irreducible components of their concerned varieties have concrete interpretations in terms of matroids. 

\smallskip

Our main object of study is the \textit{determinantal hypergraph variety} $V_{\Delta}$ which is associated to a hypergraph $\Delta$; see Definition~\ref{def:hypergraph}. 
The defining equations of $V_\Delta$ are the minors of a matrix $X$ of indeterminates and the corresponding polynomial ideal $I_\Delta$ generalizes many familiar families of determinantal ideals. For example, given integers $t\leq d\leq n$ and the $d\times n$ matrix $X=(x_{ij})$ of indeterminates, the classical ideal
$I_t(X)$ 
studied by Bruns and Conca \cite{bruns2003determinantal} is the ideal of the hypergraph $\Delta$ 
whose edges are all $t$-subsets of $[n] := \{ 1, \dots, n\}$. These ideals, along with many other families \cite{sturmfels1990grobner, herzog2010binomial, Fatemeh, Fatemeh2, pfister2019primary}, are often studied using Gr\"obner bases. When it is possible to find them, Gr\"obner bases are a powerful tool for understanding ideals. For instance, if the initial ideal is square-free, then the original ideal is radical. However, when the hypergraph $\Delta$ lies outside of one of the particularly nice families in the list above, finding a Gr\"obner basis becomes exceedingly difficult. In this paper, we develop a general framework called the \textit{decomposition strategy} for understanding the properties of arbitrary hypergraph ideals, without computing their Gr\"obner bases, using matroid varieties. 
\smallskip

A \textit{matroid} is a combinatorial axiomatization of linear dependence in a vector space. We refer to \cite{Oxley} for a complete introduction to matroids, however, it will now be useful for us to recall realizability and related notions. Given a finite collection of vectors in a fixed vector space, the collection of linearly dependent sets of vectors defines a matroid. If this process can be reversed, i.e.~for a given matroid $M$ we can find such a collection of vectors, then we call these vectors a realization of $M$. We write $\Gamma_M$ for the space of realizations of $M$. The \textit{matroid variety} $V_M$ of $M$ is the Zariski closure of the realization  space; see Definition~\ref{def:prelim_realisation}. Our main tool to study hypergraph varieties is the Decomposition Theorem which provides a matroid stratification for $V_\Delta$. One can think of this theorem as an altered version of the stratification of the Grassmannian by matroids, which was studied in \cite{gelfand1987combinatorial}. 

\begin{theorem*}[Theorem~\ref{thm:general_intersection}] The variety associated to the hypergraph $\Delta$ is the union of matroid varieties. The union is taken over all realizable matroids $M$ whose dependent sets contain the edges of $\Delta$.
\end{theorem*}

Many of the matroid varieties appearing in the decomposition above are redundant. So, to refine this decomposition, we introduce the notion of \textit{combinatorial closures} of matroids; see Definition~\ref{def:comb_closure}. More precisely, the combinatorial closure $V^{\comb}_M$ of a given matroid $M$ is the hypergraph variety associated to its set of circuits. 

We recall that a \textit{minimally dependent matroid} for a given hypergraph $\Delta$ is a realizable matroid $M$ whose dependent sets contain $\Delta$ as a subset and is minimal among all such matroids. In other words, there does not exist a realizable matroid $N$ such that $\Delta\subseteq \MD(N) \subsetneq \MD(M)$. 

\begin{proposition*}[Proposition~\ref{prop:comb_decomposition}]
The associated variety of a hypergraph $\Delta$ is the union of combinatorial closures of matroid varieties, where the union is taken over all minimally dependent matroids for $\Delta$.
\end{proposition*}

We conjecture that combinatorial closures are sufficient to find the irredundant irreducible decomposition of any hypergraph variety. In fact, we show that this is the case for some of the well-known examples of hypergraph varieties such as those corresponding to binomial edge ideals \cite{herzog2010binomial, Rauh}, conditional independence ideals with hidden variables \cite{clarke2020conditional, ollie_fatemeh_harshit} and ideals of adjacent minors \cite{HSS}. In all of these cases, each combinatorial closure is equal to its \textit{central component} which makes the computations easier.
One of the families of hypergraphs we study is the \textit{consecutive forest} hypergraphs; see Definition~\ref{def:consecutive_forest_hyp}. We will see that the combinatorial closures have non-central components. To compute the irreducible decomposition of $V_\Delta$, we apply the following method and note the corresponding sections for consecutive forest hypergraphs.

\medskip

\noindent \textbf{Decomposition Strategy} \S\ref{sec:decomposition_strategy}{\bf .}\label{page:decom}

\begin{itemize}
    \item[(i)] Identify minimally dependent matroids for $\Delta$ (\S\ref{sec:prime-collections-for-m=3}).
    
    \item[(ii)] For each minimally dependent matroid $M$, write the combinatorial closure $V_M^{\comb}$ as a minimal union of matroid varieties (\S \ref{sec:forest_like_configurations}).
    
    \item[(iii)] For each matroid variety appearing in step (ii), show that it is irreducible (\S\ref{sec:irred_configuration}).
    
    \item[(iv)] Determine redundancy of non-central components of combinatorial closures in the resulting decomposition and show the matroids are realizable (\S\ref{sec:comb_closure_forestlike} and \S\ref{sec:forestlike_realizable}).
   \end{itemize}

We are able to completely go through this strategy for various family of hypergraphs in \S\ref{sec:irred_configuration}-\S\ref{sec:Delta_st_main}. In particular, for Step~(iii) above, we show that the realization spaces of certain point and line configurations are irreducible; see \S\ref{sec:irred_configuration}. 
These are a family of matroids whose varieties appear as irreducible components of many examples of 
hypergraph varieties. We will prove Theorem~\ref{thm:lineArrangementBuildUpGeneral} which allows us to inductively build up configurations with irreducible varieties, and use this to show that all configurations with at most $6$ lines have irreducible varieties.

\begin{theorem*}[Theorems~\ref{thm:k=6_irreducible} and~\ref{thm:lineArrangementBuildUpGeneral}]
Suppose that $\MC'$ is a point and line configuration with irreducible realization space $\Gamma_{\MC'}$. If $\MC$ is obtained from $\MC'$ by adding a single line passing through at most two intersection points of $\MC'$, then $\Gamma_\MC$ is irreducible. In particular, the realization spaces of configurations with at most $6$~lines~are~irreducible.
\end{theorem*}

The conclusion of the decomposition strategy is the following characterization of the minimal irreducible decomposition of the associated variety of each consecutive forest hypergraph.

\begin{theorem*}[Theorems~\ref{thm:forestlike_comp} and~\ref{thm:adj-tree-decomp}]
The irredundant irreducible components of the consecutive forest variety $V_{\Delta_G}$ are in one-to-one correspondence with the minimally dependent matroids for the hypergraph $\Delta_G$. In particular, for each minimally dependent matroid, 
the non-central components of its combinatorial closure are completely characterized and are redundant in the irreducible decomposition.
\end{theorem*}

Our decomposition strategy gives immediate interpretations of the prime decompositions of families of ideals for which it is not feasible to calculate an explicit Gr\"obner basis due to the time complexity of the algorithm and the current available hardware. In \cite{pfister2019primary}, which is further explained in Example~\ref{ex:andreas}, the authors calculated the prime decomposition of a hypergraph ideal which contains $16$ edges of size $3$. This computation is at the limit of what is currently possible on available hardware. In this case, the prime components of this ideal have a straight-forward interpretation in terms of configurations of $12$ points in the projective plane.

Moreover, in \S\ref{sec:Delta_st_main} we describe the hypergraphs $\Delta^{s,t}$ which arise from the study of conditional independence (CI) statements; see \S\ref{sec:applications_CI} and \cite{Fink, clarke2020conditional, ollie_fatemeh_harshit}.
We give a summary of the known cases of $V_{\Delta^{s,t}}$ in Remark~\ref{rem:Known_Results} and Table~\ref{tab:min_dep_examples}. The family of varieties $V_{\Delta^{s,t}}$ for a fixed $s, t$ can be studied effectively by understanding finite families of matroids. In Table~\ref{tab:combinatorial_types_kl}, we count the number of combinatorial types of dependent matroids for $\Delta^{s,t}$. We show that as certain parameters increase, the number of combinatorial types eventually stabilizes which simplifies the computational task.
The associated variety of $\Delta^{s,t}$ determines a statistical model corresponding to a collection of CI statements with hidden variables. 
The irreducible components of these varieties give information about additional constraints satisfied by distributions within the given CI model. We describe the irreducible components of these varieties in terms of so-called \textit{grid matroids}.
The first step to understanding such decompositions is to find the set of minimally dependent matroids for $\Delta^{s,t}$.

\smallskip

When the dimension of the ambient space $d$ is low enough, we show that:

\begin{theorem*}[Theorem~\ref{thm:s-t-3}]
If 
$t\leq d\leq s+t-3$, then there is a unique minimal matroid for $\Delta^{s,t}$.
\end{theorem*}

In particular, the minimal matroid from the above theorem corresponds to the special irreducible component of the CI varieties studied in \cite{clarke2020conditional, ollie_fatemeh_harshit}, which leads to the {\em intersection axiom} for CI models. This component is particularly important for the inference problem in statistics, as it gives information about additional constraints on distributions with \textit{full support} which satisfy the given CI statements. 

When there is no bound on $d$, we show that as the parameters for this hypergraph are allowed to become large, we obtain every matroid among the minimally dependent matroids for $\Delta^{s,t}$, up to a mild equivalence.

\begin{theorem*}[Theorem~\ref{thm:hardness}]
For every matroid $M$, there exists a hypergraph of the form $\Delta^{s,t}$ and a dependent matroid $M'$ for $\Delta^{s,t}$ such that a \textit{restriction} of $M'$ is isomorphic to $M$.
\end{theorem*}

We prove this theorem using an algorithm from \cite{M14} to explicitly construct the matroid $M'$. This theorem shows that understanding the varieties of a general CI model is very difficult since matroid varieties satisfy various \textit{universality} results. Similarly, determining the realizability of a CI model is difficult as the matroids of its components might not be 
 $\mathbb{R}$-realizable. More generally, we propose the following computational problem.
\begin{question}\label{question}
Find the irreducible components of $V_{\Delta^{s,t}}$ and determine whether they contain a rational point.
\end{question}

\smallskip 

We conclude by 
highlighting some of the difficulties arising in the study of Question~\ref{question}, in particular in determining the irreducibility of matroid varieties. 
Following \cite{bokowski1989computational, richter1999universality}, consider a collection of vectors $v_1,\ldots,v_s$ in a finite dimensional vector space $V$ over a field $\mathbb{K}$, which realizes a matroid $M$. They also define a hyperplane arrangement $H_{v_1},\ldots,H_{v_s}\subset V^*$ in the dual vector space. The combinatorial type of this hyperplane arrangement is defined by $M$, thus, one can think of the realization space $\Gamma_M$ as a parameter space of hyperplane arrangements of fixed combinatorial type.
In \cite{rybnikov2011fundamental}, Rybnikov constructs a pair of combinatorially equivalent hyperplane arrangements whose complements have non-isomorphic fundamental groups. So, in general, it is not possible to study matroid varieties by picking a single generic point in the variety. 

Unfortunately, very little can be said about the geometry of matroid varieties in general. The Mn\"ev--Sturmfels Universality Theorem shows that matroid varieties satisfy \textit{Murphy’s Law in Algebraic Geometry}. Specifically, given any singularity, appearing in a semi-algebraic set, there is a matroid variety with the same singularity up to a mild equivalence; see \cite{mnev1985manifolds, mnev1988universality, sturmfels1989matroid, lee2013mnev}. In fact, point and line configurations already exhibit all singularities. More precisely, the equivalence above is defined on pointed schemes generated by $(X, p) \sim (Y, q)$ when there exists a smooth morphism $(X, p) \to (Y, q)$. 
Then, the universality theorem states that point and line configurations exhibit all equivalence classes of this relation called singularity types. 

\medskip

We also note that the defining equations of matroid varieties are in general very difficult to calculate. These equations often arise from geometric constraints satisfied by the matroid. Some of them can be interpreted as rank conditions on certain submatrices, however this is not true in general. 
Example~\ref{ex:combthreelines} shows the smallest matroid for which \textit{non-determinantal} equations appear. Proving that such a polynomial constraint holds is often achieved by finding an equivalent condition in the Grassmann-Cayley algebra \cite{sitharam2017handbook, sidman2019geometric}. In the context of CI models, such conditions give further constraints on distributions satisfying the given CI statements.

\medskip\noindent {\bf Outline of paper.} 
In \S\ref{sec:pre},  
we introduce the key concepts used throughout the paper. In particular,  
we fix our notations for hypergraph, matroid varieties and introduce the dependence order on matroids. In \S\ref{sec:decom}, we prove the decomposition theorem 
and introduce the combinatorial closures of matroid varieties. In \S\ref{sec:irred_configuration}, we define point and line configurations and prove that, for certain families, their realization spaces are irreducible. In \S\ref{sec:forest_like_configurations}, we study the family of forest-like point and line configurations. We give a complete characterization of the irreducible components of their combinatorial closures using perturbation arguments. In \S\ref{sec:adjtree}, we apply the decomposition strategy to the family of consecutive forest hypergraphs. As a result, we prove Theorem~\ref{thm:adj-tree-decomp} which gives an irredundant irreducible decomposition of the hypergraph variety.
In \S\ref{sec:Delta_st_main}, we introduce grid matroids and apply them to describe the irreducible components of the variety of $\Delta^{s,t}$; see Tables~\ref{tab:combinatorial_types_kl} and \ref{tab:min_dep_examples}. We then proceed to use algorithmic procedures to prove Theorem~\ref{thm:hardness}. In \S\ref{sec:Applications}, we show how the hypergraph $\Delta^{s,t}$ arises in algebraic statistics in the study of CI statements with hidden variables. We conclude by explaining how our results may shed light on a conjecture by Matúš in the context of CI models.

\section{Preliminaries}\label{sec:pre}

\subsection{Hypergraph varieties.}\label{sec:hypergraph}

Let $\mathbb{K}$ be a field, $d\leq n$ be two positive integers, 
$X = (x_{ij})$ be a $d\times n$ matrix of indeterminates and $R =\mathbb{K}[X]$ be the polynomial ring over $\mathbb{K}$ in the indeterminates $x_{ij}$. It is often convenient to write determinants of submatrices of $X$ as $[I | J]_X$ where $I$ and $J$ are respectively the sets of rows and columns of the submatrix. If $I = [d]$, that is the submatrix covers all rows of $X$, then we write $[J]$ for $[I | J]_X$.
We denote by $x_i$ the $i^{\rm th}$ column of $X$ and by $X_F$ the submatrix of $X$ with columns indexed by $F\subseteq[n]$.

\begin{definition}
\label{def:hypergraph}
A (simple) hypergraph $\Delta$ on the vertex set
$[n]$ is a subset of the power set $2^{[n]}$. We assume that no proper subset of an element of $\Delta$ is in $\Delta$.
The elements of $\Delta$ are called (hyper)edges.  
\begin{itemize}
    \item The \emph{determinantal hypergraph ideal} of $\Delta$ is
  \begin{equation*}
    I_{\Delta}= \big\langle [A|B]_X : A\subseteq [d], B\in\Delta, |A| = |B|  \big\rangle \subset R.
  \end{equation*}
  \item The variety of $\Delta$ is the zero set of $I_{\Delta}$ which is given by 
  \[
  V_{\Delta}=\{X\in \mathbb{C}^{d\times n}:\ \rk(X_F) < |F| \text{ for each }F \text{ in } \Delta\}.
  \] 
  \end{itemize}
\end{definition}
The ideal $I_{\Delta}$ and its variety $V_{\Delta}$ depend on the value $d$, i.e.~the dimension of the ambient space. However, to keep our notation concise, we suppress $d$ from our notation. Unless otherwise stated, all results for hypergraph ideals and their varieties hold for all $d$ as long as $d \ge \max\{|F| : F \in \Delta\}$.

\begin{problem}\label{P1} 
Find an irredundant irreducible decomposition of $V_{\Delta}$ for any hypergraph $\Delta$. 
\end{problem}

\subsection{Matroid varieties.} \label{sec:pre_mat}
In this subsection, we recall the definitions of the realization space of a matroid and its associated variety. We refer the reader to \cite{Oxley} for basic definitions concerning matroids.
\begin{definition}\label{def:prelim_realisation}
Let $M$ be a matroid on $[n]$ of rank $r$ and let $d \ge r$. If $\mathbb{K}$ is a field, a \emph{realization} of $M$ in $\mathbb{K}^d$ is a collection of vectors $X = \{x_1, \dots, x_n \} \subset \mathbb{K}^d$ such that 
\[
\{x_{i_1}, \dots, x_{i_p}\} \subset X \text{ is linearly dependent} \iff \{i_1, \dots, i_p \} \text{ is a dependent set of } M.
\]
If such a collection of vectors exists, we say that the matroid is \emph{realizable} over $\mathbb{K}$. (Other words used interchangeably with \emph{realizable} include \emph{representable} and \emph{linear}.) In this paper, if realizability is discussed without specifying $\mathbb{K}$, then $\mathbb{K}=\CC$.
The \emph{realization space} of $M$ in $\CC^d$ is
\[
\Gamma_{M} = \{X \subset \CC^d : X \text{ is a realization of } M \}.
\]
\end{definition}
So each element of $\Gamma_{M}$ is naturally identified with a $d \times n$ matrix over $\CC$.

\begin{definition}
The \emph{matroid variety} $V_{M} = \overline{\Gamma_{M}}$ is the Zariski closure of the realization space of $M$. We denote  $I_{M} = I(V_{M}) \subseteq \CC[X]$ for the corresponding ideal where $X = (x_{i,j})$ is a $d \times n$ matrix of indeterminates.
\end{definition}

Note that $I_{M}$ is a radical ideal. Similarly to hypergraph ideals and their varieties, we will always use $d$ to denote the dimension of the ambient space. In order to simplify our notation, we suppress $d$ from the notation of the realization space of matroids and their ideals and varieties.

\subsection{Minimal matroids.}\label{sec:min}

We denote by $\le$ the partial order on sets given by inclusion. We extend this notion to matroids by identifying them with their collection of dependent sets. We will denote the collection of dependent sets of a matroid $M$ by $\MD(M)$. So, given two matroids $M_1$ and $M_2$, if  $\MD(M_1) \subseteq \MD(M_2)$ then we write $M_1 \le M_2$. This is the \emph{dependency order} on matroids. However, we caution the reader that this is precisely opposite of the \emph{weak order} 
in the matroid literature \cite{Oxley, bokowski1989computational}.
Suppose that $M$ is a matroid on ground set $E$. For any given collection of subsets $\MD$ of $E$, we say that $M$ is \emph{dependent} for $\MD$ if $\MD \subseteq \MD(M)$ and write $\MD \le M$. If there does not exist a matroid $N$ on $E$ such that $\MD\subseteq\mathcal{D}(N)\subsetneq \mathcal{D}(M)$, then we say that $M$ is \emph{minimally dependent} for $\MD$, i.e.~$M$ is a smallest matroid such that $\MD \le M$.

\begin{example}
Any matroid $M$ is minimally dependent for its collection of circuits $\mathcal{C}$. In this case, $M$ is the unique such matroid. In general, there can be many different minimally dependent matroids for a collection $\MD$. For example, if $E=[5]$, and $\MD=\{1234,1235\}$, then the uniform matroid $U^3_5$, whose circuits are all $4$-subsets of $E$, and the matroid with a single circuit $123$ are both minimally dependent for $\MD$; see Figure~\ref{fig:Minimally Dependent}.

\begin{figure}[ht]
\centering
\includegraphics[scale=0.6]{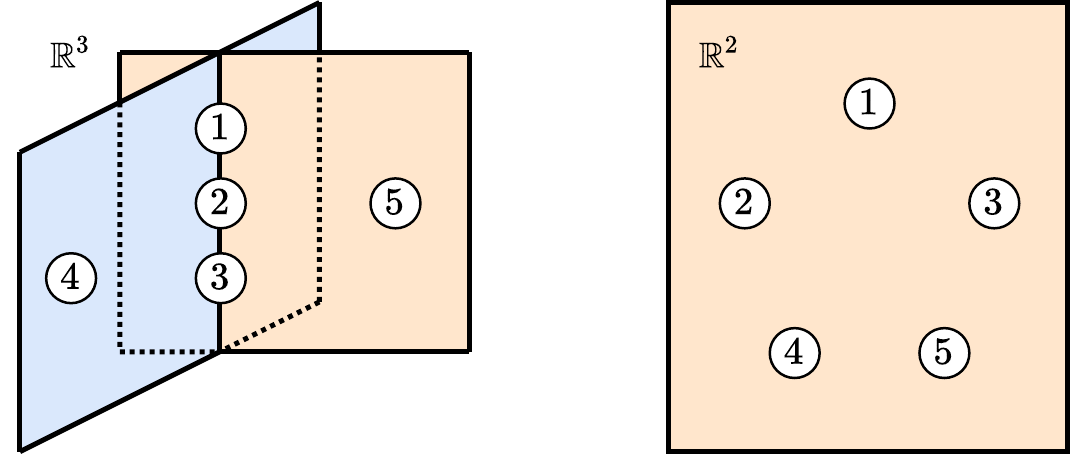}
\caption{Geometric representations for two {\em minimally dependent matroids} for $\MD=\{1234,1235\}$. 
}
\label{fig:Minimally Dependent}
\end{figure}
\end{example}

We will be interested in finding out the minimal matroids for a given collection of sets. In general, this is difficult. However, we are able to do this for some specific families; see \S\ref{sec:adjtree}.

\begin{notation}
We denote the corresponding simple hypergraph of any collection $\mathcal{D}$ of subsets of $[n]$ as:
\begin{eqnarray}\label{def:min}
\min(\mathcal{D})=\{D\in\MD: \textrm{ no set } D'\in\MD \textrm{ exists such that } D'\subsetneq D\}.
\end{eqnarray}
\end{notation}

\begin{remark}
\label{rem:big-circuits}
Let $\Delta'$ be a simple hypergraph all of whose edges have size at most $d$, and let $\Delta=\min(\Delta'\cup\binom{[n]}{d+1})$. 
Note that $I_{\Delta'}=I_{\Delta}$ and that $V_{\Delta'}=V_{\Delta}$. Moreover, $\Gamma_{M}=\emptyset$ if $M$ has rank greater than $d$. Therefore,  
finding the minimally dependent matroids of rank at most $d$ for $\Delta$ is  equivalent to finding 
those for $\Delta'$.
\end{remark}

\section{Matroid stratifications of hypergraph varieties}\label{sec:decom}
\subsection{The general decomposition theorem.}
We now use our notation from \S\ref{sec:pre} to decompose any hypergraph variety into certain matroid varieties. In subsequent sections, we will investigate techniques which will allow us to strengthen this statement to an irredundant irreducible decomposition. The following theorem may be interpreted from the perspective of the stratification of Grassmannians by matroid varieties.~See~Remark~\ref{remark: decomposition via stratification}.

\begin{theorem}[Decomposition Theorem]\label{thm:general_intersection}
Let $\Delta \subseteq 2^{[n]}$ be a hypergraph on $[n]$ and let $\MM(\Delta)$ be the collection of realizable matroids whose collections of dependent sets contain $\Delta$. Then,
\[
V_{\Delta} = \bigcup_{M \in \MM(\Delta)} V_{M}.
\]
In particular, we have $\sqrt{I_{\Delta}} = \bigcap_{M \in \MM(\Delta)} I_{M}$.
\end{theorem}

\begin{proof}
Let $M \in \MM(\Delta)$. We begin by showing that $I_{\Delta}\subseteq I_{M}$. 
By definition, $I_{\Delta}$ is generated by minors, so let $f \in I_{\Delta}$ be any such minor. We may assume without loss of generality that $f$ is a maximal minor of the submatrix $X_F$ for some $F \in \Delta$. For each matroid $M\in \MM(\Delta)$ and each point $A$ in the realization space $\Gamma_{M}$, the columns of $A_F$ are linearly dependent since $F$ is a dependent set of $M$. Thus $f$ vanishes on $\Gamma_{M}$, and by definition, $f$ vanishes on its Zariski closure which is $V_{M}$. Thus, we have that $f \in I(V_{M}) = I_{M}$, so $I_{\Delta} \subseteq I_{M}$. Hence, $V_{M}\subseteq V_{\Delta}$. 
Let $A$ be any point in the variety $V_{\Delta}$. The matrix $A$ is a realization of a matroid, which we denote $M_A$, on $[n]$ in $\CC^d$. For each $F \in \Delta$, the ideal $I_{\Delta}$ contains all $|F|$-minors of $X_F$, so all $|F|$-minors of $A_F$ vanish. Therefore, the columns of $A_F$ are linearly dependent. Hence $F$ is a dependent set in $M_A$, so $M_A \in \MM(\Delta)$.
\end{proof}

\vspace{-4mm}

\subsection{Combinatorial closure.}
\label{sec:Combinatorial closure}
In Theorem~\ref{thm:general_intersection}, we decomposed the hypergraph variety $V_\Delta$ into matroid varieties. But this decomposition may have redundant components, because $V_{M'} \subset V_M$ for some $M, M'\in \MM(\Delta)$. To solve Problem~\ref{P1}, we must determine the matroids that are necessary in the decomposition. It is clear that all matroids which are minimal with respect to the dependency order are necessary in the decomposition. One could conjecture that these are enough. Unfortunately this is \emph{not} true, as we will see in Example~\ref{ex:combthreelines}. Nevertheless, we will define an alternative decomposition for which the minimal matroids correspond to irredundant parts of a decomposition. We do this by introducing the \emph{combinatorial closure} of a matroid. This is a generalization of the \textit{weak realization space} of a matroid \cite[page~70]{bokowski1989computational} to an ambient space of dimension greater than the rank of the matroid.

In the following remark we connect our setting to that of the stratification of the Grassmannian by matroids.

\begin{remark}\label{remark: decomposition via stratification}
The Grassmannian $\Gr(d,n+d)$ admits a stratification $\Gr(d,n+d) =  \bigcup_{\mathcal{M}} S_{\mathcal{M}}$ by matroids \cite{gelfand1987combinatorial, sturmfels1989matroid}. Here, $S_{\mathcal{M}} := \{ x \in \Gr(d,n+d) \mid \mathcal{M}_x = \mathcal{M}\}$ is the matroid stratum of $\MM$, where $\mathcal{M}_x$ is the matroid associated to the point $x \in \Gr(d,n+d)$. Let us restrict to the standard affine open patch $U_0$ of $\Gr(d,n+d) $ where the matrices are of the form 
$
[M \ I]
$
such that $M$ is a $d \times n$ matrix and $I$ is the $d \times d$ identity matrix. We identify the ambient space $\mathbb{C}^{d \times n}$ with $U_0$ as they are isomorphic. After this identification, we can see $V(I_{\Delta})$ as the closed embedding of $U_0$ given by the vanishing of the corresponding Pl\"ucker coordinates.

For example, take $\Delta = \{12, 134 \}$ and $d = 3$. Now $V(I_\Delta)$, is a closed subset of the standard affine open patch $U_0$ of Grassmannian $\Gr(3, 7)$. The equations of $V(I_\Delta)$ in terms of Pl\"ucker coordinates are given by $P_{123} = 0, P_{124} = 0, P_{125} = 0, P_{126} = 0, P_{127} = 0, P_{134} = 0$.
So, $V(I_{\Delta})$ can be written in the following form:
\[ V(I_{\Delta})\cap \Gr(d,n+d) \cap U_0  = \bigcup_{\mathcal{M}} S_{\mathcal{M}} \cap V(I_{\Delta}) \cap U_0.\]
In this setting, we may restate Theorem~\ref{thm:general_intersection} as
$
\bigcup_{\mathcal{M}} S_{\mathcal{M}} \cap V(I_{\Delta}) \cap U_0 = 
\bigcup_{M \in \mathcal{M}(\Delta)} V_{M}.$
\end{remark}

\begin{definition}\label{def:comb_closure}
We define the \emph{combinatorial closure} $V_{M}^{\comb}$ of a matroid $M$ to be the union of all matroid varieties $V_{M'}$ for which $M'\geq M$. In other words,
$$
V_{M}^{\comb}=\bigcup_{M'\geq M} V_{M'}
$$ 
We will denote the ideal of the combinatorial closure by $I_{M}^{\comb}$.  
\end{definition}

\begin{remark}
\leavevmode
\begin{itemize}
    \item[(i)] Note that $V_{M}^{\comb}$ might not be a matroid variety itself. The inclusion $V_{M}\subseteq V_{M}^{\comb}$ holds in general, however the equality might not hold. (See Example~\ref{ex:combthreelines} below.) We will call $V_{M}$ the \emph{central component} of $V_{M}^{\comb}$ and matroid varieties which intersect the complement will be called \emph{non-central components}.
    
    \item[(ii)] The combinatorial closure is the closure of the union of realization spaces, i.e.~$V_{M}^{\comb} 
    =\overline{\bigcup_{M'\geq M} \Gamma_{M'}}$. This follows from the topological fact that a finite union of closures coincides with the closure of a union~of~sets. 
\end{itemize}
\end{remark}

The ideal  of the combinatorial closure of a matroid $M$ can be seen as follows. The collection of circuits $\MC(M)$ of $M$ can be considered as a hypergraph. So the hypergraph ideal $I_{\MC(M)}$ for a given $d$, is defined as: 
\[
I_{\MC(M)}=\big\langle [A|B]_X:\ A\subseteq [d], B\in\MC(M), |A| = |B|  \big\rangle.
\]

\begin{lemma}\label{lem:com}
For any matroid $M$ we have that
\[
V_{M}^{\comb} =V_{\MC(M)}\quad\text{ or equivalently }\quad
I_{M}^{\comb} =\sqrt{I_{\MC(M)}}\ .
\]
\end{lemma}
\begin{proof}
Let $X\in V_{M}^{\comb}$ be a collection of vectors, then $X\in \Gamma_{M'}$ for some $M'\geq M$. So, any $B\in\MC(M)$ is a dependent set of $M'$, and therefore $[A|B]_X =0$ for any $A\subseteq [d]$ with $|A|=|B|$. This shows that $V_{M}^{\comb} \subseteq V_{\MC(M)}$.

Similarly, let $X$ be a point in $V_{\MC(M)}$, and let $M_X$ be the matroid represented by $X$. For any $B\in\MC(M)$ and any subset $A\subseteq [d]$ with $|A|=|B|$, we have $[A|B]_X =0$. Therefore, any $B\in\MC(M)$ is also a dependent set of $M_X$, hence $M_X\geq M$. This shows that $V_{\MC(M)} \subseteq V_{M}^{\comb}$.
\end{proof}

In general, the combinatorial closures of matroid varieties are {\em reducible}.

\begin{figure}[h]
    \centering
    \includegraphics[scale = 0.8]{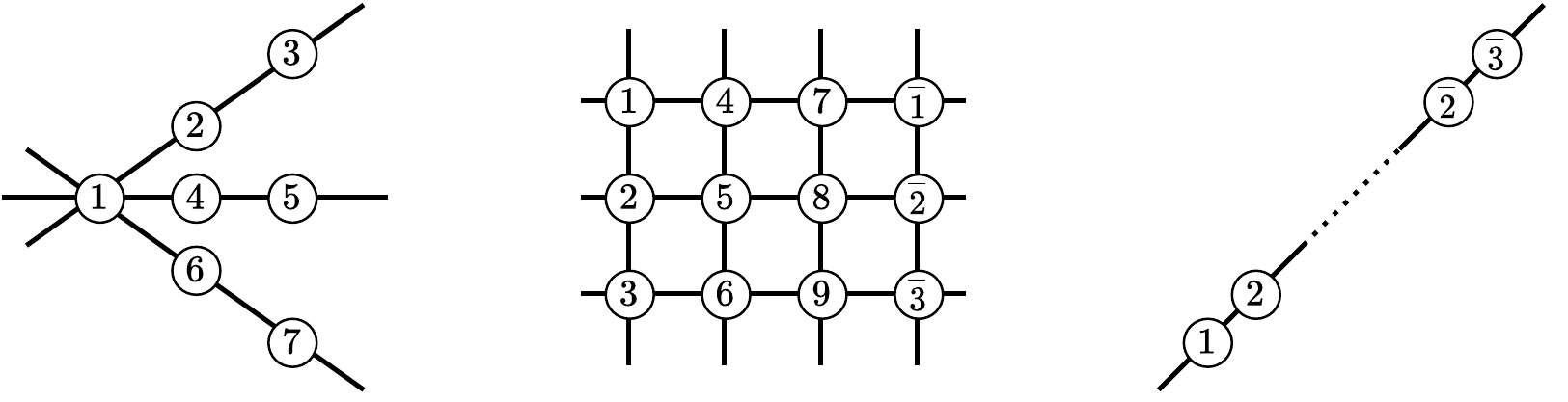}
    \caption{(Left) A realization of $M$ from Example~\ref{ex:combthreelines}.
    (Center and Right) Realizations of matroids corresponding to the prime components of $I_{\Delta}$ in Example~\ref{ex:andreas}.
    }
    \label{threelines}
\end{figure}
\vspace{-3mm}
\begin{example}\label{ex:combthreelines}
Let $d = 3$, $n = 7$ and $\Delta = \min(\{123, 145, 167 \} \cup \binom{[7]}{4})$. It is easy to check that $\Delta$ is the collection of circuits for a matroid $M$, hence  $M$ is the unique minimal matroid for $\Delta$. By Lemma~\ref{lem:com}, the ideal of the combinatorial closure is given by the radical ideal $I_{\MC(M)} = I_{\Delta} = \langle [123], [145], [167] \rangle \subseteq \CC[x_{1,1}, \dots, x_{3,7}]$.
Using Macaulay2 \cite{M2}, we find that $I_{\MC(M)}$ has two prime components given by
\[
I_{\MC(M)} = I_1 \cap I_2 =
\langle x_{1,1}, x_{2,1}, x_{3,1} \rangle \cap 
\langle [123], [145], [167], [234][567] - [235][467] \rangle\ . 
\]
By inspection, the ideal $I_1$ is the ideal of the matroid $M'$ that has a single circuit $1$. In this case, we say that $M'$ is obtained from $M$ by setting $1$ to be a loop. On the other hand, $I_2$ is the ideal of the matroid $M$ and is called the central component of the combinatorial closure. The generator $[234][567] - [235][467]$ of $I_2$ is a geometric condition satisfied by realizations of $M$. More precisely, it is a condition satisfied by six generic points which lie on three lines that intersect at a common point, as shown in Figure~\ref{threelines}.

\medskip

In Lemma~\ref{lem:star_irred_comps}, we generalize this example to an arbitrary number of lines meeting at a point.

\end{example}

\begin{remark}

We note that Example~\ref{ex:combthreelines} has been studied in \cite[Theorem~4.28]{bokowski1989computational} to resolve a problem of White, by showing that there exists a matroid $M$ and a dependent matroid $N$ (Section~\ref{sec:min}) such that $\overline{\Gamma_M} \cap \Gamma_N = \emptyset$. In Section~\ref{sec:adjtree}, we provide a family of matroids $M$ and dependent matroids $N$ such that $\overline{\Gamma_M} \cap \Gamma_N \neq \emptyset$. In particular, the perturbation procedure (Lemma~\ref{lem:perturbation_procedure}) allows us to approximate realizations of $N$ with realizations of $M$. This  gives a way to generate families of dependent matroids that give a positive answer to White's problem.

\end{remark}

We recall the definition of ideal quotients and saturation. See, e.g.~\cite[\S4]{CoxLittleOShea} for more details. 

\begin{lemma}\label{prop:saturation_quotient_radical}
Consider a polynomial ring $R = \KK[x_1, \dots, x_n]$. Let $I\subset R$ be a radical ideal and $J\subset R$ any ideal. Then $I: J = I : J^\infty$, where \[
I : J = \{f \in R : fJ \subseteq I \}\quad \text{ and }\quad
I : J^\infty = \bigcup_{i \ge 1} \, (I : J^i).
\]
\end{lemma}

For a matroid $M$, let us define its \emph{bases ideal} $J_M$ to be 
$$
J_M = \sqrt{\prod_{B \text{ basis of }M} \big\langle [A|B]_X:\ A\subseteq [d], |A| = |B|  \big\rangle}.
$$
If $M$ has rank $d$, then $J_M$ is 
generated by the product of all maximal minors of $X$ whose columns are bases~of~$M$.

\begin{proposition}\label{prop:saturation}
The ideal of the matroid variety can be obtained by saturating $I(V_M^{\comb})$ with respect to $J_M$:
$$
I_M = I(V_M^{\comb}):J_M^\infty.
$$
\end{proposition}
\begin{proof}

Geometrically, the saturation $I(V_M^{\comb}):J_M^\infty$ corresponds to the ideal of $\overline{V_M^{\comb}\setminus V(J_M)}$; see \cite[\S4, Theorem 10(iii)]{CoxLittleOShea}. By Lemma~\ref{prop:saturation_quotient_radical}, if $I$ is a radical ideal, then $I : J = I : J^\infty$.
So it is enough to show that the difference $V_M^{\comb}\setminus V(J_M)$ is the realization space $\Gamma_M$. 
For this, notice that the variety $V(J_M)$ consists of collections of vectors for which at least one of the bases $B$ of $M$ is dependent. Therefore, the difference $V_M^{\comb}\setminus V(J_M)$ consists of collections of vectors $V_M^{\comb}$ for which all bases $B$ of $M$ are independent, which is by definition the realization space $\Gamma_M$.
\end{proof}

The above observation is shown in \cite[Proposition~2.1.3]{sidman2019geometric} 
under the identification of $\CC^{d \times n}$ with an open affine patch of the Grassmannian, as explained in Remark~\ref{remark: decomposition via stratification}.

\begin{proposition}\label{prop:comb_decomposition}
Let $\Delta \subseteq 2^{[n]}$ be a hypergraph on $[n]$ and let $\MM'(\Delta)$ be the collection of minimal realizable matroids whose dependent sets contain $\Delta$. Then
\[
\sqrt{I_{\Delta}} = \bigcap_{M \in \MM'(\Delta)} I_{M}^{\comb}
\quad \textrm{or equivalently} \quad
V_{\Delta} = \bigcup_{M \in \MM'(\Delta)} V^{\comb}_{M}.
\]
\end{proposition}

\begin{proof}
Following the notation of Theorem~\ref{thm:general_intersection}, let $\MM(\Delta)$ be the collection of realizable matroids whose collections of dependent sets contain $\Delta$. We have that $\MM'(\Delta)$ consists of the minimal matroids in $\MM(\Delta)$. We also have that if $M \le M'$ then $I_{M}^{\comb} \subseteq I_{M'}^{\comb}$. So it follows that,
\[
\bigcap_{M \in \MM'(\Delta)} I_{M}^{\comb} = \bigcap_{M \in \MM(\Delta)} I_{M}^{\comb}.
\]
And so the proof is complete by Theorem~\ref{thm:general_intersection}. 
\end{proof}

\subsection{Decomposition strategy.}\label{sec:decomposition_strategy}
We can now make precise the decomposition strategy detailed
in \S\ref{sec:intro}. 
By Proposition~\ref{prop:comb_decomposition}, we have a decomposition of the hypergraph variety $V_{\Delta}$ into combinatorial closures of minimally dependent matroids for $\Delta$. To upgrade this to an irreducible decomposition, there are several remaining steps. 

First, one has to find the minimal collection of matroid varieties which cover $V_M^{\comb}$. In particular, one has to show that the matroids appearing in this collection are realizable, since otherwise they are clearly redundant. In general, these matroid varieties might not be irreducible, so one has to either prove their irreducibility  (which we will do in several examples) or find their irreducible decomposition.
By going through these steps, we will obtain such decompositions 
for various families of hypergraphs. It still might be the case that some of the components are redundant. So to obtain the actual minimal decomposition, one has to check which ones are necessary.
We remark that step (ii) of the decomposition strategy, i.e.~writing $V^{\comb}_M$ as a minimal union of matroid varieties, involves showing that the matroid varieties whose union is $V_M^{\comb}$ are nonempty. In other words, we must show that the matroids are realizable. In particular, if $V_M^{\comb}=V_M$ for each minimally dependent matroid $M$ of $\Delta$, then step (ii) can be accomplished by showing that all of the minimally dependent matroids for $\Delta$ are realizable.

\begin{example}[Fano plane] \label{example: fano plane}
Let $d = 3$ and consider the hypergraph
$\Delta = \{
124, 136, 157, 235, 267, 347, 456
\}$. It is straightforward to check that $\Delta$ is the collection of circuits for a matroid $M$, which is called the Fano plane. The matroid $M$ is not realizable over $\CC$ so its realization space $\Gamma_M$ is empty. However its combinatorial closure is non-empty and we can compute the associated prime ideals of $I_\Delta$ in \texttt{Macaulay2} \cite{M2}. We find that $I_\Delta$ has $22$ associated primes which are all matroid varieties of $4$ combinatorial types of point and line configurations; see Figure~\ref{fig:fano_components}. The configurations in the figure, from left to right, are:
\begin{itemize}
    \item A single line with $7$ points. The ideal is a hypergraph ideal $I_{\Delta_0}$ where $\Delta = \binom{[7]}{3}$.
    
    \item A \textit{quadrilateral set}, see \cite[\S8]{richter2011perspectives}, together with a loop. There are $7$ associated primes of $I_\Delta$ with this combinatorial type which are parametrised by the point of the Fano plane to be taken as the loop.
    
    \item A line with $3$ points together with a free point with $4$-labels. There are $7$ associated primes of $I_\Delta$ with this combinatorial type which are parametrised by the lines of the Fano plane.
    
    \item A line with $3$ points together with a free point. Each point on the line has two labels. There are $7$ associated primes of $I_\Delta$ with this combinatorial type which are parametrised by the point of the Fano plane to be taken as the free point.
\end{itemize}

\begin{figure}
    \centering
    \includegraphics[scale = 0.8]{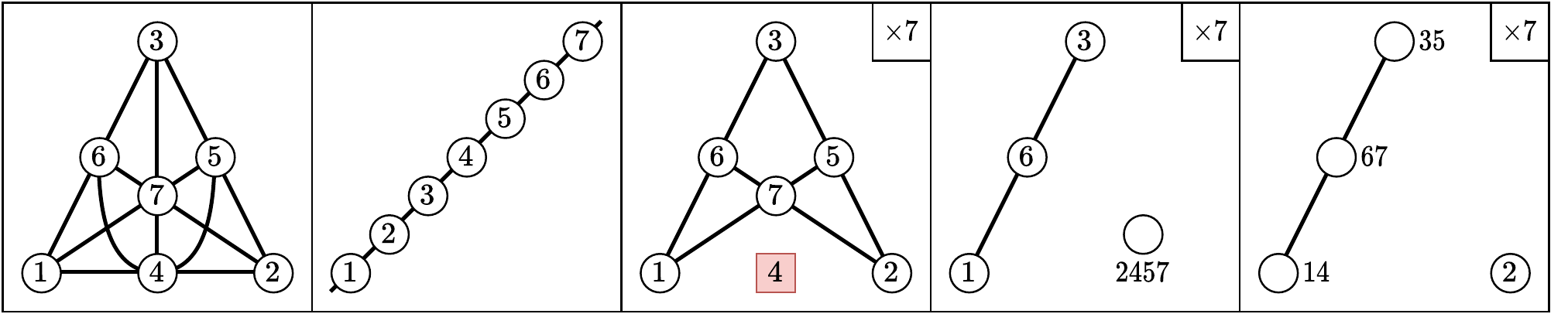}
    \caption{The Fano plane (left) and the configurations of the associated primes of its combinatorial closure.}
    \label{fig:fano_components}
\end{figure}

\end{example}

\begin{example}
\label{ex:andreas}
Let $d = 3$ and $E=[9]\cup\{\bar{1},\bar{2},\bar{3}\}$. Let  $\Delta$ be the hypergraph
$$\Delta=\{123, 456, 789,\bar{1}\bar{2}\bar{3},
147,\bar{1}14,\bar{1}47,\bar{1}17, 258,\bar{2}58,\bar{2}28,\bar{2}25,369,\bar{3}69,\bar{3}39,\bar{3}36\}$$ depicted in Figure~\ref{threelines}. It is shown in \cite[Theorem~4.1]{pfister2019primary} that $I_{\Delta}$ has two prime components such that:
\begin{itemize}
\item The first component is generated by all $3$-minors of a generic $3\times 12$ matrix. Hence, it corresponds to the matroid $M_0$ with dependent sets, $$\mathcal{C}(M_0)=\{\text{all $3$-subsets of $E$}\}.$$

\item The second component is generated by $44$ polynomials of which $16$ are the original generators of $I_{\Delta}$
and the remaining $28$ generators are all homogeneous of degree $12$.
\end{itemize}

Note that there exists a unique minimal matroid $M$ for $\Delta$. The circuits of $M$ are elements of $\Delta$ and all subsets of size $4$ of $E$, which do not contain an element of $\Delta$. (For the proof of a more general statement, see Theorem~\ref{thm:s-t-3}.) Therefore, by Proposition~\ref{prop:comb_decomposition}, we have $\sqrt{I_{\Delta}} = I_{M}^{\comb}$, so the result of \cite{pfister2019primary} describes the decomposition of the combinatorial closure of $M$. This provides a geometric meaning of the second component of the decomposition. 
In particular, it can be obtained by saturating $\sqrt{I_{\Delta}}$ as in Proposition~\ref{prop:saturation}; see Figure~\ref{threelines}.

\medskip

Therefore, we can give a geometric interpretation of the second component described above. More precisely, the $28$ non-determinantal generators 
are analogous to the geometric constraints of \textit{quadrilateral sets}; see \cite[\S8]{richter2011perspectives}. More precisely, given a generic collection of $12$ points arranged in a $3 \times 4$ grid in the projective plane, if we project this configuration onto a line then there are constraints on the distances between the projected points. The $28$ polynomials describe precisely these constraints. 
\end{example}

\section{Irreducible point and line configurations}\label{sec:irred_configuration}

In this section, we will introduce point and line configurations. These are a family of matroids whose ideals appear as prime components of many examples of hypergraph ideals. We will prove Theorem~\ref{thm:lineArrangementBuildUpGeneral} which allows us to inductively build up configurations with irreducible varieties. In particular, using this theorem, we will show that all configurations with at most $6$ lines have irreducible varieties; see Theorem~\ref{thm:k=6_irreducible}.

\begin{notation}
We write $\MC = (\MP, \ML, \mI)$ for a configuration of points and lines in the plane, or simply a configuration. The set $\MP = \{p_1, \dots, p_n \}$ is the collection of \emph{points}, $\ML = \{\ell_1, \dots, \ell_k\}$ is the collection of \emph{lines}  and $\mI \subseteq \MP \times \ML$ is the set of \emph{incidences} of points and lines. If the point $p$ lies on the line $\ell$ then $(p, \ell) \in \mI$ and we say that $p$ is incident to line $\ell$, or $\ell$ passes through $p$. We assume that any two distinct lines meet in at most one point.

For each point $p \in \mP$, we write $\ML_p \subseteq \ML$ for the set of lines which pass through the point $p$. With a slight abuse of notation, we will often identify the line $\ell\in\mL$ with the set of points which lie on the line $\ell$.

\begin{definition}
Let $\MC = (\mP, \mL, \mI)$ be a configuration. For a line $\ell \in \mL$ we define $\MC \backslash \ell = (\mP', \mL \backslash \ell, \mI \cap (\mP' \times (\mL \backslash \ell))$ to be the configuration where $\mP'$ is the collection of points of $\MC$ which do not lie solely on $\ell$.
\end{definition}

\end{notation}
For each configuration $\MC = (\mP,\mL,\mI)$, there is a simple matroid of rank at most $3$ associated to it. The matroid has ground set $\mP$, and is defined such that a $3$-subset $D \subseteq \mP$ is dependent if and only if there exists a line $\ell \in \mL$ with $D \subseteq \ell$. The realization spaces for point and line configurations, are the realization spaces for the matroid; see Definition~\ref{def:prelim_realisation}.
In the projective plane, any two distinct points determine a line and so any pair of points in $\mP$ can be taken to lie on a line. However, if there is not a third element of $\mP$ on the line, then this line does not give rise to any dependent set in the matroid. Therefore, in the remainder of this section, we require that each line in $\mL$ contain three or more points in $\mP$.

We note that configurations $\MC$ are purely combinatorial objects. In particular, they need not be realizable, i.e.~we may have $\Gamma_\MC = \emptyset$. For example, the Fano plane, see Example~\ref{example: fano plane}, is the smallest non-realizable configuration over $\CC$. In the following sections, the purpose of using configurations is to give a parametrization of the irreducible components of certain hypergraph ideals via their realization spaces $\Gamma_\MC$. For these cases, we may assume that the configurations are realizable. However, for families of configurations that we study directly, such as forest-like configurations, we will prove that they are indeed realizable.

\begin{theorem}\label{thm:k=6_irreducible}
For any configuration $\MC$ with at most $6$ lines, $\Gamma_\MC$ is irreducible with respect to the~Zariski~topology.
\end{theorem}

We prove this theorem using the Grassmann-Cayley algebra. 
We review some of its theory following \cite{sturmfels2008algorithms}.

\begin{definition}
Let $V$ be a vector space of dimension $d$ over $\mathbb{C}$ and let $\bigwedge(V)$ be the exterior algebra of $V$. The Grassmann-Cayley algebra is the vector space $\bigwedge(V)$ together with the two operations $\wedge$ and $\vee$.
We denote by $a_1 \vee a_2 \vee \dots \vee a_k \in \bigwedge(V)$ the extensor of length $k$ which is also referred to as the join of $a_1,a_2\dots,a_k$. The meet operation $\wedge$ is a binary operation on 
two extensors $A$ and $B$ of length $j, k$ with $j + k \geq d$ as: 
\[
A \wedge B := \sum_{\sigma} {\rm sign}(\sigma) [a_{\sigma(1)},\dots,a_{\sigma(d-k)},b_1,\dots,b_k]a_{\sigma(d-k+1)}\vee\cdots\vee a_{\sigma(j)},
\]
where the sum is taken over all permutations $\sigma$ of $\{1,2,\dots,j\}$. 
\end{definition}

Any extensor $A = a_1\vee\dots\vee a_k$ has an associated vector space $\overline{A} = {\rm span}(a_1,\dots,a_k)$~with~following~properties.

\begin{lemma}
\label{lemma:GrassmannCayleyAlgebra}
Let $A$ and $B$ be two extensors. Then the following hold: 
\begin{itemize}
\item Any extensor $A$ is uniquely determined from $\overline{A}$ up to a scalar multiple {\rm\cite[Section 3.3]{sturmfels2008algorithms}.}
\item The meet of two extensors is again an extensor {\rm\cite[Theorem~3.3.2~(b)]{sturmfels2008algorithms}.}
\item We have that $A \wedge B \neq 0$ if and only if $\overline{A} + \overline{B} = V$. In this case $\overline{A \wedge B} = \overline{A} \cap \overline{B}$
{\rm\cite[Theorem~3.3.2~(c)]{sturmfels2008algorithms}.}
\end{itemize}
\end{lemma}
\medskip

In order to prove Theorem~\ref{thm:k=6_irreducible}, we will take cases on the possible configurations. We begin by proving the following theorem which allows us to build new irreducible configurations from old.  The cases in the following theorem are depicted in Figure~\ref{fig:example_BuildUpGeneral} and explained further in Example~\ref{example:buildUpGeneral}.

\begin{theorem}\label{thm:lineArrangementBuildUpGeneral}
Let $\MC= (\MP, \ML, \mI)$ be a configuration and $\ell$ a line in $\ML$. Let $S = \{ p \in \ell : |\mL_p| \ge 3 \}$ be set of the points on $\ell$ which lie on at least $2$ other lines. Suppose that $|S| \le 2$. If $\Gamma_{\MC \backslash \ell}$ is irreducible then so is $\Gamma_\MC$.
\end{theorem}

\begin{proof}

Let $n$ be the number of points in $\MC \backslash \ell$ and $\ell_1, \dots, \ell_m$ be the lines of $\MC \backslash \ell$ which intersect $\ell$ in $\mL$ at a unique point. That is, if $\ell_i$ intersects $\ell$ at the point $p$, then there are no other lines through $p$, i.e.~$\mL_p = \{\ell, \ell_i \}$.
For each line $\ell_i$ we fix two distinct points $p_{i, 1}$ and $p_{i, 2}$ in $\mP \backslash \ell$ which lie on $\ell_i$.
For each realization $\gamma \in \Gamma_{\MC\backslash \ell}$ we write $\gamma_1, \dots, \gamma_n \in \CC^3$ for the points of $\gamma$ and $\gamma_{\ell_1}, \dots, \gamma_{\ell_m} \subseteq \CC^3$ for the $2$-dimensional linear subspaces corresponding to the lines. Note that the subspace $\gamma_{\ell_i}$ is the linear span of the points $\gamma_{p_{i,1}}$ and $\gamma_{p_{i,2}}$. We denote by $r$ the number of points of $\MC$ which lie on $\ell$ and no other lines.
We proceed by taking cases on $|S| \in \{ 0, 1, 2\}$. 

\medskip

\noindent \textbf{Case 1.} Assume that $|S| = 0$. We define the space
\[
X = \Gamma_{\MC \backslash \ell} \times  \Hom_{\CC}(\CC^2, \CC^3) \times (\CC^2)^r \times \CC^m,\]
where $\Hom_{\CC}(\CC^2, \CC^3)$ denotes the set of linear maps from $\CC^2$ to $\CC^3$. Here, we think of a linear map $\phi \in \Hom_{\CC}(\CC^2, \CC^3)$ as a $3 \times 2$
matrix with entries in $\CC$. We define the subset $X' \subseteq X$ to be the collection of $(\gamma, \phi, (y_1, \dots, y_r), (\lambda_1, \dots, \lambda_m)) \in X$ such that the following conditions hold:
\begin{itemize}
    \item[(a)] $\rank(\phi) = 2$,
    \item[(b)] for each $i \in [n]$, the point $\gamma_i$ does not lie in $\phi(\CC^2) \subseteq \CC^3$,
    \item[(c)] for each $i \in [r]$, the point $\phi(y_i)$ does not lie in any of the subspaces $\gamma_{\ell_1}, \dots, \gamma_{\ell_m}$,
    \item[(d)] for any  $i, j \in [r]$, if $\phi(y_i) = \phi(y_j)$ then $i = j$. 
\end{itemize}

Since $\Gamma_{\MC\backslash \ell}$, $(\CC^2)^r$ and $\Hom_{\CC}(\CC^2, \CC^3)$ are irreducible, we have that $X$ is a product of irreducible varieties over an algebraically closed field. Hence, $X$ is also irreducible. By construction, each of the conditions for the subset $X' \subseteq X$ above is a rank constraint on certain submatrices. Explicitly we have that:
\begin{itemize}
    \item $\rank(\phi) \neq 2$ if all of the $2$-minors of $\phi$ are zero,
    \item Assuming that $\rank(\phi) = 2$, a point $\gamma_i$ lies inside $\phi(\CC^2)$ if the $3$-minor $[\gamma_i, \phi(1,0), \phi(0,1)]$ is zero,
    \item A point $\phi(y_i)$ lies inside the subspace $\gamma_{\ell_i}$ if the $3$-minor $[\phi(y_i), \gamma_{p_{i,1}}, \gamma_{p_{i,2}}]$ is zero,
    \item If $\rank(\phi) = 2$ then it follows that $\phi$ is injective.
\end{itemize}
Therefore $X' \subseteq X$ is an open subset which implies that $X'$ is also irreducible. 

Now, for each $(\gamma, \phi, (y_1, \dots, y_r), (\lambda_1, \dots, \lambda_m)) \in X'$ and for each $i \in [m]$ we fix a non-zero point $q_i \in \CC^3$ which lies in the intersection of $\gamma_{\ell_i}$ and $\phi(\CC^2)$. Note that the ambient space is $\CC^3$, hence any pair of $2$-dimensional linear subspaces intersect. By construction of $X'$, we have that $\phi(\CC^2)$ and $\gamma_{\ell_i}$ do not coincide. Hence, $\dim(\phi(\CC^2) \cap \gamma_{\ell_i}) = 1$ and $q_i$ is unique up to a non-zero scalar multiple. We give an explicit formula for $q_i$ using the Grassmann-Cayley algebra as follows.
Recall that $\gamma_{\ell_i}$ is the span of $a_1 := \gamma_{p_{i,1}}$, $a_2 := \gamma_{p_{i,2}}$ and $\phi(\CC^2)$ is the span of $b_1 := \phi(1,0)$, $b_2 := \phi(0,1)$. Since the spaces $\gamma_{\ell_i}$ and $\phi(\CC^2)$ do not coincide, they must span the entire space $\CC^3$. Since $\dim(\phi(\CC^2) \cap \gamma_{\ell_i}) = 1$, by Lemma~\ref{lemma:GrassmannCayleyAlgebra}, we have that $q_i := [a_1, b_1, b_2]a_2 - [a_2, b_1, b_2]a_1 \in \gamma_{\ell_i} \cap \phi(\CC^2)$ is a non-zero vector lying in the intersection.
We define the map
\[
\psi : X' \rightarrow (\CC^3)^{n + m + r} :
(\gamma, \phi, (y_1, \dots, y_r), (\lambda_1, \dots, \lambda_m)) \mapsto 
(\gamma, \lambda_1 q_1, \dots, \lambda_m q_m, \phi(y_1), \dots, \phi(y_r)).
\]
By construction, $\phi$ is a linear map and the coordinates of the points $q_i \in \CC^3$ are polynomials in the entries of $\gamma$. Thus $\psi$ is a polynomial map and so it is continuous.

It remains to show that $\Gamma_\MC$ is an open subset of the image of $\psi$. It is easy to see that the image of $\psi$ contains $\Gamma_\MC$ since any realization of $\MC$ can be viewed as a realization of $\Gamma_{\MC \backslash \ell}$ together with some additional points on $\ell$. To show that $\Gamma_\MC$ is open in the image of $\psi$, we note that the image of $\psi$ is contained in the combinatorial closure $V^{\comb}_{\MC}$. So $\Gamma_\MC$ is obtained from the image of $\psi$ by removing the vanishing locus of the bases ideal of $J_\MC$; see Proposition~\ref{prop:saturation}.

\medskip

\noindent \textbf{Case 2.} Assume that $|S| = 1$ and write $S = \{s\}$ where $s$ is the corresponding point of intersection in $\MC$ and $\MC \backslash \ell$. For any $\gamma \in \Gamma_\MC$ we denote by $\gamma_s \in \CC^3$ the coordinates of the point corresponding to $s$. We define the space
\[
X = \{ (\gamma, \phi, (y_1, \dots, y_r), (\lambda_1, \dots, \lambda_m)) : \phi(1,0) = \gamma_s \} \subseteq \Gamma_{\MC \backslash \ell} \times \Hom_\CC(\CC^2, \CC^3) \times (\CC^2)^r \times \CC^m 
\]
where the condition $\phi(1,0) = \gamma_s$ on a matrix $\phi \in \Hom_\CC(\CC^2, \CC^3)$ is equivalent to the condition that the first column of $\phi$ is equal to $\gamma_s$. So, we have that $X \cong \Gamma_\MC \times \CC^3 \times (\CC^2)^r \times \CC^m$ which is irreducible. We define the subset $X' \subseteq X$ identically to Case 1, except that we allow $\gamma_{s} \in \phi(\CC^2)$ in condition (b). Note that $\gamma_s \neq 0$ and so by the same argument, $X' \subseteq X$ is an open subset.
The remainder of the argument from Case 1, including the construction of the points $q_1, \dots, q_m \in \CC^3$ and the map $\psi$, follows identically.

\medskip

\noindent \textbf{Case 3.} Assume that $|S| = 2$ and write $S = \{s_1, s_2\}$ for the corresponding points of intersection lying in both $\MC$ and $\MC \backslash \ell$. For any $\gamma \in \Gamma_\MC$ we denote by $\gamma_{s_1}, \gamma_{s_2} \in \CC^3$ the coordinates of the points corresponding to $s_1$ and $s_2$, respectively. We define the space
\[
X = \{ (\gamma, \phi, (y_1, \dots, y_r), (\lambda_1, \dots, \lambda_m)) : \phi(1,0) = \gamma_{s_1}, \ \phi(0,1) = \gamma_{s_2} \} \subseteq \Gamma_{\MC \backslash \ell} \times \Hom_\CC(\CC^2, \CC^3) \times (\CC^2)^r \times \CC^m 
\]
where the conditions $\phi(1,0) = \gamma_{s_1}$ and $\phi(0,1) = \gamma_{s_2}$ on a matrix $\phi \in \Hom_\CC(\CC^2, \CC^3)$ are equivalent to the conditions that the first column of $\phi$ is equal to $\gamma_{s_1}$ and the second column of $\phi$ is equal to $\gamma_{s_2}$, respectively. Thus $X \cong \Gamma_\MC \times (\CC^2)^r \times \CC^m$, which is irreducible. We define the subset $X' \subseteq X$ identically to Case 1, except that we allow $\gamma_{s_1}, \gamma_{s_2} \in \phi(\CC^2)$ in condition (b). Note that $s_1$ and $s_2$ are distinct points of $\MC$ so $\rank(\phi) = 2$. By the same argument as in Case 1, $X' \subseteq X$ is an open subset.
The remainder of the argument from Case 1, including the construction of the points $q_1, \dots, q_m \in \CC^3$ and the map $\psi$, follows identically. 
\end{proof}

\begin{example}\label{example:buildUpGeneral}
Let $\MC$ be the configuration in Case 1 of Figure~\ref{fig:example_BuildUpGeneral} and $\ell$ the line indicated in the diagram. It is shown in \cite{ollie_fatemeh_harshit} that all configurations with at most $4$ lines have irreducible realization spaces. In particular, $\Gamma_{\MC \backslash \ell}$ is irreducible. For each point $p$ on $\ell$ we have that $|\mL_p| = 2$, thus  $\Gamma_\MC$ is irreducible by Theorem~\ref{thm:lineArrangementBuildUpGeneral}. Note that, in the proof, the points $p_{1,1}, p_{1,2}$, as shown in the diagram, are used to give \textit{coordinates} on the line $\ell_1$, which allows us to give an explicit formula for the intersection point $q_1$ of $\ell$ and $\ell_1$.

In the figure, the configurations labelled Case 2 and Case 3 show examples of configurations, and the line $\ell$, corresponding to their respective cases in the proof of Theorem~\ref{thm:lineArrangementBuildUpGeneral}. The points which lie in the set $S$ are labelled by $s, s_1, s_2$. In particular, each of the realization spaces for these configurations is irreducible.

\begin{figure}
    \centering
    \includegraphics[scale=1.0]{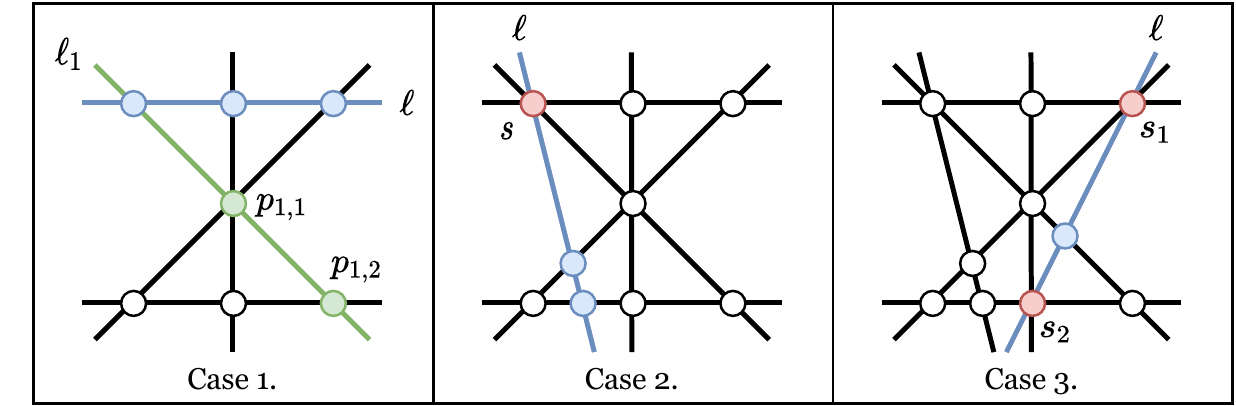}
    \caption{Configurations in Example~\ref{example:buildUpGeneral}. The left and middle figures are obtained from the figure immediately to their right by removing the line $\ell$.}
    \label{fig:example_BuildUpGeneral}
\end{figure}
\end{example}

The proof of Theorem~\ref{thm:k=6_irreducible} follows easily from the above theorem.

\begin{proof}[{\bf Proof of Theorem~\ref{thm:k=6_irreducible}.}]
We show that $\Gamma_\MC$ is irreducible by taking cases on the number of lines in $\MC$. If $\MC$ contains at most $4$ lines then it is irreducible by \cite[Corollary~4]{ollie_fatemeh_harshit}. On the other hand, if $\MC$ contains $5$ or $6$ lines, then we will show that Theorem~\ref{thm:lineArrangementBuildUpGeneral} applies, and so we reduce to a line configuration with fewer lines.

Let $\ell$ be any line of $\MC$ and let $S = \{p \in \ell : |\mL_p| \ge 3 \}$. We show that $|S| \le 2$ by contradiction. 
Assume that $|S| \ge 3$. So there are three distinct points $p_1, p_2, p_3 \in \ell$ such that $|\mL_{p_i}| \ge 3$ for each $i \in [3]$. Note that a pair of lines can intersect in at most one point. So if $\ell' \neq \ell$ is a line passing through $p_1$ then it does not pass through $p_2$ or $p_3$. Since $|\mL_{p_i}| \ge 3$ for each $i$, the total number of distinct lines passing through $p_1, p_2$ and $p_3$ is at least $7$, a contradiction.
By induction, $\Gamma_{\MC \backslash \ell}$ is irreducible. So by Theorem~\ref{thm:lineArrangementBuildUpGeneral} we have that $\Gamma_\MC$ is irreducible.
\end{proof}

\begin{remark}
    The proof of Theorem~\ref{thm:k=6_irreducible} shows that Theorem~\ref{thm:lineArrangementBuildUpGeneral} applies to all line configurations with at most $6$ lines. However, for configurations with $7$ lines, consider the Fano plane depicted in Figure~\ref{fig:fano_components}. This configuration contains $7$ points each of which belongs to $3$ lines. Therefore,  Theorem~\ref{thm:lineArrangementBuildUpGeneral} does not apply. 
\end{remark}

\section{Combinatorial closures of configurations}\label{sec:forest_like_configurations}

In this section, we will focus on point and line configurations whose underlying graph is a forest. 
 As in the previous section, if $\MC = (\mP,\mL,\mI)$ is a point and line configuration, then each line in $\mL$ contains at least three points in $\mP$, unless stated otherwise.

We will give a complete description of their combinatorial closures by describing their components. Our main tool for this section is the \textit{perturbation argument} which gives a way to determine when a particular point in the combinatorial closure belongs to a specific realization space. Let us begin by making precise our notation.

\medskip
Let $\MC = (\mP,\mL,\mI)$ be a point and line configuration with the ordered set of points $\mP= \{p_1,\ldots,p_n\}$. Let  $G_{\MC}$ be a graph with vertex set $\MP$ and edges $\{p_i, p_j\}$ such that
\begin{itemize}
    \item both points lie on the same line $\ell\in \mL$ (which is necessarily unique) and
    \item there is no point $p_k$ lying on $\ell$ with $i<k<j$ (or $j<k<i$). 
\end{itemize}

Intuitively, $G_{\MC}$ is obtained from $\MC$ by making points into vertices and lines into paths. Moreover, every edge of $G_{\MC}$ is in exactly one of these paths. Each pair of paths intersects in at most one vertex, and this occurs if and only if the corresponding lines intersect in a point of $\mP$.

\begin{definition}\label{def:forestlike}
A configuration $\MC$ is \emph{forest-like} if the corresponding graph $G_{\MC}$ is a forest, i.e.~it has no cycle.
\end{definition}

Note that the graph $G_{\MC}$ depends on the choice of ordering of the points in $\mP$. However, the condition for $G_{\MC}$ to be a forest is independent of this choice. Essentially, swapping the order of two points on a line corresponds to simply ``flipping" an edge of $G_{\MC}$. We now make this more precise.

\begin{lemma}
\label{lem:ind-of-order}
If $\MC$ is forest-like, then any ordering of its points $\MP$ makes $G_{\MC}$ a forest. 
\end{lemma}

\begin{proof}
Let $\mP= \{p_1,\ldots,p_n\}$ be an ordering of $\mP$ such that $G_{\MC}$ is a forest. Inductively, it suffices to show that the order $\{p_1,\ldots,p_{i-1},p_{i+1},p_i,p_{i+2},\ldots,p_n\}$ gives rise to a forest.

If $p_i$ and $p_{i+1}$ do not lie on the same line, then there is no change to $G_{\MC}$. So it suffices to consider the case where $p_i$ and $p_{i+1}$ lie on a line $\ell\in\mL$. Then $\{p_i,p_{i+1}\}$ is an edge in $G_{\MC}$. Let us call these vertices of the graph $v$ and $w$. Let $A(v)$ be the set of vertices of $G_{\MC}$, adjacent to $v$, whose corresponding points in $\MC$ lie on some line $\ell' \neq \ell$, and define $A(w)$ similarly. Then swapping the order of $i$ and $i+1$ corresponds to replacing the edges in $\{\{x,v\}:x\in A(v)\}\cup\{\{x,w\}:x\in A(w)\}$ with the edges in $\{\{x,v\}:x\in A(w)\}\cup\{\{x,w\}:x\in A(v)\}$. One can see that this resulting graph is still a forest if $G_{\MC}$ was a forest.
\end{proof}

Because of Lemma \ref{lem:ind-of-order},
Definition~\ref{def:forestlike} is well-defined for point and line configurations with unordered sets of vertices. So, we say that $\MC$ is forest-like if there exists an ordering of its vertices such that the resulting configuration is forest-like.

The next result will also be helpful.

\begin{lemma}
\label{lem:one-line}
If $\MC=(\mP,\mL,\mI)$ is a forest-like configuration with $\mL\neq\emptyset$, then there exists a line $\ell$ which intersects $\MC \setminus \ell$ in at most one point of $\mP$.
\end{lemma}

\begin{proof}
Because two lines intersect in at most one point, it will suffice to show that there exists a line $\ell$ which intersects $\MC \setminus \ell$ (at a point of $\mP$) in at most one line. Suppose otherwise for a contradiction. Then every line intersects with at least two other lines. Let $\ell_1,\ell_2,\ldots$ be a sequence of lines in $\MC$ such that $\ell_i$ intersects both $\ell_{i-1}$ and $\ell_{i+1}$. Since there are only finitely many lines, eventually there will be a repeated element $\ell_k$ of the sequence.

Since these lines correspond to paths in $G_{\MC}$, there is a cycle in $G_{\MC}$ whose set of edges includes at least one edge from each path corresponding to the lines in the sequence $\ell_k,\ell_{k+1},\ldots,\ell_k$. This cycle contradicts the assumption that $G_{\MC}$ is a forest.
\end{proof}

It is perhaps worth noting that forest-likeness is a rather restrictive condition for a matroid of rank $3$. However, forest-like configurations are not regular (realizable over every field) matroids. For example, no line with four or more points is a regular matroid. Each forest-like configuration is realizable over all infinite fields. To show this, it will be useful to recall the matroid-theoretic notion of \emph{freely adding} an element to a flat of a matroid. Intuitively, this operation takes a new element of the ground set and adds it to a flat as ``freely" as possible, that is, we keep as many sets independent as possible; see \cite{Crapo65} and \cite[Section 7.2]{Oxley}. We will also use this notion in Section \ref{sec:forestlike_realizable}.

\begin{definition}
Let $F_0$ be a flat of a matroid $M$ with rank function $r_M$. We say that $M'$ is the single-element extension obtained by \emph{freely adding} $e$ to $F_0$ if the flats of $M'$ fall into the following disjoint classes:
\begin{itemize}
    \item flats $F$ of $M$ that do not contain $F_0$,
    \item sets $F\cup e$ where $F$ is a flat of $M$ that contains $F_0$, and
    \item sets $F\cup e$ where $F$ is a flat of $M$ that does not contain $F_0$, and there is no flat $F'$ of $M$ of rank $r_M(F)+1$ such that $F\subseteq F'$ and $F_0\subseteq F'$.
\end{itemize}
\end{definition}
The following results are fairly well-known. In particular, Lemma~\ref{lem:extension} follows from a result of Piff~and~Welsh \cite{piff1970vector}; see also \cite[Proposition~11.2.16]{Oxley}. We provide a more direct proof to keep the~paper~self-contained.

\begin{lemma}
\label{lem:extension}
Let $\mathbb{F}$ be an infinite field, and let $M'$ be a matroid obtained by freely adding an element to a flat $F$ of a matroid $M$. Then $M'$ is $\mathbb{F}$-realizable if and only if $M$ is $\mathbb{F}$-realizable.
\end{lemma}

\begin{proof}
Realizability over a field is closed under deletion. Therefore, if $M'$ is $\mathbb{F}$-realizable, then so is $M$.

For the converse, let $M$ be $\mathbb{F}$-realizable, and let $r$ be the rank function of $M$. A flat $F$ of $M$ corresponds to a linear subspace of $\mathbb{F}^{r(M)}$ of dimension $r(F)$. To construct a representation of $M'$ we must find a point $p$ in the subspace corresponding to $F$ that is not contained in any subspace corresponding to a flat $F'$ of $M$ that does not contain $F$. For any such flat $F'$, the rank of $F\cap F'$ is strictly less than $r(F)$. Therefore, since $\mathbb{F}$ is an infinite field, there are infinitely many such points $p$.
\end{proof}

\begin{lemma}
\label{lem:coloop}
Let $M'$ be obtained by adding a coloop to an $\mathbb{F}$-realizable matroid $M$. Then $M'$ is $\mathbb{F}$-realizable.
\end{lemma}

\begin{proof}
Let $A$ be a matrix whose columns form a realization of $M$. We construct a matrix realizing $M'$ by adding a row and column to $A$ such that all entries of this row and column are $0$ except for the entry contained in both the row and column. 
\end{proof}

\begin{proposition}
\label{prop:induct}
Let $\mathcal{C}=(\mathcal{P},\mathcal{L},\mathcal{I})$ be a forest-like point and line configuration, and let $M$ be the simple matroid of rank at most $3$ associated with it. Then $M$ is realizable over all infinite fields.
\end{proposition}

\begin{proof}
We proceed by induction on the number of lines in $\mathcal{L}$. If $|\mathcal{L}|=0$, then $M$ is the uniform matroid $U_{r,|\mP|}$, where $r=\min\{3,|\mP|\}$. For all positive integers $n$, it is clear that $U_{n,n}$ is realizable over all fields. If $n>3$, then $U_{3,n}$ is obtained from $U_{3,3}$ by repeatedly freely adding elements to the rank-$3$ flat of the matroid. Therefore, by Lemma \ref{lem:extension}, $U_{r,|\mP|}$ is realizable over all infinite fields for all values of $n$.

Now, suppose that we have shown that the result holds for all forest-like configurations with at most $k-1$ lines and we wish to prove the result for $\mathcal{C}=(\mathcal{P},\mathcal{L},\mathcal{I})$, where $|\mathcal{L}|=k$. By Lemma \ref{lem:one-line}, there is a line $\ell$ in $\mL$ that intersects with at most one of the points in $\MC\setminus\ell$. By the induction hypothesis, the matroid associated with $\MC\setminus\ell$ is realizable over all infinite fields.

First, we consider the case where $\ell$ intersects with a point $v$ in $\MC\setminus\ell$. Let $w$ be an additional point in $\ell$, and let $T=\ell\setminus\{v,w\}$. The matroid $M\backslash T$ is obtained from $\MC\setminus\ell$ either by adding $w$ as a coloop or by freely adding $w$ to the unique rank-$3$ flat (depending on the rank of $\MC\setminus\ell$). Then $M$ is obtained from $M\backslash T$ by freely adding the points in $T$ to the rank-$2$ flat defined by $v$ and $w$.

Now, we consider the case where $\ell$ and $\MC\setminus\ell$ have no points in common. Let $w_1$ and $w_2$ be points in $\ell$, and let $T=\ell\setminus\{w_1,w_2\}$. The matroid $M\backslash(T\cup\{w_2\})$ is obtained from $\MC\setminus\ell$ either by adding $w_1$ as a coloop or by freely adding $w_1$ to the unique rank-$3$ flat (depending on the rank of $\MC\setminus\ell$). Then $M\backslash T$ is obtained from $M\backslash(T\cup\{w_2\})$ by freely adding $w_2$ to the unique rank-$3$ flat. Then $M$ is obtained from $M\backslash T$ by freely adding the points in $T$ to the rank-$2$ flat defined by $w_1$ and $w_2$.

In either case, Lemmas \ref{lem:extension} and \ref{lem:coloop} imply that $M$ is realizable over all infinte fields.
\end{proof}

One can see that the last line $\ell$ to be added to a forest-like configuration in the inductive process described in the proof of Proposition \ref{prop:induct} has the property that $|\{p \in \ell : |\mL_p| \ge 3 \}| \le 1$. In particular, the only point $p\in\ell$ that might have the property that $|\mL_p| \ge 3 $ is the one point that was on at least one of the other $|\mathcal{L}|-1$ lines in $\mathcal{L}$.
 
 So, we have the following straightforward corollary of Theorem~\ref{thm:lineArrangementBuildUpGeneral}.

\begin{corollary}
\label{thm:forestlike_irred}
The realization space of a forest-like configuration is irreducible with~respect~to~Zariski~topology.
\end{corollary}

The following proof introduces the \emph{perturbation procedure} which we will use throughout this section.

\begin{lemma}[Perturbation procedure]\label{lem:perturbation_procedure}
Let $\MC$ be a forest-like configuration with realization space $\Gamma_{\MC} \subseteq \CC^d$. Assume that $|\mL_p|\leq2$ for every $p\in \mP$. For every $\epsilon > 0$ and for any $A \in V^{\comb}_{\MC} \backslash \Gamma_{\MC}$ there exists $A' \in \Gamma_{\MC}$ such that $||A - A'|| < \epsilon$, where $|| \cdot ||$ is the Euclidean norm on $\CC^{d \times n}$.

\end{lemma}

\begin{proof}
We think of $A$ as a realization of a configuration $\MC_A$. As $A \in V^{\comb}_{\MC} \backslash \Gamma_{\MC}$ and the rank of the corresponding matroid is at most $3$, the dependencies satisfied by the configuration $\MC_A$ that are not in $\MC$ are the following:
\begin{itemize}
    \item A point of $\MC$ may be a loop in $\MC_A$,
    \item Two distinct points of $\MC$ may coincide in $\MC_A$,
    \item Two distinct lines in $\MC$ may coincide in $\MC_A$,
    \item A triple of non-collinear points in $\MC$ may lie on a common line in $\MC_A$.
\end{itemize}

We now construct the realization $A'$ by induction on the number of lines. For the base case, assume that the configuration $\MC_A$ has no lines. Let us form $A'$ by going through the points in order $p_1, \dots, p_r$
and perturbing the corresponding vector in $A$ by at most $\epsilon / r$. For each $i \in [r]$, the vector corresponding to point $p_i$ is perturbed such that it forms no dependencies, listed above, with subsets of points from $\{p_j : 1 \le j \le i \}$.

For the induction step, let us assume that for any forest-like configuration $\MC' = (\MP', \ML', \mI')$ with at most $n-1$ lines, we have that for all $\epsilon > 0$ and any point $x \in V^{\comb}_{\MC'}$ there exists $x' \in \Gamma_{\MC'}$ such that $||x - x'|| < \epsilon$. Since the graph $G_\MC$ is a forest, we may take $\ell \in \ML$ to be a line which intersects $\MC \backslash \ell$ in at most one point. By the inductive hypothesis, we can perturb vectors in $A$ corresponding to the configuration $\MC \backslash \ell$ so that no additional dependencies are satisfied. It remains to show that the points on the line $\ell$ may be perturbed to remove all additional dependencies. We proceed by applying some or all of the following steps in order. In particular, after each step we ensure that no new dependencies are introduced which would be removed in a previous step. 

\begin{enumerate}

    \item[(a)] Suppose that the line $\ell$ intersects a point $p$ in $\MC_A$ which does not belong to $\ell$ in $\MC$. We may perturb the line $\ell$ in $A$ to produce $A'$ so that $\ell$ does not pass through the point $p$. In particular, if $\ell$ contains $k$ points which are not $p$, then we perturb each point by at most $\epsilon / k$.

    \item[(b)] Suppose that the line $\ell$ coincides with another line $\ell'$ in $\MC_A$. Since $\ell$ intersects $\MC \backslash \ell$ in at most one point, we may rotate the line $\ell$ in $A$ (i.e.~perturb each point along the line) by small amount to obtain $A'$ so that $\ell$ does not coincide with any other line. 
    In particular, if $\ell$ contains $k$ points which do not lie on the intersection of $\ell$ and $\ell'$ in $\MC$, then we perturb each point by at most $\epsilon / k$.
    
    \item[(c)] Suppose that two points in $\ell$ coincide or some of the points are loops. We may perturb these points away from each other so that they remain on the line $\ell$. Similarly if any of the points are loops, then we can perturb these points away from zero by the same method.  In particular, if a point $p$ in $\mathcal C_A$ is a loop then we recall that $p$ is incident to at most two lines. If it is incident to exactly two lines $\ell, \ell'$ in $\mathcal C$, then we may perturb $p$ along $\ell \cap \ell'$. Otherwise if $p$ is incident only to $\ell$ in $\mathcal C$ then we may perturb it to some non-zero point on $\ell$.

    \item[(d)] Suppose that for some point $p_i$ in $\ell$, there are two distinct points $p_j, p_k$ in $\MC \backslash \ell$ such that $p_i, p_j$ and $p_k$ lie on a line $\ell'$ in $\MC_A$ but do not lie on a line in $\MC$. By the step~(a) of the procedure, we have that $\ell$ and $\ell'$ are distinct lines. For this case, it is useful to consider all lines of $\MC_A$, including those which contain only two points. The points of intersection between $\ell$ and all other distinct lines of $\MC_A$ is a finite set. Since we work over the infinite field $\CC$ and the line is homeomorphic to $\CC^1$ with respect to Euclidean topology, therefore we may perturb $p_i$ along $\ell$ so that it lies only on $\ell$ and no other line.
\end{enumerate}
As a result of the above procedure, we have constructed a realization $A'$ of 
$\MC$ with $||A - A'|| < \epsilon$. 
\end{proof}

\begin{figure}
    \centering
    \includegraphics[scale = 0.75]{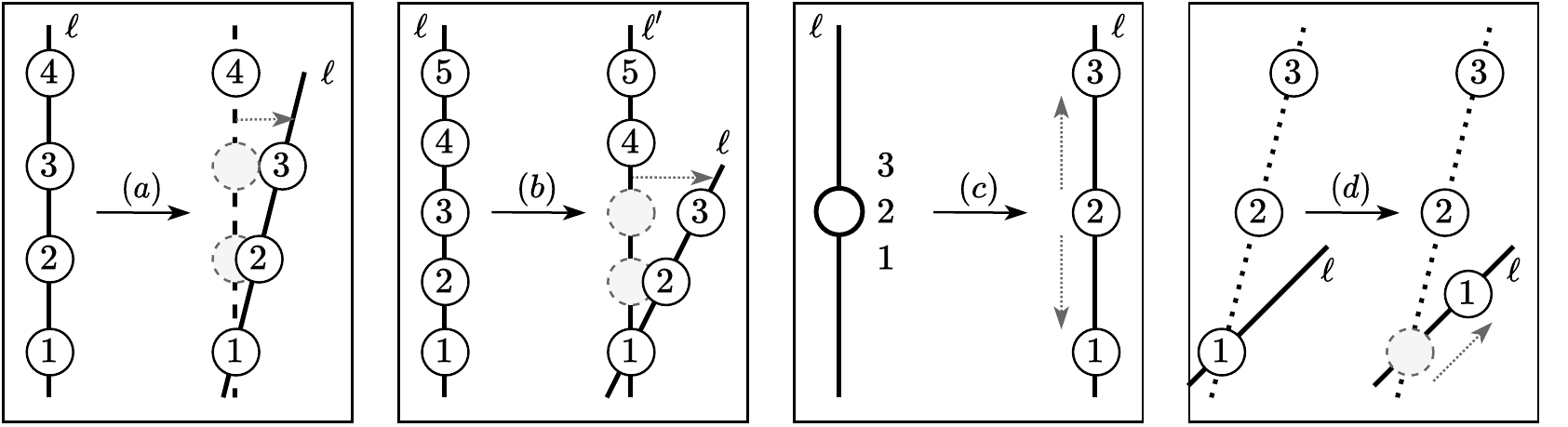}
    \caption{Depiction of the steps in the perturbation procedure: (a) line $\ell$ is perturbed away from point $4$, (b) points $2,3 \in \ell'$ are perturbed off $\ell$, (c) coincident points $1,2,3$ are perturbed away from each other along $\ell$ and (d) point $1$ is perturbed along $\ell$ away from the intersection.}
    \label{fig:perturbation_arguments}
\end{figure}

\begin{remark}
By perturbation procedure, the corresponding 
configurations are realizable~over~any~subfield~of~$\CC$.

\end{remark}

Let $\MC$ be the configuration with two lines $123$ and $345$. Consider the following points in $V^{\comb}_{\MC}$:
\begin{eqnarray}\label{eq:A_Example}
A = \begin{bmatrix}
1 & 1 & 0 & 0 & 0 \\
0 & 1 & 0 & 1 & 0 \\
0 & 0 & 0 & 1 & 1 
\end{bmatrix}
\quad \text{and} \quad
A' = 
\begin{bmatrix}
1 & 1 & 0 & 0 & 0 \\
0 & 1 & \epsilon & 1 & 0 \\
0 & 0 & 0 & 1 & 1 
\end{bmatrix}.
\end{eqnarray}
Note that $A \notin \Gamma_\MC$ is not a realization of $\MC$ because $3$ is a loop in $\MC_A$. See Figure~\ref{fig: perturbation example}. Following the perturbation argument, we perturb $A$ to $A'$ which corresponds to moving point $3$ so that it is non-zero and lies on the intersection of the planes spanned by $1$, $2$ and $4$, $5$.

\begin{figure}[h]
    \centering
    \includegraphics[scale = 0.7]{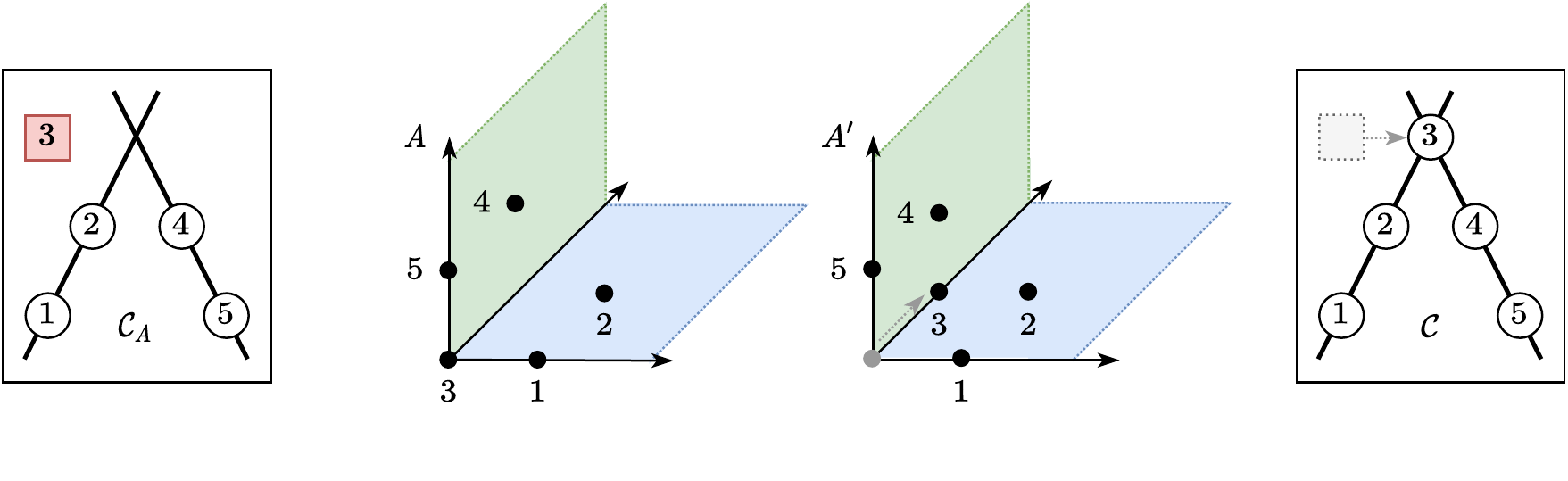}
    \caption{On the right a configuration $\MC$ and on the left the configuration $\MC_A$ for the matrix $A$ in \eqref{eq:A_Example}, together with a depiction of the perturbation of point $3$ taking $A \in V^{\comb}_{\MC} \backslash \Gamma_\MC$ to $A' \in \Gamma_\MC$. The loop $3$ in $\MC_A$ is shown in a square and the shaded planes in $\CC^3$ correspond to the lines in $\MC_A$ and $\MC$.}
    \label{fig: perturbation example}
\end{figure}

\vspace{-3mm}

\begin{theorem}\label{thm:forestlike_closure}
Let  $\MC=(\mP,\mL, \mI)$  be a forest-like configuration with $|\mL_p|\leq2$ for every $p\in \mP$. Then the combinatorial closure $V^{\comb}_\MC$ coincides with the matroid variety $V_\MC$ and in particular is irreducible.
\end{theorem}
\begin{proof}
Recall that $V^{\comb}_\MC = V_{C(\MC)}$ where $C(\MC)$ denotes the circuits of the matroid corresponding to $\MC$. So it remains to show that $V_{C(\MC)} = V_{\MC}$. It is clear that $V_{\MC} \subseteq V_{C(\MC)}$. For the opposite inclusion, by the perturbation procedure each $A \in V_{C(\MC)}$ is a limit point of the configuration space $\Gamma_{\MC}$ with respect to the Euclidean topology. This implies $V_{C(\MC)} \subseteq V_{\MC}$, since the Zariski topology is coarser than the Euclidean topology.
\end{proof}
\vspace{-2mm}

\begin{definition}\label{def: setting points to loops}
Let $\MC$ be a point and line configuration and $p$ a point. Let $\MD$ be the dependent sets of the matroid associated to $\MC$ on ground set $\MP$. By \textit{setting $P \subseteq \MP$ to be loops}, we mean the matroid $M$ whose dependent sets are given by $\MD \cup \{D \subseteq \MP \colon P \cap D \neq \emptyset\}$. In other words the circuits of $M$ are: the circuits of the matroid associated to $\MC$ which do not meet $P$; and $P$ thought of as a set of loops.
\end{definition}

Before proving an analogous result for a general forest-like point and line configuration, we consider the following simple configuration which has at most one point that lies on more than two lines.

\begin{lemma}\label{lem:star_irred_comps}
Let $\MC = (\MP, \ML, \mI)$ be a forest-like configuration which has a unique point $p := p_i$ such that $|\ML_{p}| > 2$ and for all other points $p_j$ we have $|\ML_{p_j}| \le 2$.
Then the combinatorial closure $V^{\comb}_\MC$ has exactly two irreducible components: one is the central component and the other arises by setting the intersection point $p$ to be a loop.
\end{lemma}
\begin{proof}
For each element $A \in V^{\comb}_{\MC}$, we write $A_p \in \CC^3$ for the column vector corresponding to the intersection point $p$. We have that $V_{\MC}$ is the central component of $V^{\comb}_{\MC}$ and we write $V_0 = \{A \in V^{\comb}_{\MC} : A_p = \underline 0\}$ for the collection of elements of the combinatorial closure where the intersection point is zero. Let $\MC_0$ be the configuration obtained from $\MC$ by setting the intersection point to be a loop. Clearly, we have $V_{\MC_0} \cup V_{\MC} \subseteq V_0 \cup V_{\MC} \subseteq V^{\comb}_{\MC}$. Since $\Gamma_{\MC_0}$ and $\Gamma_{\MC}$ are irreducible by Corollary~\ref{thm:forestlike_irred}, the lemma follows from showing that $V_{\MC_0} \cup V_C = V^{\comb}_{\MC}$ by proving the opposite inclusion.

Take any element $A \in V^{\comb}_{\MC}$ and fix $\epsilon > 0$. We will show that there exists $A' \in \Gamma_{\MC_0} \cup \Gamma_{\MC}$ such that $||A - A' || < \epsilon$ by applying the perturbation procedure. Let us take cases on whether $A_p$ is zero.

\medskip
\noindent\textbf{Case 1.} Assume that $A_p = \underline 0$.
Since $\MC_0$ is a configuration that contains only points $p_j$ with $|\ML_{p_j}| \le 2$, we may apply the perturbation procedure to $\MC_0 \backslash p$.
As a result we have $A' \in \Gamma_{\MC_0}$.

\medskip
\noindent\textbf{Case 2.} Assume that $A_p \neq \underline 0$.
We may now apply  steps (a) to (d) in the proof of the  perturbation procedure to construct a point $A' \in \Gamma_{\MC}$ with $||A - A'|| < \epsilon$. In the procedure,  by assumption we have that  $A_p \neq \underline 0$. This assumption guarantees that if $p, p_1, p_2$ are collinear points in $\MC$ then the corresponding points are collinear in $\MC_A$, and all lines arising in this way pass through the common point $p$.
\end{proof}

We are now ready to state our main result in this section for a general forest-like point and line configuration. 
\begin{theorem}\label{thm:forestlike_comp}
Let $\MC=(\mP,\mL, \mI)$ be a forest-like 
configuration. Let $S = \{ p \in \mP : |\mL_p| \ge 3 \}$ be the collection of points contained in at least $3$ lines. Then the combinatorial closure $V^{\comb}_\MC$ has at most $2^{|S|}$ irredundant irreducible components. Moreover, these components can be obtained 
from $\MC$ by setting a subset of $S$ to be loops.
\end{theorem}

\begin{proof}
For each element $A \in V^{\comb}_{\MC}$ and point $p \in \MP$, we write $A_p \in \CC^3$ for the column vector corresponding to the point $p$. For each subset $J \subseteq S$, we write $\MC_J$ for the configuration obtained from $\MC$ by setting the points in $J$ to be loops. We will show that $\bigcup_{J \subseteq S} V_{\MC_J} = V^{\comb}_{\MC}$. Note that by Corollary~\ref{thm:forestlike_irred}, each variety $V_{\MC_J}$ is irreducible. By construction, it is immediately clear that $V_{\MC_J} \subseteq V^{\comb}_{\MC}$ for each $J$. So, to prove the theorem, 
it remains to show the opposite inclusion, i.e.~$V^{\comb}_{\MC} \subseteq \bigcup_{J \subseteq S} V_{\MC_J}$.
Fix $A \in V^{\comb}_{\MC}$ and $\epsilon > 0$. We will construct $A' \in \Gamma_{\MC_J}$ for some $J\subseteq S$ such that $||A - A'|| < \epsilon$. 

We proceed by induction on $s := |S|$.
If $s = 1$, then the result follows by Lemma~\ref{lem:star_irred_comps}. So let us assume that $s > 1$. 
Since $G_{\MC}$ is a forest, we can find a point $p \in S$ such that the path (if it exists) between any other pair of points in $S$ does not pass through $p$. Let $P \subseteq \MP$ be the collection of points $q \in \MP$ lying in the same connected component of $G_\MC$ as $p$, such that for all $p' \in S \backslash \{ p\}$, if there is a path from $q$ to $p'$, then the path passes through $p$. By convention, we assume $p \in P$.

Consider the configuration $\MC' = (\mP', \mL', \mI')$ obtained from $\MC$ by removing points $P$. We have $S' := \{p \in \mP' : |\mL'| \ge 3 \} = S \backslash \{p\}$. Clearly, $\MC'$ is a forest-like configuration and $|S'| = s - 1$. So, by induction, we can perturb the points in $\mP'$ so that they lie in a realization space $\Gamma_{\MC'_{J'}}$ for some $J' \subseteq S'$. More precisely, there exists a subset $J' \subseteq S\backslash \{p \}$ and a realization $A_{\mP'}' \in \Gamma_{\MC'_{J'}}$ such that $||A_{\mP'} - A_{\mP'}'|| < \epsilon / 2$, where $A_{\mP'}$ is the set of vectors obtained from $A$ by removing the points $P$. We proceed by taking cases on whether $A_p$ is zero.

\noindent\textbf{Case 1.} Assume that $A_p = \underline 0$. Let $J = J' \cup \{p\}$, we construct $A' \in \Gamma_{\MC_J}$ from $A'_{\mP'}$ by taking the vectors in $A$ for the points in $P \backslash \{p\}$ and applying the perturbation procedure. By the perturbation procedure we have ensured that $A' \in \Gamma_{\MC_J}$ and by perturbing each point $p \in P$ by a distance of at most $\epsilon / 2|P|$, we have that $||A - A'|| \le ||A_{\mP'} - A'_{\mP'}|| + (|P| - 1) (\epsilon / 2|P|) < \epsilon$.

\noindent\textbf{Case 2.} Assume that $A_p \neq \underline 0$. 
Let $J = J'$ and construct $A' \in \Gamma_{\MC_J}$ from $A'_{\mP'}$ by taking the vectors in $A$ for points in $P$ and applying  steps (a) to (d) in the proof of the  perturbation procedure. In the procedure,  by assumption we have that  $A_p \neq \underline 0$. This assumption guarantees that all lines of $\MC$ passing through $p$ are contained in a line of $\MC_A$ passing through $p$.
As a result we obtain a realization $A' \in \Gamma_{\MC_J}$ and by perturbing each point $p \in P$ by a distance of at most $\epsilon / 2|P|$, we have that $||A - A'|| \le ||A_{\mP'} - A'_{\mP'}|| + |P| (\epsilon / 2|P|) < \epsilon$.
\end{proof}

\begin{remark}
It is not hard (but a bit tedious) to classify all irreducible components of $V^{\comb}_\MC$ for forest-like configurations $\MC$. This classification follows from the same proof of Theorem~\ref{thm:forestlike_comp} by taking into account that components may become redundant; see Example~\ref{example:E_irred_comps}.
\end{remark}

\begin{example}\label{example:E_irred_comps}
Consider the    
configuration $\MC$ in Figure~\ref{fig:E_irred_comps}. Note that the points $1, 2$ and $3$ have degree $3$ in the graph $G_{\MC}$. So by Theorem~\ref{thm:forestlike_comp}, $V^{\comb}_{\MC}$ has at most $8$ irredundant irreducible components which arise from setting the intersection points: $1, 2, $ and $3$ to be loops in some combination. When points are set to loops, we remove the lines which contain at most two non-loop points, as they do not contribute any new dependencies. The variety of $\MC_{\{1\}}$ appears as an irredundant component of $V^{\comb}_{\MC}$. However $\MC_{\{1,2\}}$ gives rise to a redundant component since its variety is contained the variety of $\MC_{\{1\}}$.
The irreducible decomposition of $V^{\comb}_{\MC}$ has $4$ irredundant components which are: the central component $V_{\MC}$, and the non-central components $V_{\MC_{\{1\}}}, V_{\MC_{\{2\}}}, V_{\MC_{\{3\}}}$.

\begin{figure}
    \centering
    \includegraphics[scale=0.8]{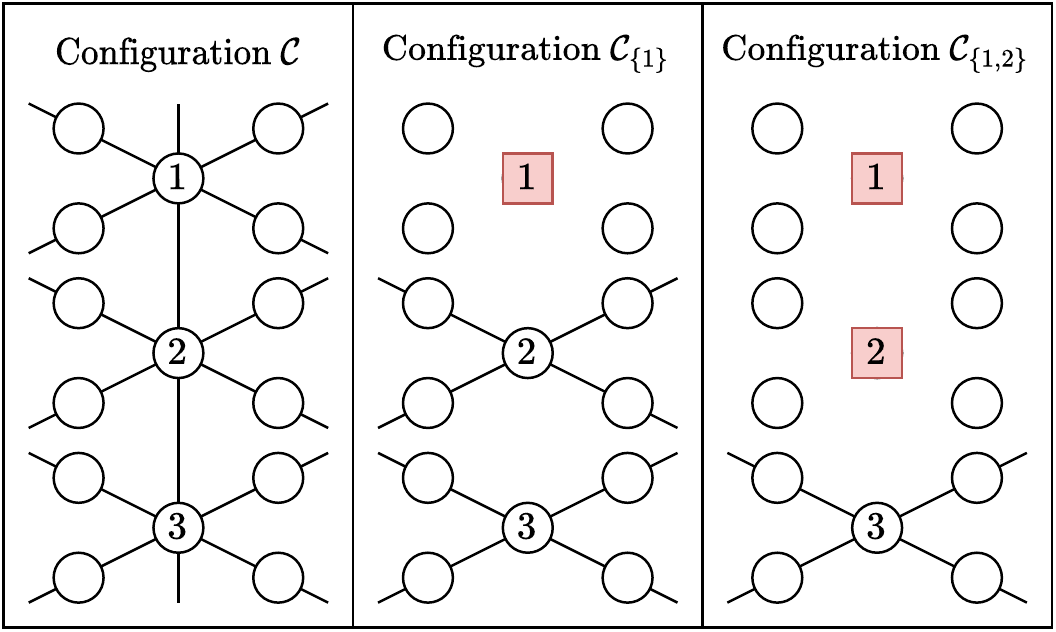}
    \caption{Depiction of $\MC$, $\MC_{\{1\}}$ and $\MC_{\{1,2\}}$ in Example~\ref{example:E_irred_comps}. Shaded squares represent loops of the configuration. 
      }
    \label{fig:E_irred_comps}
\end{figure}
\end{example}

\section{Consecutive forest hypergraphs \texorpdfstring{$\Delta_G$}{DG}}\label{sec:adjtree}

We fix a $d \times n$ matrix $X = (x_{i,j})$  of indeterminates and the polynomial ring $R = \CC[X]$. Throughout this section, we fix a forest $G$ (acyclic graph) with vertices labeled $1, \dots, n$.

\begin{definition}
\label{def:consecutive_forest_hyp}
We define the {\em consecutive forest hypergraph of $G$} as
    \begin{eqnarray*}\label{eq:tree_hypergraph}
    \Delta_{G}=\min
    \left(
    P_2(G)
    \cup \binom{[n]}{4}
    \right)
    \quad\text{where}\quad
    P_2(G)=\{P \subseteq V(G):\ P \text{ is a $2$-path in $G$}\}.
    \end{eqnarray*}
We recall that a $k$-path $P$ in $G$ is a subgraph of $G$ whose vertex set is a non-repeating sequence $p_1, \dots, p_{k+1} \subseteq V(G)$ and whose edge set is $\{\{p_i,p_{i+1}\}:1 \le i \le k\}$. We identify $P$ with its vertex set, and simply write $C \subseteq P$ instead of $C \subseteq V(P)$. If $Q$ is another path, then $P \cup Q$ is the set of all vertices of $G$ lying in $P$ and $Q$.
\end{definition}

\begin{remark}
The hypergraph ideal $I_{\Delta_{G}}$ of the $n$-path $G$ is exactly the  ideal $I_{3n}(3)$ of adjacent minors studied in \cite{HSS}. We note that our construction generalises this particular family by allowing $G$ to be an arbitrary graph. In \cite{HSS}, the ideals $I_{mn}(m)$ are studied for general $m \ge 3$ which can be thought of as the consecutive forest hypergraph of the $n$-path except with higher order minors.
These ideals can be studied using positroid varieties as described in \cite{knutson2013positroid}. It is not too difficult to show that each of the prime components of $I_{mn}(m)$ is a positroid  variety. For instance, one can use the characterisation of positroids via their excluded minors \cite[Theorem~16]{Oh2009Combinatorics} or the decomposition of the Grassmannian into positroid varieties \cite[Section~5.2]{knutson2013positroid}. In particular, we note that the defining ideals of positroid varieties are generated by minors \cite[Theorem~5.15]{knutson2013positroid}. However, this is no longer true for the  consecutive forest hypergraph varieties. More precisely, the varieties arising in the irreducible decomposition of $V_{\Delta_G}$ are, in general, not positroid varieties; for instance see Example~\ref{example: non-positroid component E}.
\end{remark}

\subsection{Minimal matroids.} 
\label{sec:prime-collections-for-m=3}

In this section, we will define so-called prime collections which are collections of subsets of vertices of $G$. To each prime collection $\MS$ we will associate a unique matroid $M_{\MS}$.
We will then see that such matroids are realizable over the real numbers and have irreducible realization spaces. Moreover, we prove that such matroids appear as minimal matroids for consecutive forest hypergraphs.

\begin{definition}
\label{def:prime-singletons}
Let $\MS$ be a collection of singleton subsets of $[n]$. We say $\MS$ is a \textit{prime collection of singletons} for $G$ if $\MS$ satisfies the following inductive definition.
\begin{itemize}
    \item The empty set $\MS = \emptyset$ is a prime collection of singletons.
    \item If $|\MS| \ge 1$, then $\MS = \{\{s_1\}, \dots, \{s_t\} \}$ is a prime collection of singletons if, for each natural number $i$ with $1 \le i \le t$,
    $\MS \backslash \{ s_i\}$ is a prime collection of singletons and  $s_i$ satisfies the following two rules:
    \begin{itemize}
        \item[1.] $s_i$ is not a leaf or isolated vertex of $G'$,
        \item[2.] If $s_i$ is adjacent to a leaf of $G'$ then it has degree at least $3$ in $G'$,
    \end{itemize}
    where $G'$ is the induced subgraph of $G$ obtained by deleting the vertices in the set $\{s_1, \dots, s_{i-1},s_{i+1}, \dots, s_{t} \}$.
\end{itemize}
\end{definition}
\noindent 
For ease of notation and when it is clear, we denote a collection of singletons $\MS = \{\{ s_1\}, \dots, \{s_t \} \}$ as $\{s_1, \dots, s_t \}$. 

\begin{definition}
\label{def:prime_collection}
Let $\MS$ be a collection of singleton subsets ${\rm sing}(\MS)$ and $2$-subsets of $V(G)$. Let $G'$ be the induced subgraph of $G$ obtained by removing all vertices  in ${\rm sing}(\MS)$. 
\begin{itemize}
\item We say $\MS$ is a \textit{prime collection} for $G$ if:
\begin{itemize}
    \item The ${\rm sing}(\MS)$ is a prime collection of singletons for $G$.
     \item  The $2$-subsets in $\MS$ are a subset of edges of $G'$ such that for every vertex $v \in V(G')$ which is incident to an edge in $\MS$, there exists an edge $\{v, w \}$ of $G'$ which is not in $\MS$. 
\end{itemize}
\item We say $\MS$ is a \textit{valid collection} for $G$ if each $2$-subset is disjoint from each singleton set.
\end{itemize}
\end{definition}

Note that every prime collection is a valid collection, but the converse is not true.

\begin{example}\label{example:prime_collection}
Consider the graph $G_1$ in Figure~\ref{fig:example_prime_collection}. Let $\MS_1 = \{1, 23, 24, 45 \}$. We see that the singletons, i.e.~$\{1 \}$, form a prime collection of singletons for $G_1$ because $1$ is adjacent to a leaf of $G_1$ and has degree at least three.
Then we consider the induced subgraph $G'$ of $G_1$ obtained by removing the vertex $1$. The vertices which lie in some $2$-subset in $\MS$ are shaded yellow and vertices in $G'$ which do not lie in any $2$-subset are white. Since every yellow vertex is adjacent to a white vertex, we have that the $2$-subsets form a prime collection of $2$-subsets for $G'$. Therefore, $\MS_2$ is a prime collection for $G_1$.

On the other hand, in the graph $G_2$, the set $\MS_2 = \{1,2,34,45,67\}$ is a valid collection since each $2$-subset in $\MS_2$ does not contain any of the singleton subsets. However $\MS$ is not a prime collection because $1$ is a leaf and so the singletons do not form a prime collection of singletons. Additionally $45 \in \MS$ is not an edge of $G_2$.

\begin{figure}
    \centering
    \includegraphics[scale=0.8]{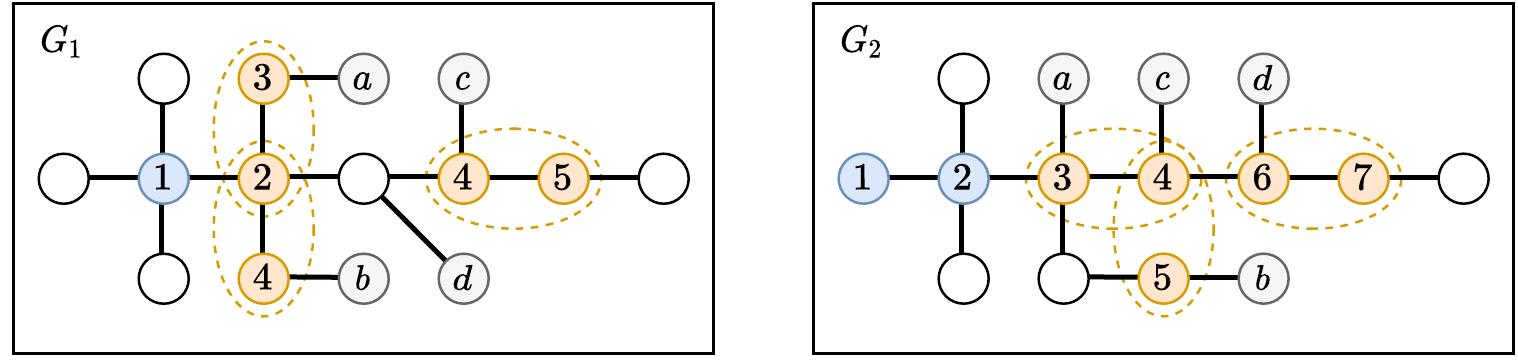}
    \caption{Graphs in Examples~\ref{example:prime_collection} and \ref{example:clouds}. (Left) $G_1$ with prime collection $\MS_1 = \{1, 23, 24, 45\}$ and clouds $234$ and $45$. 
    (Right) $G_2$ has a valid collection $\MS_2 = \{1, 2, 34, 45, 67 \}$ which is not prime. The clouds are $345$ and $67$. 
    }
    \label{fig:example_prime_collection}
\end{figure}

\end{example}

In the following, we show how a valid collection for $G$  can be extended to the set of circuits of a dependent matroid for the consecutive forest hypergraph $\Delta_G$.

\begin{definition}
\label{def:cloud and crossing}
Let $\MS$ be a valid collection.
\begin{itemize}
    \item Let $G(\MS)$ be the graph whose vertex set is given by the union of all $2$-subsets of $\MS$ and whose edges are the $2$-subsets of $\MS$. A \emph{cloud of $\MS$} is the set of vertices of a connected component of $G(\MS)$. 
    
    \item We say that a path $P: v_0,v_1, \dots, v_t$ 
     in $G$ 
    \emph{crosses} the cloud containing $v_j$ and $v_k$ if there exists $i < j < k < \ell$ such that some cloud of $\MS$ contains $\{v_j, v_k\}$ but no cloud of $\MS$ contains either $\{v_i, v_j\}$ or $\{v_k, v_{\ell}\}$. (This may mean that $v_i$ or $v_\ell$ does not lie in any cloud.)

    \item We say that the subset $A\subset V(G)$ is \emph{blocked} by $\MS$ if there exist $v, w \in A$ such that one of the following conditions holds:
    \begin{itemize}
        \item[(i)] The vertices $v$ and $w$ lie in different connected components of $G$.
        \item[(ii)] The vertices $v$ and $w$ lie in the same connected component of $G$ and the unique path from $v$ to $w$ in $G$ crosses a cloud or contains a vertex $i\in V(G)\backslash\{v,w\}$ such that $\{i\}$ is a singleton in $\MS$.
    \end{itemize} 
    
\end{itemize}
\end{definition}

We note that if a set $A \subseteq V(G)$ is not blocked by some valid set $\MS$ then it follows that $A$ is contained within a connected component of $G$.

\begin{example}\label{example:clouds}
In Figure~\ref{fig:example_prime_collection}, we have illustrated a prime collection $\MS_1$ for the graph $G_1$ which contains two clouds $234$ and $45$. The path from $a$ to $b$ crosses the cloud $234$ whereas the path from $c$ to $d$ does not cross any cloud. Therefore any set of vertices containing $a$ and $b$ is blocked by $\MS$, but  
$\{4, c, d \}$ is not blocked by $\MS$.

We also illustrate a non-prime valid collection $\MS_2$ for the graph $G_2$ which contains two clouds. The path from $a$ to $b$ crosses the cloud $345$ since the path passes through both $3$ and $5$. However the path from $c$ to $d$ does not cross any clouds of $\MS_2$
\end{example}

We now give the construction of a matroid whose circuits include $\MS$ for a given valid collection $\MS$.

\begin{proposition}
\label{pro:MS}
Let $\MS$ be a valid collection. Let $\MC$ be the collection of subsets of $[n]$ consisting of the
\begin{enumerate}
    \item singletons in $\MS$,
    \item $2$-subsets of clouds of $\MS$,
    \item $3$-subsets of $[n]$ which are not blocked and do not contain any set in $\MS$, and
    \item $4$-subsets of $[n]$ containing none of the sets listed in 1, 2, or 3 above. 
\end{enumerate} Then $\MC$ is the collection of circuits of a matroid.
\end{proposition}

\begin{proof}
We need to check that $\MC$ satisfies the circuit elimination axiom. That is, if $C_1,C_2\in\MC$ and $x\in C_1\cap C_2$, we must check that there is a set in $\MC$ contained in $(C_1\cup C_2)\backslash\{x\}$. This is clear if either $C_1$ or $C_2$ is a set of size $1$ or $4$. If $C_1$ and $C_2$ both have size $2$, then $C_1$ and $C_2$ are contained in the same cloud. Thus, $(C_1\cup C_2)\backslash\{x\}$ is contained in the same cloud, implying that $(C_1\cup C_2)\backslash\{x\}\in\MC$.

Suppose $|C_1|=2$ and $|C_2|=3$. Let $C_1=\{v,x\}$ and $C_2=\{x,y,z\}$.

Since $C_2$ is not blocked, it is contained in a connected component of $G$.

Let $P_1$ be the path from $x$ to $y$; let $P_2$ be the path from $v$ to $y$; and let $P_3$ be the path from $v$ to $x$. Since $\{x,y,z\}\in\MC$, we have $\{x,y\}\notin\MC$. Therefore, $y$ is not contained in the same cloud that contains $v$ and $x$. Since a cloud contains every vertex on the path between two vertices in the cloud, we see then that either $P_3\subseteq P_1$, or $P_3\subseteq P_2$, or $P_3=P_1\cap P_2$. In either case, the fact that $\{x,y\}$ is not blocked implies that $\{v,y\}$ is not blocked. Similarly, $\{v,z\}$ is not blocked. We already know that $\{y,z\}$ is not blocked because $\{y,z\}\subseteq\{x,y,z\}$. Therefore, we conclude that $\{v,y,z\}$ is not blocked and contains a member of $\MC$.

Finally, we consider the case where $|C_1|=|C_2|=3$. If $|C_1\cap C_2|=1$, then $(C_1\cup C_2)\backslash\{x\}$ has size $4$ and therefore contains a member of $\MC$. If $|C_1\cap C_2|=2$, let $C_1=\{v,x,y\}$ and $C_2=\{w,x,y\}$. We will show that $(C_1\cup C_2)\backslash\{x\}=\{v,w,y\}$ contains a member of $\MC$. Since $\{v,x,y\}$ and $\{w,x,y\}$ are not blocked, the pairs $\{v,y\}$ and $\{w,y\}$ are not blocked. Suppose for a contradiction that the pair $\{v,w\}$ is blocked. Recalling that there is a unique path between any pair of vertices in a connected component of a forest, one can see that either $\{v,x\}$ or $\{w,x\}$ is blocked, a contradiction, or $x$ is contained in the cloud that blocks $\{v,w\}$. So we deduce that $x$ is contained in the cloud that blocks $\{v,w\}$. But then, either $\{v,y\}$ or $\{w,y\}$ is blocked, or $x$ and $y$ are contained in the same cloud, a contradiction.
\end{proof}
\begin{notation} 
Given a valid collection $\MS$, we denote by $M_{\MS}$ its corresponding matroid from Proposition~\ref{pro:MS}.
\end{notation}

\begin{remark}[Flats of $M_\MS$]\label{rem:flats}
The set of loops $L$ of $M_{\MS}$ is the set of singletons in $\MS$. 
Note that the loops of a matroid are contained in every flat of the matroid. The flats of $M_{\MS}$ of rank $1$ are either of the form $B\cup L$, where $B$ is a cloud, or of the form $\{x\}\cup L$, where $x \in [n]$ is a non-loop element contained in no cloud. The flats of rank $2$ are either of the form $F\cup L$, where $F$ is an inclusion-wise maximal set that is not blocked, or of the form $F_1\cup F_2$, where $F_1$ and $F_2$ are flats of rank $1$ such that $\{x_1,x_2\}$ is blocked for every pair $\{x_1,x_2\}\in F_1\cup F_2$. Since, all $4$-subsets of $[n]$ are dependent in $M_{\MS}$, the matroid has rank at most $3$.
\end{remark}

We now show that the matroids $M_\MS$ associated to prime collections $\MS$ in Definition~\ref{def:prime_collection} are precisely the minimal matroids for $\Delta_G$. 
Firstly, we observe that the simplification of $M_\MS$ is the matroid associated to a forest-like configuration.

\begin{proposition}\label{prop: MS is forestlike for prime S}
Let $\MS$ be a prime collection. Then the simplification of $M_\MS$ is the matroid of a forest-like configuration.
\end{proposition}

\begin{proof}
By definition, the simplification $(M_\MS)^s$ is the matroid obtained from $M_\MS$ by deleting all its loops and deleting elements from parallel classes such that each parallel class contains only one element. Since $M_\MS$ has rank at most $3$, it follows that $(M_\MS)^s$ is completely determined by its dependent rank-$2$ flats (that is, the flats of rank $2$ with at least $3$ elements). We define the point and line configuration $\MC$ whose points are the elements of the ground set of $(M_\MS)^s$ and whose lines are the dependent rank-$2$ flats of $(M_\MS)^s$. Clearly, we have that the matroid associated to $\MC$ is equal to $(M_\MS)^s$. It remains to show that $\MC$ is forest-like.

We proceed by induction on $|\MS|$. If $\MS = \emptyset$, then we have that the matroid associated to $\MC$ is equal to $M_\MS$. For any labelling of the points in $\MC$, we show that $G_\MC$ is a disjoint union of paths. By Definition~\ref{def:cloud and crossing}, the $3$-subsets of $V(G)$ contained in a connected component of $G$ are exactly those which are not blocked by $\MS$. So, by the definition of the matroid $M_\MS$, its dependent rank-$2$ flats are the connected components of $G$. It follows immediately from the definition of the graph $G_\MC$ that its connected components are paths. Hence $\MC$ is forest-like.

Assume that $|\MS| \ge 1$. Then, either $\MS$ contains a singleton of $V(G)$ or an edge of $G$.

\noindent \textbf{Case 1.} Let $\{v\} \in \MS$ be a singleton. By definition, we have that $\MS \backslash \{v\}$ is a prime collection. By induction, the simplification $(M_{\MS \backslash \{v\}})^s$ is the matroid of a forest-like configuration $\MC'$. Since $v$ is not contained in any cloud of $\MS$, it follows that $v$ is contained in a unique dependent rank-$2$ flat $F = \{v, f_1, \dots, f_k \}$ of $(M_{\MS \backslash \{v\}})^s$. 
The dependent rank-$2$ flats of $(M_\MS)^s$ are obtained from the dependent rank-$2$ flats of $(M_{\MS \backslash \{v\}})^s$ by replacing $F$ with a disjoint collection of flats $F_1, \dots, F_d$ which partitions $\{f_1, \dots, f_k\}$. Each flat $F_i$ corresponds to a neighbor of $v$ in $G$. To see this, let $A \subseteq V(G)\backslash v$ be a $3$-subset and suppose that $A$ is not blocked by $\MS\backslash\{v\}$. Then $A$ is not blocked by $\MS$ if and only if for any pair of vertices $x, y \in A$ we have that the path from $x$ to $y$ in $G$ does not pass through $v$. Therefore the point and line configuration $\MC$ is obtained from $\MC'$ by replacing the line containing $v$ with $d$ non-intersecting lines. Since $\MC'$ is forest-like, it follows that $\MC$ is also forest-like.

\smallskip

\noindent \textbf{Case 2.} Let $\{v,w\} \in \MS$ be an edge of $G$. Since $\MS$ is a prime collection, the induced subgraph $G'$ of $G$ on the vertices of the cloud of $\MS$ containing $\{v,w\}$ is connected. By assumption $G$ is a forest, hence $G'$ is forest. So, without loss of generality, we may assume that $v$ is a leaf in $G'$. By the definition of a prime collection, it follows that $\MS \backslash \{v,w \}$ is a prime collection. By induction, the simplification $(M_{\MS\backslash\{v,w\}})^s$ is the matroid of a forest-like configuration $\MC'$. Let $F$ be the unique dependent flat of rank $2$ of $(M_{\MS \backslash \{v,w\}})^s$ that contains $v$. The dependent rank-$2$ flats of $(M_\MS)^s$ are obtained from the dependent rank-$2$ flats of $(M_{\MS \backslash \{v,w\}})^s$ by replacing $F$ with a collection of flats $F_1, \dots, F_d$ such that $F_i \cap F_j = \{p\}$ where $p$ is the element of the ground set of $(M_\MS)^s$ that corresponds to the cloud of $\MS$ containing $v$ and $w$. To see this, let $A \subseteq V(G)$ be a $3$-subset and suppose that $A$ is not blocked by $\MS \backslash \{v,w\}$. Then $A$ is blocked by $\MS$ if and only if there exist $x, y \in A$ such that the path from $x$ to $y$ in $G$ crosses a cloud of $\MS$ via the edge $\{v, w\}$. Therefore, the point and line configuration $\MC'$ is obtained from $\MC$ by removing the line containing $v$ and replacing it with $d$ lines that pass through a common vertex. Since $\MC'$ is forest-like, it follows that $\MC$ is also forest-like.
\end{proof}

We now state and prove our main result, and then the auxiliary results used in this proof.
 
\begin{theorem}\label{thm:m3_globally_min}
The minimal matroids for $\Delta_G$ are 
$
\MM_{G} = \{ M_{\MS} : \MS \text{ is a prime collection}\}.
$
\end{theorem}

\begin{proof}
On the one hand, if we take a minimal matroid $M$ for $\Delta_G$, then by Lemma~\ref{lem:m3_globally_min_implies_MS} we have that $M = M_\MS$ for some prime collection $\MS$.
On the other hand, let $\MS$ be a prime collection. First, we observe that, for every prime collection $\MS$, every member of $\Delta_G$ is indeed a dependent set in $M_\MS$. More precisely, $\Delta_G\subseteq\MD(M_\MS)$ because every path in $G$ with three vertices either contains a singleton set in $\MS$, contains two elements from a cloud of $\MS$, or is not blocked. Since $\Delta_G\subseteq\MD(M_\MS)$, Lemma~\ref{lem:m3_globally_min_implies_MS} implies that there is a prime collection $\MT$ such that $\Delta_G\subseteq\MD(M_\MT) \subseteq \MD(M_\MS)$ and such that $M_\MT$ is a minimal matroid for $\Delta_G$. But Lemma~\ref{lem:m3_MS<MT_implies_MS=MT} implies that $\MS = \MT$. Therefore, $M_\MS$ is a minimal matroid for $\Delta_G$.
\end{proof}

\begin{example}\label{example:adjacent_tree_decomposition}
Let $G$ be the graph on vertex set $[7]$ with edges $E(G) = \{12, 23, 34, 45, 56, 47 \}$. The consecutive forest hypergraph for $G$ is given by $\Delta = \{123, 234, 345, 347, 456, 457\}$. In Figure~\ref{fig:adjacent_tree_decomposition_example}, we write down the prime collections $\MS$ for $G$. For each $\MS$, the matroid $M_{\MS}$ is the matroid of a point and line configuration which we illustrate in the figure. For instance if $\MS = \emptyset$, then $M_{\MS}$ is the matroid of the point and line configuration with all seven points lying on a single line. If $\MS = \{34, 45 \}$, then the corresponding configuration 
has a line containing the points $1,2,3,4,5$ and two free points $6$ and $7$. The three points $3,4,5$ coincide since they are a cloud of $\MS$.
\begin{figure}
    \centering
    \includegraphics[width=\textwidth]{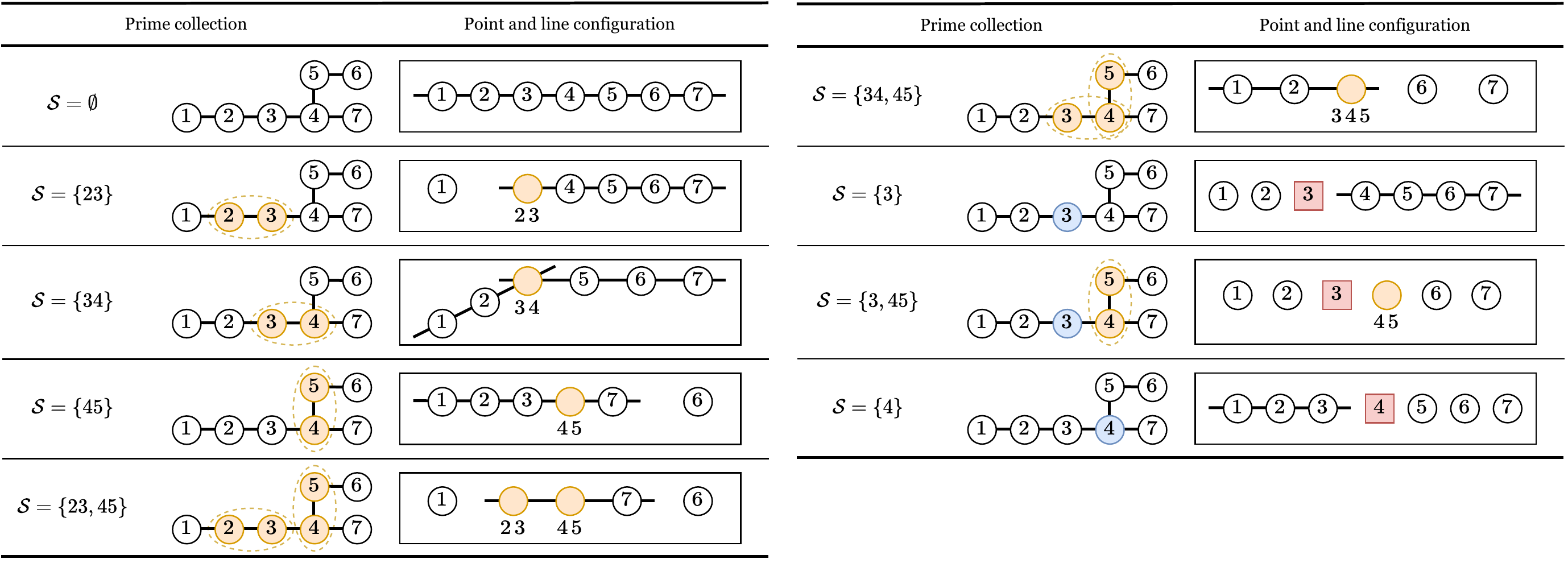}
    \caption{Prime collections $\MS$ and their matroid $M_{\MS}$ as a point and line configuration from Example~\ref{example:adjacent_tree_decomposition}. The dotted pairs of vertices are the two subsets of $\MS$ and the square points in the configurations represent loops.}
    \label{fig:adjacent_tree_decomposition_example}
\end{figure}
\end{example}

We devote the rest of this subsection to prove the lemmas used in the proof of Theorem~\ref{thm:m3_globally_min}. We first prove a matroid result that we will need later.

\begin{lemma}
\label{lem:dep-set-elim}
Let $M$ be a matroid, and let $\{a,b,c,d\}$ be a subset of the ground set of $M$ such that $\{a,d\}$ is an independent set. If $\{a,b,d\}$ and $\{a,c,d\}$ are dependent sets, then so is $\{a,b,c\}$.
\end{lemma}

\begin{proof}
Since $\{a,b,d\}$ is dependent, the rank of $\{a,b,d\}$ is at most $2$. Since $\{a,c,d\}$ is dependent and $\{a,d\}$ is independent, $c$ is in the closure of $\{a,d\}$. Therefore, $c$ is in the closure of $\{a,b,d\}$. Since the rank of $\{a,b,d\}$ is at most $2$ and $c$ is in the closure of $\{a,b,d\}$, it follows that the rank of $\{a,b,c,d\}$ is at most $2$. Therefore, $\{a,b,c\}$ is a dependent set.
\end{proof}

Now we consider the loops of a minimal matroid and give an algorithmic formulation of 
Definition~\ref{def:prime-singletons}. 
Given a collection $\MS$ of singletons, if $\MS$ \textit{passes} the following algorithm,  it is a good candidate to be a prime collection of singletons for $G$.
While every prime collection of singletons will pass Algorithm~\ref{alg:prime_check}, not all sets which pass the algorithm are prime collections of singletons. In order to guarantee that a set $\MS$ is a prime collection of singletons, we require that $\MS$ passes regardless of the ordering of its elements.

\medskip
\begin{algorithm}[H]
    \DontPrintSemicolon
    \KwIn{$\MS = (s_1, \dots, s_t)$ an ordered list of vertices of $G$}
    \KwOut{{\bf pass} or {\bf fail}. (If fail then $\MS$ is not a prime collection of singletons for $G$.)}
    
    \textbf{Initialize:}
    $H \leftarrow G$\;
    \For{$i \leftarrow 1$ \KwTo $t$}{
    \If{$s_i$ is a leaf or isolated vertex of $H$}{\Return {\bf fail}}
    \If{$s_i$ is adjacent to a leaf of $H$ {\bf and} $s_i$ has degree $2$ in $H$}{\Return {\bf fail}}
    $H \leftarrow$ induced subgraph of $H$ obtained by deleting vertex $s_i$
    }
    \Return {\bf pass}
    \caption{Prime collection test}
    \label{alg:prime_check}
\end{algorithm}
\medskip

\begin{proposition}\label{prop:algorithm_prim_singletons}
Let $\MS = \{s_1, \dots, s_t \}$ be a collection of singleton sets of vertices of $G$. Then $\MS$ is a prime collection of singletons for $G$ if and only if $\MS$ passes Algorithm~\ref{alg:prime_check} for any ordering of its elements.
\end{proposition}

\begin{proof}
Suppose that $\sigma \in S_t$ is a permutation such that the algorithm fails with the input $(s_{\sigma(1)}, \dots, s_{\sigma(t)})$. In particular, assume that the algorithm fails at step $ p \in \{1, \dots, t \}$. Then it follows that $\{s_{\sigma(1)}, \dots, s_{\sigma(p)} \}$ is not a prime collection of singletons. So, by Definition~\ref{def:prime-singletons}, the set $\MS$ is not a prime collection of singletons.
Conversely, if $\MS$ is not a prime collection of singletons then, by definition, there exists $p \in \{1, \dots, t\}$ and $j \in \{1, \dots, p \}$ such that $\{s_1, \dots, s_{p-1} \}$ is a prime collection of singletons and either: $s_j$ is a leaf or isolated vertex in $G'$, or $s_j$ is adjacent to a leaf of $G'$ and has degree two, where $G'$ is the induced subgraph of $G$ obtained by deleting the vertices $\{s_1, \dots, s_{j-1}, s_{j+1}, \dots, s_p \}$. Now we order the elements of $\MS$ as:
$
(s_1, \dots, s_{j-1}, s_{j+1}, \dots, s_p, s_j, s_{p+1}, \dots, s_t).
$
By construction, we have that Algorithm~\ref{alg:prime_check} fails in the for-loop when checking vertex $s_j$.
\end{proof}

\begin{lemma}\label{lem:m3_globally_min_loops}
Let $M$ be a minimal matroid for $\Delta_G$. Then the loops of $M$ are a prime collection of~singletons~for~$G$.
\end{lemma}

\begin{proof}
We follow the algorithmic description of the construction of a prime collection of singletons. Let $\MS$ be the loops of $M$. Suppose for a contradiction that $\MS$ fails to be a prime collection of singletons for $G$. Then there exists an ordering of the loops, say $\MS = (s_1, s_2, \dots, s_p)$, such that at step $i$ we are unable to choose $s_i$ for the prime collection. Let $G'$ be the induced subgraph of $G$ obtained by removing vertices $s_1, \dots, s_{i-1}$. In order for us to be unable to choose $s_i$ for the prime collection there are two cases:
\begin{enumerate}
    \item $s_i$ is a leaf or isolated vertex of $G'$,
    \item $s_i$ is adjacent to a leaf and has degree two in $G'$.
\end{enumerate}

\noindent\textbf{Case 1.}  We construct a new matroid $M'$ on $[n]$ as follows. First construct the matroid $M \backslash s_i$ by deleting $s_i$. 

If $s_i$ is a leaf of $G'$, denote by $x$ the unique vertex adjacent to $s_i$. Then, let $M'$ be the matroid obtained from $M \backslash s_i$ by adding to the ground set the element $s_i$ in the same parallel class as $x$.

 If $s_i$ is an isolated vertex of $G'$, note that every neighbor of $s_i$ in $G$ is a loop of $M$. Let $M'$ be the matroid obtained from $M \backslash s_i$ by freely adding $s_i$ to the ground set of $M \backslash s_i$.
 
 In either case, note  that the collection of dependent sets of $M'$ contains $\Delta_G$ because every path with three vertices containing $s_i$ contains either $x$ or some loop of $M'$.

 We now show that $M' \lneq M$. Take any dependent set $D$ of $M'$. If $s_i \in D$, then clearly $D$ is a dependent set of $M$ since $s_i$ is a loop of $M$. On the other hand, if $s_i \not\in D$, then $D$ is a dependent set of $M' \backslash s_i = M \backslash s_i$. Hence $D$ is a dependent set of $M$. To show that $M \neq M'$ we simply note that $s_i$ is a loop in $M$ but is not a loop in $M'$. So we have shown $\Delta_G\subseteq\MD(M')\subsetneq \MD(M)$. Hence $M$ is not a minimal matroid for $\Delta_G$, a contradiction.

\smallskip

\noindent\textbf{Case 2.} Let $s_i$ be a vertex of degree two which is adjacent to a leaf of $G'$. Let $x$ and $y$ be the vertices adjacent to $s_i$, with $x$ being a leaf. We construct a new matroid $M'$ by first deleting $s_i$ to form $M\backslash s_i$. Then we define $M'$ to be the matroid obtained by adding $s_i$ to the same parallel class as $y$ in $M\backslash s_i$. Note that $\MD(M')$ must contain $\Delta_G$ because any path with three vertices containing $s_i$ (and not containing any loop of $M'$) must also contain $y$. To show that $M' \lneq M$ we use the same argument as Case 1, so we have that $\Delta_G\subseteq\MD(M')\subsetneq \MD(M)$. Therefore $M$ is not a minimal matroid for $\Delta_G$, a contradiction.
\end{proof}

We now characterize matroids which are minimal among all the matroids whose circuits of size one and two are exactly the members of $\MS$. One can think of this as a constrained notion of minimality.  In the following lemma, note that $M$ is fixed, and $\MS$ is obtained from $M$. In a matroid, every circuit of size two is disjoint from every circuit of size one. Therefore, $\MS$ is a valid collection for $G$, and $M_{\MS}$ is well-defined.

\begin{lemma}[Constrained minimal matroids]\label{lem:m3_locally_min}
Let $M$ be a matroid with $\Delta_G\subseteq\MD(M)$ whose circuits of size one and two are the members of $\MS$. Then $M_{\MS} \le M$.
\end{lemma}

\begin{proof}
By assumption, all of the circuits of $M_{\MS}$ of size one and two are also circuits of $M$. All $4$-subsets of $[n]$ are dependent in both $M$ and $M_\MS$. Therefore, to show that $M_\MS \le M$, it suffices to show that all $3$-circuits of $M_\MS$ are dependent in $M$. Let $D = \{a,b,c\}$ be such a circuit of $M_\MS$.
Let $G'$ be the induced subgraph of $G$ obtained by removing those vertices $v$ for which $\{ v\} \in \MS$. Now since $D$ is a circuit in $M_\MS$, we have that for any pair of elements $x,y \in D$, there is a path in $G'$ from $x$ to $y$ and it is not blocked by $\MS$. We denote by $\spn(x , y)$ the collection of vertices on the path from $x$ to $y$. Similarly we define the convex hull of a set of vertices:
\[
\spn(A) = \bigcup_{x,y \in A} \spn(x,y).
\]
By abuse of notation, we identify $\spn(A)$ with the induced subgraph of $G'$ whose vertices are $\spn(A)$.

To complete the proof of the theorem, it suffices to prove the following claim. Since, as a result
we have that $\{a,b,c \}$ is a dependent set in $M$. Therefore, every dependent set in $M_\MS$ is also dependent in $M$, as desired. 

\medskip

\noindent \textbf{Claim.} \textit{All $3$-subsets of ${\rm Conv}(a,b,c)$ are dependent in $M$.}

We must show that every $3$-subset $\{a',b',c'\}\subseteq\spn(a,b,c)$ is dependent in $M$. We proceed by induction on $|\spn(a',b',c')|$.
For the base case, suppose that $|\spn(a',b',c')| = 3$. Then $\spn(a',b',c')$ is a path in $G$ with three vertices. By definition of $\Delta_G$, we have that $\{a',b',c' \} \in \Delta_G$. Since $\Delta_G\subseteq\MD(M)$, we have that $\{ a', b', c'\}$ is dependent in $M$.  
For the inductive step, assume that all $3$-subsets of ${\rm Conv}(a,b,c)$ whose convex hull has size at most $k\geq3$ are dependent and $|\spn(a',b',c')| = k+1$. Since $|\spn(a',b',c')| \ge 4$, there is a vertex $d \in \spn(a', b', c')\backslash \{a',b',c'\}$. Note that neither $a'$, $b'$, $c'$, nor $d$ can be a loop of $M$ because then it would be a loop of $M_\MS$. This implies either that $\{a,b,c\}$ is blocked or that an element of $\{a,b,c\}$ is a loop of $M_\MS$. This is impossible since $\{a,b,c\}$ is a circuit of $M_\MS$.

By induction, $\{a',b',d\}$, $\{a',c',d\}$, and $\{b',c',d\}$, are all dependent since their respective convex hulls lie strictly inside $\spn(a',b',c')$. Since $\{a',b',d\}$ and $\{a',c',d\}$ are dependent, Lemma~\ref{lem:dep-set-elim} implies that either $\{a',b',c'\}$ or $\{a',d\}$ is dependent. Similarly, if $\{a',b',c'\}$ is not dependent, then $\{b',d\}$ and $\{c',d\}$ are dependent. Therefore, since $d$ is not a loop, we may assume that the rank of $\{a',b',c',d\}$ is at most one, implying that $\{a',b',c'\}$ is dependent. This completes the proof of the claim.
\end{proof}

\begin{lemma}\label{lem:m3_globally_min_implies_MS}
If $M$ is a minimal matroid for $\Delta_G$, then $M = M_{\MS}$ for some prime collection $\MS$.
\end{lemma}

\begin{proof}
Let $\MS$ be the collection of all circuits of $M$ of size one and two. Since $\Delta_G\subseteq\MD(M)$, Lemma~\ref{lem:m3_locally_min} implies that $M_\MS \le M$. However, $M$ is a minimal matroid. Since $\Delta_G\subseteq\MD(M_\MS)$, the minimality of $M$ implies that $M = M_\MS$. So it suffices to show that $\MS$ is a prime collection.

By Lemma~\ref{lem:m3_globally_min_loops}, we have that the loops of $M$ form a prime collection of singletons for $G$. The collection of $2$-element circuits of a matroid does not intersect the set of loops, so $\MS$ is a valid collection of subsets.

Let $G'$ be the induced subgraph of $G$ obtained by deleting the vertices that are singleton sets in $\MS$. Suppose for a contradiction that $\MS$ is not a prime collection. Then it follows that there exists a vertex $v$ in $G'$ such that, for every vertex $w$ adjacent to $v$ in $G'$, we have that $\{v,w \}$ is a member of $\MS$.

Let $w$ be a vertex adjacent to $v$, and consider the valid collection $\MS'=\MS\backslash\{\{v,w\}\}$. A set is a singleton in $\MS$ if and only if it is a singleton in $\MS'$. Since every vertex adjacent to $v$ is in the same cloud of $\MS$ as $v$, a $3$-subset of $[n]$ is blocked by $\MS$ if and only if it is blocked by $\MS'$. We also have $\Delta_G\subseteq \MD(M_\MS')$.

Since the collection of $2$-subsets in $\MS'$ is properly contained in the the collection of $2$-subsets in $\MS$, we have that $M_\MS'\lneq M_\MS$, a contradiction. Therefore, we conclude that $\MS$ is a prime collection.
\end{proof}

The last ingredient in the proof of Theorem~\ref{thm:m3_globally_min} is to show 
that $M_\MS$ is minimal for any prime collection $\MS$.

\begin{lemma}\label{lem:m3_MS<MT_implies_MS=MT}
If $\MS$ and $\MT$ are prime collections with $M_\MS \ge M_\MT$, then $\MS = \MT$.
\end{lemma}

\begin{proof}
Let us begin by considering the loops of $M_\MS$ and $M_\MT$. Since $M_\MS \ge M_\MT$, we have that every loop of $M_\MT$ is a loop of $M_\MS$. Suppose for a contradiction that $M_\MS$ has strictly more loops than $M_\MT$. Let $G_\MS$ be the induced subgraph of $G$ obtained by removing the loops of $M_\MS$. We define $G_\MT$ similarly. Pick any loop $v$ in $M_\MS$ that is not a loop in $M_\MT$.
We now show that $\deg(v) \geq 2$ in $G_\MT$. Suppose otherwise. Then all but at most one of the neighbors of $v$ are loops in $M_\MT$. Since $M_\MS \ge M_\MT$, all but at most one of the neighbors of $v$ are loops in $M_\MS$ also. But this fact, with the assumption that $v$ is a loop in $M_\MS$, contradicts the assumption that $\MS$ is a prime collection. Therefore, $\deg(v) \geq 2$ in $G_\MT$. Similar reasoning shows that each vertex adjacent to $v$ in $G_\MT$ must be adjacent to a vertex in $G_\MT$ that is not a loop in $M_\MS$.

Since $\MT$ is a prime collection, there is at least one vertex $x$ adjacent to $v$ in $G_\MT$ such that $\{x,v\}\notin\MT$. Since $\deg(v) \geq 2$ in $G_\MT$, let $y\neq x$ be another neighbor of $v$ in $G_\MT$. We have already observed that there must be a vertex adjacent to $x$ in $G_\MT$ that is not a loop in $M_\MS$. Call this vertex $x'$. We may choose $x'$ so that $\{x,x' \}\notin\MS$ because, if there is no such vertex, then $x$ has no neighbor outside the cloud containing $x$ in $G_\MS$, which implies that $\MS$ is not a prime collection. Similarly, we let $y'$ be any vertex adjacent to $y$ such that $y$ is not a loop in $M_\MS$ and such that $\{y,y' \}\notin\MS$.
For $X\subseteq [n]$, let $\text{cl}(X)$ denote the closure of $X$ in $M_\MT$. Since $\Delta_G\subseteq \MD(M_\MT)$, we have that $\{x,x',v \}$ and $\{x,y,v\}$ are both dependent sets in $M_\MT$. By assumption, $\{x,x' \}$ is independent in $M_\MS$, implying that it is independent in $M_\MT$ also. Since $\{x,x',v \}$ is dependent, we have $v\in\text{cl}(\{x,x' \})$. Thus $\text{cl}(\{x,x' \})=\text{cl}(\{x,x',v \})$. Similarly, since $\{x,v\}$ is independent in $M_\MT$ but $\{x, y, v\}$ is dependent, we have that $y\in\text{cl}(\{x,v \})\subseteq\text{cl}(\{x,x',v \})=\text{cl}(\{x,x' \})$. Since $y\in\text{cl}(\{x,x' \})$, we have that $\{x,x',y\}$ is a dependent set in $M_\MT$.
However, by assumption, we have that $\{x,x' \}$ is independent in $M_\MS$. Hence $\{x,x',y \}$ is independent in $M_\MS$ because it is blocked by $v$. So we have found a dependent set in $M_\MT$ which is independent in $M_\MS$, a contradiction. Therefore, we deduce that $M_\MS$ and $M_\MT$ have the same loops.

To show that $\MS = \MT$, it remains to show that the $2$-subsets in $\MS$ are the same as the $2$-subsets in $\MT$. Since $M_\MS$ and $M_\MT$ have the same loops and $M_\MS \ge M_\MT$, it follows that every $2$-subset in $\MT$ is also in $\MS$. Suppose for a contradiction that $\MS$ has strictly more $2$-subsets than $\MT$. Therefore, either $\MS$ has more clouds than $\MT$, or there is a cloud of $\MT$ that is properly contained in a cloud of $\MS$. In either case, since every vertex in the cloud is adjacent to a vertex not in the cloud, there is a $3$-element set that is blocked by $\MS$ but not by $\MT$. This contradicts the assumption that $M_\MS \ge M_\MT$.
\end{proof}

\subsection{Irreducible matroid varieties.}\label{sec:comb_closure_forestlike}

In this subsection, we combine the results of \S\ref{sec:prime-collections-for-m=3}
to obtain a classification of the irreducible components of $V_{\Delta_G}$. 
In particular, we will prove Theorem~\ref{thm:adj-tree-decomp}.
 This gives a straightforward description of the irreducible components.

\begin{proposition}\label{prop:prime_reduction}
Let $G$ be a forest
and $\MS$ be a prime collection for $G$ as in Definition~\ref{def:prime_collection}.
Then for any non-central component $V_{N}$ of $V_{M_\MS}^{\comb}$ there exists another prime collection $\MS'$ for $G$, such that $V_{N}\subseteq V_{M_{\MS'}}$.
\end{proposition}

\begin{proof}

Let $J$ be the ideal associated to the matroid $N$ and for notation we write $N$ as $M_J$. By Proposition~\ref{prop: MS is forestlike for prime S}, the simplification of $M_\MS$ is forest-like. So, by Theorem~\ref{thm:forestlike_comp},

the matroid $M_J$ is obtained from $M_\MS$ by setting all elements inside certain clouds of $\MS$ to be  loops, see Definition~\ref{def: setting points to loops}.  Denote these clouds by $C_1, \dots, C_s$. More precisely, the circuits of $M_J$ are the singleton subsets of 
$C_1\cup \cdots \cup C_s$ along with the 
collection of circuits of $M_\MS\backslash(C_1\cup \cdots \cup C_s)$. Note that $M_\MS < M_J$.

Let $(\MP,\ML, \mI)$ be the point and line configuration whose matroid is the simplification of $M_\MS$. By Theorem~\ref{thm:forestlike_closure}, if a non-central component $J$ of $I_{M_\MS}^{\comb}$ exists, there is a point in $\MP$ contained in at least three lines in $\ML$. This implies that a cloud of $\MS$ contains at least three elements.

Let $C$ be a cloud of $\MS$ that is not one of $C_1, \dots C_s$. Let $p_i\in\MP$ be the corresponding point in the simplification of $M_{\MS}$. 
We define the \emph{star of $p_i$} to be $S(p_i) =\bigcup_{\ell\in \mL_i} \ell$, where $\ell$ is identified with the set of points it passes through.  
By Lemma~\ref{lem:star_irred_comps},
we see that all generators of $I_{M_\MS}$ constructed using points in $S(p_i)$ are also generators of $J$. However, we also see 
that $I_{M_\MS}\nsubseteq J$ because the elements of $C_1\cup \cdots \cup C_s$ are loops of $M_J$. Therefore, it is enough to show that there exists a prime collection $\MS'$ such that:
\begin{itemize}
    \item $M_{\MS'} < M_J$ and
    \item every cloud of $\MS'$ contained in any of $C_1, \dots, C_s$ has size at most two.
\end{itemize}

We will construct $\MS'$ from $\MS$ by an inductive procedure modifying the clouds $C_1, \dots, C_s$. Let $\MS_0=\MS$. At each step of the procedure, we construct $\MS_{i+1}$ from $\MS_i$ by modifying a cloud $C$ of $\MS_i$ with $|C|\geq 3$ such that $C\subseteq C_1\cup\cdots\cup C_s$. By induction, we assume $\MS_i$ is a prime collection. At each step of this procedure, we construct a prime collection $\MS_{i+1}$ such that every set that is dependent in $M_{\MS_{i+1}}$ but not in $M_{\MS_i}$ is dependent in $M_J$. Since $M_\MS < M_J$, this implies that $M_{\MS'} < M_J$.

Let $G'$ be the induced subgraph of $G$ obtained by deleting the singletons in $\MS_i$. Let $G[C]$ be the induced subgraph of $G$ with vertex set $C$, and let $\ell$ be a leaf of $G[C]$. Since $|C| \ge 3$, the neighbor $x$ of $\ell$ in $G[C]$ is not a leaf of $G[C]$. In particular, $\deg(x) \ge 3$ in $G'$, so we can add $x$ to the collection of singletons of $\MS_i$ to result in a new prime collection of singletons. Therefore, we modify $\MS_i$ as follows. Make $x$ a singleton set in $\MS_{i+1}$ and remove all $2$-subsets in $\MS_i$ which contain $x$. We now consider the neighbors of $x$ which lie outside $C$. For each such neighbor $y$ of $x$, if $y$ belongs to a cloud $C'\neq C$ and $x$ is the only neighbor of $y$ not contained in $C'$, then we remove $y$ from the cloud $C'$. This ensures that the new collection is a prime collection. Let $\MS_{i+1}$ be this new prime collection. Note that this procedure does not add any new $3$-circuits because all paths crossing $C'$ that contain $y$ necessarily pass through $x$.

As a result of this procedure, the cloud $C$ no longer contains $x$ or $\ell$ and might also be split into multiple smaller clouds. Note that all other clouds affected by this procedure have decreased in size. Thus, the only subset of $[n]$ that is dependent in $\MS_{i+1}$ but not in $\MS_i$ is the singleton set $\{x\}$, which is a loop of $M_J$. Therefore, $M_{\MS_{i+1}} < M_J$ as desired. We apply this procedure inductively until we obtain a prime collection $\MS'$ such that every cloud of $\MS'$ contained in $C_1\cup\cdots\cup C_s$ has size $2$.
\end{proof}

\begin{theorem}
\label{thm:adj-tree-decomp}
Let $G$ be a forest and let $\Delta_G$ be the corresponding consecutive forest hypergraph. Then,
\[
V_{\Delta_G}=\bigcup_{\MS} V_{M_\MS},
\]
is an irredundant irreducible decomposition
of $V_{\Delta_G}$, where the union is taken over all prime collections $\MS$ of $G$. In particular, a minimal prime decomposition of $\sqrt{I_{\Delta_G}}$ is given as
$
\sqrt{I_{\Delta_G}} = \bigcap_{\MS} I_{M_\MS}.
$
\end{theorem}

\begin{proof}
By Theorem~\ref{thm:m3_globally_min}, the set of all minimal matroids of $\Delta_G$ is $\{M_\MS:\MS\text{ is a prime collection}\}$. It follows, by Proposition~\ref{prop:comb_decomposition}, that
$
V_{\Delta_G} = \bigcup_{\MS} V^{\comb}_{M_{\MS}}.
$
It is easy to see that, the simplification of the matroid $M_\MS$ is~a~forest-like 
configuration. So, the irreducible components of $V^{\comb}_{M_\MS}$ are in one-to-one correspondence with the irreducible components of the combinatorial closure of the matroid variety of this configuration, which are given by Theorem~\ref{thm:forestlike_comp}.
By Proposition~\ref{prop:prime_reduction}, for every non-central component $V_{M_J}$ of $V^{\comb}_{M_\MS}$, there exists a prime collection $\MS'$ such that $V_{M_J} \subseteq V_{M_\MS'}$. Moreover, by Theorem~\ref{thm:realizable}, each of the matroids $M_\MS$ is realizable. Hence, the irredundant irreducible decomposition of $V_{\Delta_G}$ is given by
$
V_{\Delta_G} = \bigcup_{\MS} V_{M_{\MS}}
$, 
as desired. 
\end{proof}

In the example below we show that the irreducible components of $V_{\Delta_G}$ are not necessarily positroid varieties. 
\begin{example}\label{example: non-positroid component E}
Let $G$ be the graph on the vertex set $\{ 1,2, \dots, 9\}$ and edge set $\{12, 23, 14, 45, 56, 47, 78,89\}$. Let $\MS = \{14, 47\}$ be a prime collection. By Remark~\ref{rem:flats}, the matroid $M_\MS$ has three rank $2$ flats given by $12347, 14567$ and $14789$. So the simplification of $M_\MS$ is the matroid of the configuration given by three lines passing through a point. See Example~\ref{ex:combthreelines}. In particular, the ideal of $M_\MS$ is not generated by determinants. So by \cite[Theorem~5.15]{knutson2013positroid}, the matroid $M_\MS$ is not a positroid.
\end{example}

\vspace{-3mm}

\subsection{Realizability.}\label{sec:forestlike_realizable}

We are now ready to consider the minimal matroids for $\Delta_G$. Recall that, if $\MS$ is a valid collection for $G$, then the matroid $M_\MS$ arising from Proposition~\ref{pro:MS} is minimal with respect to $\Delta_G$. Here, we show that $M_{\MS}$ is realizable over the real numbers hence realizable over $\CC$.

\begin{theorem}\label{thm:realizable}
If $\MS$ is a valid collection, then $M_{\MS}$ is $\mathbb{R}$-realizable.
\end{theorem}

\begin{proof}
Let $r$ be the rank function of $M_\MS$. The elements of $M_\MS$ are the vertices of the forest $G$. Let $L$ be the set of loops of $M_\MS$, so $L$ consists of the singletons in $\MS$. We will proceed inductively by considering subgraphs of $G$ with increasing numbers of vertices. For the base case, we begin with the set $L$ of loops of $M$. The restriction $M_\MS|L$ can be represented by the $3\times|L|$ zero matrix.

For the inductive step, suppose that we have shown that $M_\MS|V(G_0)$ is realizable over $\mathbb{R}$ for some subgraph $G_0$ of $G$ containing the vertices in $L$. Let $v\in V(G)\backslash V(G_0)$. We will show that the restriction $M_{\MS}|(V(G_0)\cup\{v\})$ is obtained from $M_{\MS}|(V(G_0))$ either by adding $v$ as a coloop or by freely adding $v$ to a flat. Thus, Lemma~\ref{lem:extension} or Lemma~\ref{lem:coloop} implies that $M_{\MS}|(V(G_0)\cup\{v\})$ is realizable over the reals. We consider the following cases.
\begin{itemize}
\item[(1)] There is a vertex $w\in V(G_0)$ such that $v$ and $w$ are in the same cloud.
\item[(2)] Case (1) does not hold, but there are vertices $w,x\in V(G_0)\backslash L$, not in the same cloud, such that $\{v,w,x\}$ is an unblocked set.
\item[(3)] Neither Case (1) nor Case (2) holds.
\end{itemize}

For Case (1), note that $C\cup L$ is a flat of rank $1$ for every cloud $C$. Therefore, if $C$ is the cloud containing $v$ and $w$, then we obtain $M_{\MS}|(V(G_0)\cup\{v\})$ by freely adding $v$ to the flat $C\cup L$.

In Case (2), the fact that $w$ and $x$ are not in the same cloud and the fact that neither $w$ nor $x$ is in $L$ imply that $r(\{w,x\})=2$. The fact that $\{v,w,x\}$ is an unblocked set implies that $r(\{v,w,x\})=2$ also. Thus, $v$ is in the closure of $\{w,x\}$ in the matroid $M_{\MS}$. Since Case (1) does not hold, $M_{\MS}|(V(G_0)\cup\{v\})$ is obtained by adding $v$ freely to the closure of $\{w,x\}$.

Now we consider Case (3). Since Case (1) does not hold, Proposition~\ref{pro:MS} implies that $M_\MS$ has no circuit of size $2$ consisting of $v$ and some vertex of $G_0$. Since Case (2) does not hold, Proposition~\ref{pro:MS} implies that $M_\MS$ has no circuit of size $3$ consisting of $v$ and two vertices of $G_0$. Therefore, the only circuits of $M_\MS$ containing $v$ and some vertex of $G_0$ are sets of size $4$. Thus, the matroid $M_{\MS}|(V(G_0)\cup\{v\})$ is obtained either by freely adding $v$ to the ground set $V(G_0)$ (if $r(V(G_0))=3$) or by adding $v$ as a coloop (if $r(V(G_0))<3$).
\end{proof}

\section{Hypergraphs \texorpdfstring{$\Delta^{s,t}$}{Dst}}
\label{sec:Delta_st_main}
Here, we study hypergraph ideals which arise naturally in the study of conditional independence statements. We begin with notation and the general setup of the problem. We denote the $k \times \ell$ matrix of integers
\begin{eqnarray}\label{eq:Y}
\MY = (\MY_{i,j})_{i,j} = 
\begin{bmatrix}
    1       & k + 1     & \dots     & (\ell - 1)k + 1 \\ 
    2       & k + 2     & \dots     & (\ell - 1)k + 2 \\
    \vdots  & \vdots    & \ddots    & \vdots \\
    k       & 2k        & \dots     & \ell k
\end{bmatrix} \ .
\end{eqnarray}
For each $i \in [k]$ and $j \in [\ell]$, the rows and columns of $\MY$ are denoted
\[
R_i = \{ \MY_{i,1}, \MY_{i,2}, \dots, \MY_{i, \ell} \} = \{i, k+i, \dots, (\ell - 1)k + i \},
\]
\[
C_j = \{ \MY_{1,j}, \MY_{2,j}, \dots, \MY_{k,j}\} = \{(j-1)k + 1, (j-1)k + 2, \dots, (j-1)k + k \}.
\]
For each $s$ and $t$ with $s \le k$ and $t \le \ell$, we define $\Delta^{s,t}$ be the following collection of subsets of $[k\ell]$, 
\[
\Delta^{s,t} = \bigcup_{1 \le i \le k} \binom{R_i}{t} \cup \bigcup_{1 \le j \le \ell} \binom{C_j}{s}.
\]
If the values of $s$ and $t$ have been fixed, then we simply write $\Delta$ for $\Delta^{s,t}$. We are interested in studying the ideals $I_{\Delta^{s,t}}$ for $d \ge \max\{s,t\}$ as these are examples of conditional independence ideals with hidden variables.

\begin{example}
\label{exa:k2l4s2t3}
Let $k = 4, \ell = 7, s = 2$ and $t = 3$. We have
\[
\MY = 
\begin{bmatrix}
1 & 5 & 9  & 13 & 17 & 21 & 25 \\
2 & 6 & 10 & 14 & 18 & 22 & 26 \\
3 & 7 & 11 & 15 & 19 & 23 & 27 \\
4 & 8 & 12 & 16 & 20 & 24 & 28
\end{bmatrix}
\quad \textrm{and} \quad
\Delta^{2,3} = \left\{ 
\binom{\{1,5,9,13,17,21,25\}}{3} 
\cup
\dots 
\cup
\binom{\{25,26,27,28 \}}{2}
\right\}.
\]
Calculating the dependent matroids for $\Delta^{2,3}$, we find that all such matroids are point and line configurations. In Table~\ref{tab:combinatorial_types_kl}, there are $10$ combinatorial types of configurations which appear as these matroids. Explicitly, these are the point and line configurations which have at most $4$ lines and at most $7$ points.
\begin{table}[h]
    \centering
    \begin{tabular}{c|ccccccccccccc}
    \toprule
        \diagbox{$k$}{$\ell$} & $3$ & $4$ & $5$ & $6$ & $7$ & $8$ & $9$ & $10$ & $11$ & $12$ & $\cdots$ & $\cdots$ & $\infty$ \\
        \midrule
        $2$ & 
        $2$ & $2$ & $3$ & $\textbf{4}$ & $4$ & $4$ & $4$ & $4$ & $4$ & $4$ & $4$ & $4$ & $4$ \\
        $3$ & 
        $2$ & $2$ & $3$ & $5$ & $7$ & $8$ & $\textbf{9}$ & $9$ & $9$ & $9$ & $9$ & $9$ & $9$ \\
        $4$ & 
        $2$ & $2$ & $3$ & $6$ & $10$ & $13$ & $20$ & $23$ & $24$ & $\textbf{25}$ & $25$ & $25$ & $25$ \\
        \bottomrule
    \end{tabular}
    \caption{The number of \textit{combinatorial types} of configurations appearing among dependent matroids for $\Delta^{2,3}$. Note that increasing $\ell$ leads to having arbitrarily many irreducible components for the variety $V_{\Delta^{2,3}}$. 
     }
    \label{tab:combinatorial_types_kl}
\end{table}
\end{example}

\begin{remark}\label{rem:Known_Results}
The minimal prime decomposition of $I_{\Delta^{s,t}}$ has been extensively studied in \cite{herzog2010binomial, Rauh, clarke2020conditional, pfister2019primary}. 
In each case, we find a matroidal description of the prime components. In Table~\ref{tab:min_dep_examples} we give a unified perspective on these results where the minimally dependent matroids can be uniquely identified by their loops.
In the case $k = s = 2$, it is possible to generalize the prime components of $I_{\Delta^{s,t}}$, described in \cite{ollie_fatemeh_harshit} for the $t = 3$, to all $t \ge 3$. A complete description of this can be found in our forthcoming work. 
In this case the minimally dependent matroids are given by configurations of points lying in $(t-1)$-dimensional affine subspaces. When $s = 2$, $t = 3$ and $k,\ell$ are arbitrary, the dependent matroids for $\Delta^{s,t}$ are given by configurations with at most $k$ points and $\ell$ lines. The specific case with $d = k = s = t = 3$ and $\ell = 4$ is given in Example~\ref{ex:andreas}. The ideal $I_{\Delta^{s,t}}$ has two components both of which are matroid varieties corresponding to the configurations in Figure~\ref{threelines}.

\begin{table}[h]
    \centering
    \begin{tabular}{cp{0.55\linewidth}p{0.25\linewidth}}
        \toprule
         $(k,\ell,s,t)$ & Minimally dependent matroids $M_{\MS}$ for $\Delta^{s,t}$ & Simplification \\
        \midrule
        $(k, \ell, 2, 2)$ & 
        The parallel classes of $M_{\MS}$ are the connected components of $\Delta^{s,t}\backslash \MS$. $M_{\MS}$ is minimally dependent if for any $\MT \subseteq \MS$, there exists a parallel class of $M_{\MT}$ which contains at least two distinct parallel classes of $M_{\MS}$.
        &
        $M_\MS$: Uniform matroid on the ground set of parallel classes of $M_\MS$.
        \\
        \midrule
        $(k, \ell, 2, \ell)$ &
        $\MS = \emptyset$: $M_{\emptyset}$ is the uniform matroid on $[k\ell]$ of rank $\ell-1$. \newline
        $\MS \neq \emptyset$:  
        $M_{\MS}$ is minimally dependent if $|\MS \cap R_i| = 1$ for each row $R_i$ and $\MS \cap C_j \neq \emptyset$ for at least two distinct columns $C_j$ of $\MY$. The parallel classes of $M_{\MS}$ are $C_j \backslash \MS$ for each column $C_j$.
        &
        $M_\emptyset$: Already simplified.\newline
        $M_{\MS}$: Uniform matroid on the ground set of parallel classes of $M_{\MS}$.
        \\
        \midrule
        $(2, \ell, 2, 3)$ & 
         $\MS:$ 
         a \textit{minimal set} from \cite[Definition~3.15]{ollie_fatemeh_harshit}. $M_{\MS}$ is the matroid of a point and line configuration. The parallel classes of $M_{\MS}$ are: $\MC = \bigcup C_j$ where the union is taken over all columns $C_j$ of $\MY$ such that $C_j \cap \MS = \emptyset$, and the sets $C_j \backslash \MS$ where $C_j \cap \MS \neq \emptyset$. The circuits of size $3$ in $M_{\MS}$ are given by $3$-subsets of the rows $R_1\backslash\MS$ and $R_2\backslash\MS$ respectively.
         & 
        $M_{\MS}$: Matroid of a point and line configuration with at most two lines intersecting at $\MC$ (possibly empty set). 
        \\
        \bottomrule
    \end{tabular}
    \caption{The minimally dependent matroids for $\Delta^{s,t}$, which are uniquely determined by their loops $\MS \subseteq [k\ell]$.}
    \label{tab:min_dep_examples}
\end{table}\end{remark}

\begin{example}
Let $s = 2$ and $t = 3$. For each $2 \le k \le 4$ and $3 \le \ell$ we calculate the number of possible combinatorial types of points and line configurations which appear among dependent matroids for $\Delta^{s,t}$. These calculations are displayed in Table~\ref{tab:combinatorial_types_kl}. Note that with $12$ points, we observe all combinatorial types of configurations with at most $4$ lines. So, in order to determine the components of the hypergraph variety $V_{\Delta^{s,t}}$, we need only show that a finite number of point and line configurations have irreducible varieties.
\end{example}

\subsection{Grid matroids of small rank.}

Although the minimal matroids for $\Delta^{s,t}$ are known for some specific values of $s$ and $t$, it seems difficult to determine the minimal matroids for arbitrary values of $s$ and $t$. However, there is no difficulty if $d$ is sufficiently small. We first recall the notion of affine matroids.
\begin{definition}
\label{def:affine}
Let $\mathbb{F}$ be a field, and let $S=\{v_1,\ldots,v_k\}$ be a collection of (not necessarily distinct) vectors in $\mathbb{F}^{d-1}$. Let $v_i'\in\mathbb{F}^d$ be the vector whose first coordinate is $1$ and whose other coordinates are those of $v_i$. 
\begin{itemize}
    \item The collection $S$ is \emph{affinely dependent} if $k>0$ and there are elements $a_1,\ldots,a_k\in\mathbb{F}$ that are not all $0$ with $\sum_{i=1}^k a_iv_i={\bf 0}$
 and $\sum_{i=1}^k a_i=0$.
 This is equivalent to the condition that $\{v_1',\ldots,v_k'\}$ is linearly dependent. 
 \item
 A matroid $M$ on the set $[n]$ is \emph{affine} over the field $\mathbb{F}$ if there is a function $\phi:[n]\rightarrow\mathbb{F}^{d-1}$ such that $X\subseteq[n]$ is an independent set in $M$ if and only if $\phi(X)$ is affinely independent. In this case, $M$ has rank at most $d$ and can be realized by a $d\times n$ matrix whose columns are $v_1',\ldots,v_n'$. An affine matroid must be loopless, and it is simple if and only if $\phi$ is injective. If $F$ is a flat of $M$ of rank $t$, then there is a $(t-1)$-dimensional affine subspace of $\mathbb{F}^{d-1}$ whose intersection with $[n]$ is $\phi(F)$.
 \end{itemize}
 \end{definition}

\begin{theorem}\label{thm:s-t-3}
Let $s,t,k,\ell,d$ be positive integers such that $3\leq s\leq t\leq\ell$, $s\leq k$, and $t\leq d\leq s+t-3$. Then $\MC=\min(\Delta^{s,t}\cup\binom{[k\ell]}{d+1})$ is the collection of circuits of an $\mathbb{R}$-realizable matroid on $[k\ell]$ of rank $d$. This is the unique minimal matroid for $\Delta^{s,t}$ in this case.
\end{theorem}

\begin{proof}
One can check fairly easily that $\mathcal{C}$ satisfies circuit elimination and therefore is the collection of circuits of a matroid. However, this will be unnecessary because we will prove this theorem by constructing a realization for an affine matroid over $\mathbb{R}$ whose circuits are $\mathcal{C}$. Because $\mathcal{C}$ is itself the collection of circuits of a matroid $M$, there can be no matroid $N\neq M$ with rank at most $d$ and ground set $[k\ell]$ such that $\Delta^{s,t}\subseteq\MD(N)\subseteq\MD(M)$. Thus, $M$ is the unique minimal matroid for $\Delta^{s,t}$ with $d\leq s+t-3$.

To construct a realization for an affine matroid over $\mathbb{R}$ whose circuits are $\mathcal{C}$, we must define a function $\phi:[k\ell]\rightarrow \mathbb{R}^{d-1}$.
In such a matroid, each $R_i$ must be a flat of rank $t-1$ and each $C_j$ must be a flat of rank $s-1$.
A flat of rank $d-m$ in an affine matroid over $\mathbb{R}$ corresponds to an affine subspace of $\mathbb{R}^{d-1}$ of dimension $d-m-1$. For each such subspace, there are $m$ distinct affine hyperplanes in $\mathbb{R}^{d-1}$ whose intersection is the subspace. Therefore, each $\phi(R_i)$ must be defined by the intersection of a collection $\mathcal{H}_{R_i}$ of $d-t+1$ affine hyperplanes. These hyperplanes are defined by the following equations, with the matrix $A_i=[a_{p,q}]$ having full row rank.
\[\begin{array}{lccclcl}
 a_{1,1}x_1 & +& \cdots& + &a_{1,d-1}x_{d-1} & = & c_1\\
&&&&&\vdots&\\
 a_{d-t+1,1}x_1 & +& \cdots &+ &a_{d-t+1,d-t+1}x_{d-1} & = & c_{d-t+1} \\
\end{array}\]

We choose $k$ such collections of hyperplanes $\mathcal{H}_{R_1}, \mathcal{H}_{R_2},\ldots,\mathcal{H}_{R_k}$ such that the rank-$(t-1)$ subspaces they define are in ``general position''. That is, for every $n<t$, the intersection of every collection of $n$ such subspaces is a subspace of rank $t-n$.
Similarly, each $\phi(C_j)$ must be defined by the intersection of a collection $\mathcal{H}_{C_j}$ of $d-s+1$ hyperplanes defined by the following equations, with the matrix $B_j=[b_{p,q}]$ having full row rank.  
\[\begin{array}{lccclcl}
 b_{1,1}x_1 & +&\cdots&+ &b_{1,d-1}x_{d-1} & = & c_{d-t+2}\\
&&&&&\vdots&\\
 b_{d-s+1,1}x_1 & +&\cdots&+ &b_{d-s+1,r-1}x_{d-1} & = & c_{2d-s-t+2} \\
\end{array}\]

Choosing the subspaces defined by $\mathcal{H}_{R_1}, \mathcal{H}_{R_2},\ldots,\mathcal{H}_{R_k},\mathcal{H}_{C_1},\mathcal{H}_{C_2},\ldots,\mathcal{H}_{C_\ell}$ will ensure that every subset of $\phi([k\ell])$ that should be affinely dependent is indeed so. However, we must also ensure that every subset of $\phi([k\ell])$ that should be affinely independent is so. That is, for every subset $X$ with $|X|\leq d$ such that no  $t$-subset of any $R_i$ or $s$-subset of any $C_j$  is contained in $X$, $\phi(X)$ must be affinely independent. To do this, we will choose the collections $\mathcal{H}_{C_1},\mathcal{H}_{C_2},\ldots,\mathcal{H}_{C_\ell}$ successively. Suppose we have chosen the collections $\mathcal{H}_{C_1},\ldots,\mathcal{H}_{C_{j-1}}$ and now must choose $\mathcal{H}_{C_j}$. If $X$ is a subset of $C_1\cup\dots\cup C_j$ such that $\phi(X)$ is to be affinely independent, then $\phi(X-C_j)$ is a basis for a subspace $S$ of $\mathbb{R}^{d-1}$. We must choose $\mathcal{H}_{C_j}$ so that, for each $i\leq k$, the intersection of all hyperplanes in $\mathcal{H}_{R_i}\cup\mathcal{H}_{C_j}$ is a subspace of rank $d-(d-s+1)-(d-t+1)=d-(2d-s-t+2)=s+t-d-2$ that avoids all such subspaces $S$.

The requirement that this subspace has rank $s+t-d-2$ is equivalent to the matrix $\left[\begin{array}{c}
A_i\\
\hline
B_j\\
\end{array}\right]$ having full row rank. This subspace must be nonempty since $R_i\cap C_j\neq\emptyset$. Thus, we must have $s+t-d-2\geq1$, which is true because $d\leq s+t-3$. The requirement of avoiding all subspaces defined by independent sets can be achieved as there are only finitely many such subspaces to avoid but infinitely many collections to choose to be $\mathcal{H}_{C_j}$.
\end{proof}

\subsection{General grid matroids.}
\label{sec:general}
Given a collection $\mathcal{D}$ of subsets of a ground set $E$, we wish to find the matroids that are minimally dependent for $\mathcal{D}$. However, for our purposes here, it is necessary to find the matroids minimally dependent for $\mathcal{D}$ among the class of realizable matroids.

Some work toward these ideas has been done by Mart\'{i}-Farr\'{e} \cite{M14}. Below, we recall the algorithmic procedure used in \cite{M14} to obtain minimally dependent matroids. This algorithm is not quite completely satisfactory from our perspective for two reasons. First, the algorithm does not take realizability into account at all. Mart\'{i}-Farr\'{e} and de Mier worked with realizability in \cite{MM15}, but as they said there, ``The problems under consideration are far from being solved." Second, although the algorithm is guaranteed to give all of the minimally dependent matroids, it may also give additional matroids that are not minimally dependent. One must compare the matroids given by the algorithm to determine which ones are minimally dependent.

In fact, we will use the algorithm of Mart\'{i}-Farr\'{e} to prove a result that illustrates the difficulties of determining the minimal matroids for $\Delta^{s,t}$ in general. First, we recall some terminology and notation, much of it coming from \cite{M14}. A \emph{clutter} on a set $\Omega$ is a collection of subsets of $\Omega$ such that no set is contained in another. (The term \emph{clutter} is another word for what we have been calling a \emph{simple hypergraph}.) If $\Lambda$ is a clutter on $\Omega$, let $\Lambda^+=\{A\subseteq\Omega:A_0\subseteq A\textnormal{ for some }A_0\in\Lambda\}$. If $\Upsilon$ is any collection of subsets of $\Omega$, then recall from Section~\ref{sec:pre} that $\min(\Upsilon)$ denotes the clutter of inclusion-wise minimal sets in $\Upsilon$. For $B\subseteq \Omega$, define \[I_{\Lambda}(B)=\bigcap_{A\in\Lambda,A\subseteq B} A.\] It follows from the circuit elimination axiom that $\Lambda$ is the collection of circuits of a matroid if and only if $I_{\Lambda}(A_1\cup A_2)=\emptyset$ for all $A_1,A_2\in\Lambda$ with $A_1\neq A_2$. 
To describe the algorithm, we use the notation of Mart\'{i}-Farr\'{e} and de Mier in \cite{MM17}. Let $\Lambda$ be a clutter with $A_1,A_2\in\Lambda$. An $\alpha_1$-transformation of $\Lambda$ is 
\[
\alpha_1(\Lambda;A_1,A_2)=\begin{cases}
                  \min(\Lambda\cup\{A_1\cap A_2\})&\textnormal{if } I_{\Lambda}(A_1\cup A_2)\neq\emptyset,\\
                  \Lambda&\textnormal{otherwise.}\\
\end{cases}
\]
The $\alpha_2$- and $\alpha_3$-transformations of $\Lambda$ are
\[\alpha_2(\Lambda)=\min(\Lambda\cup\{(A_1\cup A_2)\backslash\{x\}:A_1,A_2\in\Lambda,A_1\neq A_2,x\in A_1\cap A_2\})\textnormal{ and}
\]
\[\alpha_3(\Lambda)=\min(\Lambda\cup\{(A_1\cup A_2)\backslash I_{\Lambda}(A_1\cup A_2):A_1,A_2\in\Lambda,A_1\neq A_2\}).
\]
The following result is proved in \cite[Theorem 13]{M14}.

\begin{theorem}
Let $\Lambda\neq\{\emptyset\}$ be a clutter on a finite set $\Omega$ and let $M$ be a minimally dependent matroid for $\Lambda$ whose collection of circuits is $\mathcal{C}(M)$. There is a sequence of clutters $\Lambda=\Lambda_0,\Lambda_1,\dots,\Lambda_r=\mathcal{C}(M)$ such that, for each $i\geq1$, we have $\Lambda_{i-1}^+\subsetneq\Lambda_i^+$ and such that $\Lambda_i$ is either an $\alpha_1$-, $\alpha_2$-, or $\alpha_3$-transformation of $\Lambda_{i-1}$.
\end{theorem}

So, in order to obtain all minimally dependent matroids for a clutter, it suffices to perform every possible combination of $\alpha_1$-, $\alpha_2$-, and $\alpha_3$-transformations on the clutter. Then, one must compare the resulting matroids and discard any matroid whose collection of dependent sets strictly contains the collection of dependent sets of another matroid obtained from the algorithm.
Although the elements of the ground set of $\Delta^{s,t}$ are the integers $1,2,\dots,\ell k$, it will be convenient to think of each element as an ordered pair $(i,j)$, where $(i,j)$ is the unique element of $R_i\cap C_j$. In the remainder of this section, we show that every matroid can be obtained as a restriction of a matroid obtained from some $\Delta^{s,t}$ by $\alpha_1$- and $\alpha_2$-transformations. We prove the following result.

\begin{theorem}
\label{thm:hardness}
Let $M$ be a matroid on ground set $[n]$. There are positive integers $s,t,k,\ell$ and a sequence of clutters $\Delta^{s,t}=\Lambda_0,\Lambda_1,\dots,\Lambda_{r}$ on the set $[k\ell]$ such that, for each positive integer $i\leq r$, $\Lambda_i$ is either an $\alpha_1$- or an $\alpha_2$-transformation of $\Lambda_{i-1}$ and $\Lambda_{r}$ is the collection of circuits of a matroid with a restriction isomorphic to $M$.
\end{theorem}

We will prove Theorem~\ref{thm:hardness} after giving some definitions and proving some lemmas below. We remark that these lemmas and Theorem~\ref{thm:hardness}, as well as their proofs, are true if one replaces $\alpha_2$-transformations with $\alpha_3$-transformations.
We will use $\alpha_1$- and $\alpha_2$-transformations to obtain a matroid whose restriction to $\{(1,j):1\leq j\leq n\}$ is isomorphic to $M$. That isomorphism will be $j\rightarrow(1,j)$ for each $j\in[n]$. Thus, to simplify notation, we identify the element $(1,j)\in R_1$ 
with the element $j\in[\ell]$. (Similarly, we identify subsets of $R_1$~with~subsets~of~$[\ell]$.)

\begin{definition}
\label{def:M'}
Let $t\geq3$, and let $M$ be a matroid of rank $t-1$ on $[n]$. Let $c$ be the number of loops of $M$. Let $c'=\max\{c,t-1\}$ and $n'=n+c'-c$. Let $M'$ be the matroid on $[n']$ obtained from $M$ by adding $c'-c$ loops. Let $\ell=n'+2(t-2)$ and $M^+$ be the matroid on $[\ell]$ obtained from $M'$ by freely adding $\ell-n'$ elements to~the~flat~$[n']$.
\end{definition}

Note that $c'$ is the number of loops of both $M'$ and $M^+$. Also note that, since the rank of $M'$ is $t-1$, there are at least $t-1$ non-loop elements of $M'$. Therefore, $n'\geq t-1+c'\geq2(t-1)$.

\begin{definition}
\label{def:Lambda_p}
Let $M$ be a matroid of rank $t-1$ on ground set $[n]$, where $t\geq3$. Let $\ell=n'+2(t-2)$, while $s=3$ and $k=5$. For a non-negative integer $p\leq t-1$, let $\Lambda_p(M)$ be the clutter on $[k\ell]$ consisting of the inclusion-wise minimal sets among
\begin{itemize}
    \item[(1)] $3$-subsets of $C_j$, where $1\leq j\leq\ell$,
    \item[(2)] $t$-subsets of $R_i$, where $1\leq i\leq k$,
    \item[(3)] subsets of $[n']\subseteq R_1$ of size at least $(t-p)$ that are circuits of $M'$,
    \item[(4)] sets of the form $A\cup\{x\}\subseteq R_1$, where $x\in\{\ell-2p+1,\ell-2p+2\}$ and $A$ is a $(t-p-1)$-subset of $[\ell-2p]$ containing exactly one circuit of $M'$, and
    \item[(5)] sets of the form $H_i\cup K_j$ where $H_i\subseteq R_i\backslash C_j$ and $K_j\subseteq C_j\backslash R_i$, with $|H_i|=t-1$ and $|K_j|=2$.
\end{itemize}
We define a \emph{cross set} to be a set of the form described in (5).
\end{definition}

Note that $\Lambda_0(M)=\alpha_2(\Delta^{s,t})=\alpha_3(\Delta^{s,t})$ and the members of $\Lambda_p(M)$ must be inclusion-wise minimal. Hence, not all sets listed under (1)-(5) above are necessarily members of $\Lambda_p(M)$.~Therefore, we need~the~following~result.

\begin{lemma}
\label{lem:ciruits-included}
Every subset of $R_1$ of size at least $t-p$ that is a circuit of $M'$  is a member of $\Lambda_p(M)$.
\end{lemma}

\begin{proof}
It suffices to show that no member of $\Lambda_p(M)$ is properly contained in a circuit of $M'$. Other than circuits of $M'$, the only members of $\Lambda_p(M)$ that are contained in $R_1$ are $t$-subsets of $R_1$ and the sets described in (4) of Definition~\ref{def:Lambda_p}. No circuit of $M'$ has size greater than $t$. Therefore, no $t$-subset of $R_1$ is properly contained in a circuit of $M'$.
If $p=t-1$, then $t-p-1=0$. Since the empty set is never a circuit of a matroid, none of the sets described in (4) of Definition~\ref{def:Lambda_p} exist. If $p\leq t-2$, then $\ell-2p+1\geq \ell-2(t-2)+1 = n'+2(t-2)-2(t-2)+1>n'$. Since $[n']$ is the ground set of $M'$, no set described in (4) is contained in a circuit of $M'$.
\end{proof}

\begin{lemma}
\label{lem:p-1_to_p}
Let $M$ be a matroid of rank $t-1$, where $t\geq3$. Let $1\leq p\leq t-1$. There is a sequence of clutters $\Lambda_{p-1}(M)=\Lambda_{p,0},\Lambda_{p,1},\dots,\Lambda_{p,q}=\Lambda_p(M)$, where $\Lambda_{p,i}$ is an $\alpha_1$-transformation of $\Lambda_{p,i-1}$.
\end{lemma}

\begin{proof}
Note that $\Lambda_p(M)=\min(\Lambda_{p-1}(M)\cup\Gamma)$, where $\Gamma$ is the collection of
\begin{itemize}
\item $(t-p)$-subsets of $[n]\subseteq R_1$ that are circuits of $M'$ and
\item sets of the form $A\cup\{x\}$, where $x\in\{\ell-2p+1,\ell-2p+2\}$ and $A$ is a $(t-p-1)$-subset of $[\ell-2p]$ containing exactly one circuit of $M'$.
\end{itemize}
By induction, assume $\Lambda_{p,i-1}=\min(\Lambda_{p-1}(M)\cup\Gamma')$, where $\Gamma'\subseteq\Gamma$. Let $X\in\Gamma\backslash\Gamma'$. We proceed by~taking~cases~on~$p$.

\smallskip

\noindent{\bf Case 1.} Assume that $p=1$. Let $m_1,m_2\in[\ell]\backslash X$. (This is possible because $\ell\geq t+1$ if and only if $n'\geq-t+5$. We have $t\geq3$, implying that $3t\geq9$. This implies $n'\geq2t-2\geq-t+7$.) Let $K_{m_1}=\{(2,m_1),(3,m_1)\}$ and $K_{m_2}=\{(4,m_2),(5,m_2)\}$. Note that $X\cup K_{m_1}$ and $X\cup K_{m_2}$ are cross sets and therefore members of $\Lambda_{p,0}$. Also note that $X\cup K_{m_1}$ and $X\cup K_{m_2}$ are the only members of $\Lambda_{p,i-1}$ contained in $X\cup K_{m_1}\cup K_{m_2}$. Therefore, $\alpha_1(\Lambda_{p,i-1};X\cup K_{m_1},X\cup K_{m_2})=\min(\Lambda_{p,i-1}\cup\{X\})$. By iterating this process, we obtain $\Lambda_1(M)$.
\smallskip

\noindent{\bf Case 2.} Assume that $p>1$. By Definition~\ref{def:Lambda_p}, we see that $X\cup\{\ell-2(p-1)+1\}\in\Lambda_{p-1}(M)$ and $X\cup\{\ell-2(p-1)+2\}\in\Lambda_{p-1}(M)$. Also note that $X\cup\{\ell-2(p-1)+1\}$ and $X\cup\{\ell-2(p-1)+2\}$ are the~only~members of $\Lambda_{p,i-1}$ contained in $X\cup\{\ell-2(p-1)+1,\ell-2(p-1)+2\}$. This follows from the fact that all members of $\Lambda_{p,i-1}$ are inclusion-wise minimal (so $X\cup\{\ell-2(p-1)+1,\ell-2(p-1)+2\}$ is not a member of $\Lambda_{p,i-1}$) and the fact that $t\geq3$ (so $\{\ell-2(p-1)+1,\ell-2(p-1)+2\}$ is not a member of $\Lambda_{p,i-1}$).  Therefore, $\alpha_1(\Lambda_{p,i-1};X\cup\{\ell-2(p-1)+1\},X\cup\{\ell-2(p-1)+2\})=\min(\Lambda_{p,i-1}\cup\{X\})$. By iterating this process, we obtain $\Lambda_p(M)$.
\end{proof}

\begin{proof}[{\bf Proof of Theorem~\ref{thm:hardness}}]
First we consider the case where the rank of $M$ is at most $1$. Let $c$ be the number of loops of $M$. Let $s=t=k=2$ and $\ell=\max\{c+1,n\}$. For $i\leq c$, let $\Lambda_i=\alpha_1(\Lambda_{i-1};\{(1,i),(1,\ell)\},\{(1,i),(2,i)\})=\Lambda_{i-1}\cup\{(1,i)\}$. The result is that $\{(1,j)\}\in\Lambda_c$ for every $j\leq c$. Now, let $\Lambda_{c+1}=\alpha_2(\Lambda_c)$. This is the union of $\Lambda_c$ with the collection of $2$-subsets of $\{(2,j):1\leq j\leq\ell\}\cup\{(1,j):c+1\leq j\leq\ell\}$. This is the collection of circuits of a matroid whose restriction to $\{(1,j):1\leq j\leq n\}$ is isomorphic to $M$.

Now let $M$ be a matroid on $[n]$ of rank $t-1$, where $t\geq3$.  Let $c$, $c'$, $n'$, $M'$, $\ell$, and $M^+$ be as in Definition~\ref{def:M'}, while $s=3$ and $k=5$. As noted above, $\alpha_2(\Delta^{s,t})=\Lambda_0(M)$. By repeated use of Lemma~\ref{lem:p-1_to_p}, we can obtain $\Lambda_{t-1}(M)$ from $\Lambda_0(M)$ by a sequence of $\alpha_1$-transformations.

Now, let $\Gamma$ be the clutter consisting of all $2$-subsets of columns $C_j$ for $1\leq j\leq\ell$ other than the subsets containing a loop of $M'$. We will construct a sequence of clutters $\Lambda_{t-1}(M)=\Lambda'_0,\Lambda'_1,\dots,\Lambda'_q=\min(\Lambda_{t-1}(M)\cup\Gamma)$ such that $\Lambda'_i$ is an $\alpha_1$-transformation of $\Lambda'_{i-1}$. By induction, assume $\Lambda'_{i-1}=\min(\Lambda_{t-1}(M)\cup\Gamma')$, where $\Gamma'\subseteq\Gamma$. Let $X\in\Gamma\backslash\Gamma'$. Then there are positive integers $j,m_1,m_2$ such that $X=\{(m_1,j),(m_2,j)\}$. Let $\{m_3,m_4\}\subseteq\{2,3,4,5\}\backslash\{m_1,m_2\}$. Let $K_{m_3}\subseteq R_{m_3}\backslash C_j$ and $K_{m_4}\subseteq R_{m_4}\backslash C_j$ with $|K_{m_3}|=|K_{m_4}|=t-1$. Moreover, we can choose $K_{m_3}$ and $K_{m_4}$ so that no column contains an element of both. (This is possible because $\ell\geq2(t-1)+1$ if and only if $n'\geq3$. This is true because $n'\geq2(t-1)\geq4$.) Note that the only members of $\Lambda'_{i-1}$ contained in $X\cup K_{m_3}\cup K_{m_4}$ are the cross sets $X\cup K_{m_3}$ and $X\cup K_{m_4}$. Therefore, $\alpha_1(\Lambda'_{i-1};X\cup K_{m_3},X\cup K_{m_4})=\min(\Lambda'_{i-1}\cup\{X\})$. By iterating this process, we obtain $\min(\Lambda_{t-1}(M)\cup\Gamma)$.
Let $\Lambda''=\min(\Lambda_{t-1}(M)\cup\Gamma)$. By Lemma~\ref{lem:ciruits-included}, the circuits of $M'$ are members of $\Lambda_{t-1}(M)$. The circuits of $M^+$ that are not circuits of $M'$ are the $t$-subsets of $R_1$ not containing a circuit of $M'$. Thus, the members of $\Lambda''$ are:
\begin{itemize}
    \item $2$-subsets of $C_j$, where $1\leq j\leq\ell$, other than the subsets containing a loop of $M'$,
    \item subsets of $R_1$ that are circuits of $M^+$, and
    \item $t$-subsets of $R_i$, where $2\leq i\leq k$.
\end{itemize}
Recall that $c$ is the number of loops of $M$ and $c'$ is the number of loops of $M'$. Without loss of generality, assume that the loops of $M'$ are $\{n-c+1,n-c+2,\dots,n'\}$.

Let $\Psi$ be the clutter consisting of singleton sets $\{(i,j)\}$, where $2\leq i\leq5$ and either $1\leq j\leq n-c$ or $n'+1\leq j\leq\ell$. (So $j$ is not a loop of $M^+$.) We will construct a sequence of clutters $\Lambda''=\Lambda''_0,\Lambda''_1,\dots,\Lambda''_z=\min(\Lambda''\cup\Psi)$ such that $\Lambda''_i$ is an $\alpha_1$-transformation of $\Lambda''_{i-1}$. By induction, assume $\Lambda''_{i-1}=\min(\Lambda''\cup\Psi')$, where $\Psi'\subseteq\Psi$. Let $\{(i,j)\}\in\Psi\backslash\Psi'$. Let $A_1=\{(i,j)\}\cup\{(i,m):n-c+1\leq m\leq n-c+t-1\}$ and let $A_2=\{(i,j),(1,j)\}$. Since $n'\geq n-c+t-1$, no element of $A_1\cup A_2$ is a singleton set in $\Psi$. Thus, $A_1$ and $A_2$ are the only members of $\Lambda''_{i-1}$ contained in $A_1\cup A_2$ and $\alpha_1(\Lambda''_{i-1};A_1,A_2)=\min(\Lambda''\cup\{(i,j)\})$. By iterating this process, we obtain $\min(\Lambda''\cup\Psi)$.

Let $\Lambda'''=\min(\Lambda''\cup\Psi)$. The members of $\Lambda'''$ are
\begin{itemize}
    \item singleton sets $\{(i,j)\}$ where $2\leq i\leq5$ and either $1\leq j\leq n-c$ or $n'+1\leq j\leq\ell$,
    \item $2$-subsets of $C_j\backslash\{(1,j)\}$, where $n-c+1\leq j\leq n'$,
    \item $t$-subsets of $\{(i,j):n-c+1\leq j\leq n'\}$ for some fixed $i$ with $2\leq i\leq5$, and
    \item subsets of $R_1$ that are circuits of $M^+$.
\end{itemize}
Note that each member of $\Lambda'''$ is either contained in $R_1$ or disjoint from $R_1$.
For every clutter, a process of repeated $\alpha_2$-transformations on the clutter must eventually terminate in a clutter that is the collection of circuits of a matroid. Therefore, repeated $\alpha_2$-transformations on $\Lambda'''$ will result in a clutter that is the collection of circuits of a matroid $N$ whose circuits are either contained in $R_1$ or disjoint from $R_1$. Since the members of $\Lambda'''$ contained in $R_1$ are the circuits of $M^+$, circuit elimination implies that the circuits of $N$ contained in $R_1$ are precisely the circuits of $M^+$. In particular, the restriction of $N$ to $\{(1,j):1\leq j\leq n\}$ is isomorphic to $M$.
\end{proof}

Theorem~\ref{thm:hardness} applies to all matroids, regardless of rank. However, in the case where the matroid is simple and has rank at most $3$, the result can be illustrated more clearly using a different procedure from the one used in the proof of the theorem. We use point and line configurations.
Consider the Fano matroid whose point and line configuration is the Fano plane given in Figure~\ref{fig:fano_components}.
It is well-known that the Fano matroid is not realizable over any field of characteristic other than $2$; therefore, it is not $\mathbb{C}$-realizable. In the case where $n=7$ and \vspace{-2mm}
$$
\MD=\left\{L_1=\{1,2,4\},L_2=\{1,3,6\},L_3=\{1,5,7\},L_4=\{2,3,5\},L_5=\{2,6,7\},L_6=\{3,4,7\},L_7=\{4,5,6\}\right\}\cup\binom{[n]}{4},
$$ 
the only minimally dependent matroid with respect to $\min(\MD)$ is the Fano matroid, which is not relevant for our purposes because it is not $\mathbb{C}$-realizable. However, this can arise as the simplification of a matroid coming from a point and line configuration with $s=2$ and $t=3$.

Indeed, consider the case $k=\ell=7$. Let $(i,j)$ be the unique element contained in $R_i\cap C_j$. Iteratively performing $\alpha_1$-transformations (defined earlier in this subsection), we can obtain from the clutter $\Delta^{2,3}$ a clutter that contains all of the singletons in $R_i$ except for the elements $(i,j)$ such that $j\in L_i$. In particular, the members of this clutter are the following:
\begin{itemize}
    \item all singletons except for the elements $(i,j)$ such that $j\in L_i$,
    \item $3$-subsets of the form $\{(i,j)\in R_i:j\in L_i\}$, and
    \item $2$-subsets of $C_j$ of the form $\{(i_i,j),(i_2,j)\}$, where $j\in L_{i_1}\cap L_{i_2}$.
\end{itemize} Then, by performing two transformations (both $\alpha_2$-transformations, both $\alpha_3$-transformations, or one of each), we obtain a clutter that is the collection of circuits of a matroid whose simplification is the Fano matroid and whose rank-$1$ flats each have three non-loop elements. The second transformation is needed to include the Fano matroid's circuits of size $4$. (This works because, for each $4$-element circuit $C$ of the Fano matroid, there is a fifth element $x$ such that $C\cup \{x\}$ is the union of two $3$-element circuits whose intersection is $x$.)

\section{Conditional independence models}\label{sec:Applications}
\label{sec:applications_CI}

Conditional independence (CI) models play an important role in algebraic statistics \cite{Studeny05:Probabilistic_CI_structures}. Given a collection of random variables and knowledge of the conditional dependencies, or independencies, among them, we can ask what are the distributions that satisfy them. In a more general setting, some of the random variables appearing in a CI model can be prescribed as unobserved (or hidden). Our goal is to determine when certain constraints on the observed variables arise from conditions on the hidden variables \cite{Steudel-Ay}. This problem can be restated algebraically by noting that probability distributions satisfying CI statements are the solutions of certain polynomial equations \cite{DrtonSturmfelsSullivant09:Algebraic_Statistics, Sullivant} which generate the so-called \textit{CI ideal}. The distributions satisfying a given collection of CI statements can be recovered by intersecting the CI ideal with the probability simplex. When there are no hidden variables, these polynomials are binomials and their associated ideals are well-studied; see e.g.~\cite{Fink,herzog2010binomial,Rauh,SwansonTaylor11:Minimial_Primes_of_CI_Ideals}. However, in the presence of hidden variables, the polynomials become far more complicated of arbitrarily high degrees and very difficult to calculate; see e.g.~\cite{pfister2019primary,clarke2020conditional}.

\medskip
Let $X, Y_1, Y_2$ be observed and $H_1, H_2$ be hidden random variables taking values in the finite sets $\MX, \MY_1, \MY_2, \MH_1, \MH_2$ of cardinalities $|\MX|=d$, $|\MY_1|=k$, $|\MY_2|=\ell$, $|\MH_1|=s-1$, $|\MH_2|=t-1$. 
Consider the CI model given by:
\begin{eqnarray}\label{eq:C}
\MC : \ \ 
\ind X {Y_1} \mid \{Y_2, H_1\} \quad \textrm{and} \quad  \ind X {Y_2} \mid \{Y_1, H_2\}.
\end{eqnarray}
Then the CI ideal associated to $\MC$ is precisely the hypergraph ideal $I_{\Delta^{s,t}}$ in \S\ref{sec:Delta_st_main}; see \cite{ollie_fatemeh_harshit, clarke2020conditional}.

\begin{example}\label{example:CI-3-4-2}
{\rm 
Let $d = k = s = t = 3$ and $\ell = 4$. 
The joint distribution of $Y_1$ and $Y_2$ has state space $\MY = \MY_1 \times \MY_2$ which is identified with the $3 \times 4$ matrix $\MY$ with values in the set $[12]$ as in \eqref{eq:Y}. 
In this case, the CI ideal $I_{\Delta} \subseteq \CC[P] := \CC[p_{x,y} : x \in \MX, y \in \MY]$ is the hypergraph ideal from Example~\ref{ex:andreas}.
The ideal $I_{\Delta}$ has two prime components. One component, associated to the matroid $M_0$ from Example~\ref{ex:andreas}, is generated by all $3$-minors of the matrix of variables $P$. This is the ideal associated to the CI statements $\ind{X}{\{Y_1, Y_2 \}} \mid H_1$ and $\ind{X}{\{Y_1, Y_2 \}} \mid H_2$. Note that $H_1$ and $H_2$ are hidden random variables taking the same number of values, and so the ideals do not distinguish them. The other prime component $I_M$ is the ideal of the configuration described in Example~\ref{ex:andreas} consisting of $12$ points and $7$ lines. Besides the rank constraints given by the original CI model, the ideal $I_M$ contains geometric constraints which are satisfied by all distributions which do not lie in $V_{M_0}$.
}
\end{example}

We note that for large values of $d$, understanding the algebraic properties of the ideal $I_{\Delta^{s,t}}$ and its (primary) decomposition is hard. In particular, Theorem~\ref{thm:hardness} shows that, up to simplification, any matroid can appear among the dependent matroids for $\Delta^{s,t}$. And, in the example above, some high-degree polynomials may appear in the generating sets of the primary components of $I_{\Delta^{s,t}}$. However, for $d \le s + t - 3$ we have shown in Theorem~\ref{thm:s-t-3} that $\Delta^{s,t}$ has a unique minimally dependent matroid $M$. So, for a generic point $p$ in the variety $V_{\Delta^{s,t}}$, the additional polynomial constraints on $p$ arise from geometric constraints on realizations of $M$.  

\begin{remark}
Suppose that $H_1$ and $H_2$ are both constant. The \textit{intersection axiom} states that any distribution satisfying $\MC$ in \eqref{eq:C} \textit{generically} satisfies $\ind X {\{Y_1, Y_2\}}$. Here, genericity means that the distributions have non-zero probabilities.
The intersection axiom has been studied for various CI models; see e.g.~\cite{herzog2010binomial, Rauh, clarke2020conditional, ollie_fatemeh_harshit, pfister2019primary}. We note that our family of hypergraph varieties include all these cases as examples. In particular, the corresponding ideal $I_{\Delta^{s,t}}$ has a distinguished prime component with a particular statistical significance, since the distributions which lie inside do not contain any structural zeros \cite{Sturmfels02:Solving_polynomial_equations}. Importantly, if we assume, without loss of generality, that $s \le t$ then this prime ideal can be realized as the CI ideal of $\ind X {\{Y_1, Y_2\}} \mid H_2$. We may therefore deduce a hidden variable version of the intersection axiom as follows:
\[
\MC = \{\ind X {Y_1} \mid \{H_1,Y_2\}, \ \ind X {Y_2} \mid \{Y_1, H_2\} \}
\implies
\ind{X}{\{Y_1, Y_2\} \mid H_2}.
\]
\end{remark}

Finally, we give connections of our work to an interesting conjecture by Matúš; see \cite{matu1999conditional}.

\begin{conjecture}[\cite{matu1999conditional}]\label{conj:matus_rational}
For any discrete conditional independence model $\MC$, there exists a distribution $p \in V(J_{\MC})$ such that all joint probabilities of $p$ are rational. 
\end{conjecture}

In Theorem~\ref{thm:hardness}, we have seen that for large enough $s,t,k,\ell$, any matroid may appear among the dependent matroids for $\Delta^{s,t}$. A natural approach is to carefully choose additional conditional independence and dependence statements for the model \eqref{eq:C}, in order to guarantee that any distribution $p$ satisfying $\MC$ is a realization of a given, realizable, matroid. 
Note that there exist matroids that are not realizable over the rationals but are realizable over a real field extension. Hence, this might lead to a characterization of CI models with hidden variables for which Conjecture~\ref{conj:matus_rational} does not hold.

\medskip
\noindent{\bf Acknowledgement.}
We would like to thank Leonid Monin for helpful conversations. We would like to thank the anonymous referees for their insightful comments.
Clarke, Mohammadi and Motwani were supported by the grants G0F5921N (Odysseus programme) and G023721N from the Research Foundation - Flanders (FWO), and the UGent BOF grant STA/201909/038. Clarke was partially supported as an overseas researcher under Postdoctoral Fellowship of Japan Society for the Promotion of Science (JSPS).

\bigskip
\newcommand{\etalchar}[1]{$^{#1}$}

\bigskip
\noindent
 \footnotesize {\bf Authors' addresses:}

 \bigskip 

 \noindent Department of Mathematics, University of Bristol, Bristol, UK \\
 E-mail address: {\tt oliver.clarke@bristol.ac.uk}
 \medskip

 \noindent School of Mathematics, University of Bristol, Bristol, UK \\
 \noindent Heilbronn Institute for Mathematical Research, Bristol, UK \\
 \noindent Current address: Department of Mathematics, Vanderbilt University, Nashville, Tennessee, USA
 \\
 E-mail address: {\tt kevin.m.grace@vanderbilt.edu}

 \medskip

   \noindent Department of Computer Science, KU Leuven, Celestijnenlaan 200A, B-3001 Leuven, Belgium\\ 
   Department of Mathematics, KU Leuven, Celestijnenlaan 200B, B-3001 Leuven, Belgium\\
 Department of Mathematics and Statistics,
 UiT – The Arctic University of Norway, 9037 Troms\o, Norway
 \\ E-mail address: {\tt fatemeh.mohammadi@kuleuven.be}

 \medskip

\noindent Department of Mathematics, Ghent University, Gent, Belgium \\
\noindent New address:  Departments of Computer Science and Mathematics, KU Leuven, Celestijnenlaan 200A, B-3001 Leuven, Belgium\\
E-mail address: {\tt harshitjitendra.motwani@ugent.be}


\begin{thebibliography}{EHHM13}

\bibitem[BC03]{bruns2003determinantal}
Winfried Bruns and Aldo Conca.
\newblock ``Gr{\"o}bner bases and determinantal ideals.''
\newblock {\em Commutative Algebra, Singularities and Computer Algebra}
  179 (2003): 9--66.

\bibitem[BS89]{bokowski1989computational}
J{\"u}rgen Bokowski and Bernd Sturmfels.
\newblock ``Computational synthetic geometry.''
\newblock {\em Lecture Notes in Math.} 1355 (1989).

\bibitem[CLO15]{CoxLittleOShea}
David~A. Cox, John Little, and Donal O’Shea.
\newblock ``Ideals, Varieties, and Algorithms: An Introduction to
  Computational Algebraic Geometry and Commutative Algebra (4th edition).''
\newblock {\em Undergrad. Texts Math.} (2015).

\bibitem[CMM21]{ollie_fatemeh_harshit}
Oliver Clarke, Fatemeh Mohammadi, and Harshit~J Motwani.
\newblock ``Conditional probabilities via line arrangements and point
  configurations.''
\newblock {\em Linear Multilinear Algebra} (2021): 1--33.

\bibitem[CMR20]{clarke2020conditional}
Oliver Clarke, Fatemeh Mohammadi, and Johannes Rauh.
\newblock ``Conditional independence ideals with hidden variables.''
\newblock {\em Adv. in Appl. Math.} 117 (2020): 102029.

\bibitem[Cra65]{Crapo65}
Henry~H. Crapo.
\newblock ``Single-element extensions of matroids.''
\newblock {\em Journal of Research of the National Bureau of Standards, Section
  B}, 69B (1965): 55--65.

\bibitem[DSS09]{DrtonSturmfelsSullivant09:Algebraic_Statistics}
Mathias Drton, Bernd Sturmfels, and Seth Sullivant.
\newblock  ``Lectures on Algebraic Statistics.''
\newblock {\em Oberwolfach Semin.} 39 (2009).

\bibitem[EHHM13]{Fatemeh}
Viviana Ene, J{\"u}rgen Herzog, Takayuki Hibi, and Fatemeh Mohammadi.
\newblock ``Determinantal facet ideals.''
\newblock {\em Michigan Math. J.} 62 (2013): 39--57.

\bibitem[Fin11]{Fink}
Alex Fink.
\newblock ``The binomial ideal of the intersection axiom for conditional
  probabilities.''
\newblock {\em J. Algebraic Combin.} 33 (2011): 455--463.

\bibitem[GGMS87]{gelfand1987combinatorial}
Israel~M. Gelfand, R.~Mark Goresky, Robert~D. MacPherson, and Vera~V.
  Serganova.
\newblock ``Combinatorial geometries, convex polyhedra, and schubert cells.''
\newblock {\em Adv. Math.} 63 no. 3 (1987): 301--316.

\bibitem[GS]{M2}
Daniel~R. Grayson and Michael~E. Stillman.
\newblock ``Macaulay2.'' A software system for research in algebraic geometry.
\newblock Available at https://faculty.math.illinois.edu/Macaulay2/.

\bibitem[Har13]{hartshorne2013algebraic}
Robin Hartshorne.
\newblock ``Algebraic geometry.''
\newblock {\em Grad. Texts in Math.} 52 (2013).


\bibitem[HHH{\etalchar{+}}10]{herzog2010binomial}
J{\"u}rgen Herzog, Takayuki Hibi, Freyja Hreinsd{\'o}ttir, Thomas Kahle, and Johannes Rauh.
\newblock ``Binomial edge ideals and conditional independence statements.''
\newblock {\em Adv. in Appl. Math.} 45 no. 3 (2010): 317--333.

\bibitem[HS04]{HSS}
Serkan Ho{\c{s}}ten and Seth Sullivant.
\newblock ``Ideals of adjacent minors.''
\newblock {\em J. Algebra}, 277 no. 2 (2004): 615--642.


\bibitem[KLS13]{knutson2013positroid}
Allen Knutson, Thomas Lam and David E. Speyer.
\newblock ``Positroid varieties: juggling and geometry.''
\newblock {\em Compos. Math} 149 no. 10 (2013): 1710--1752.


\bibitem[LV13]{lee2013mnev}
Seok~Hyeong Lee and Ravi Vakil.
\newblock ``Mn{\"e}v-{S}turmfels universality for schemes.''
\newblock {\em Clay Mathematics Proceedings}  18, A celebration of algebraic geometry, (2013): 457--468.

\bibitem[Mat99]{matu1999conditional}
Franti\v{s}ek Mat\'u\v{s}.
\newblock ``Conditional independences among four random variables {III}: Final
  conclusion.''
\newblock {\em Combin. Probab. Comput.} 8 no. 3 (1999): 269--276.

\bibitem[MF14]{M14}
Jaume Mart\'{i}-Farr\'{e}.
\newblock ``From clutters to matroids.''
\newblock {\em Electron. J. Combin.}, 21 no. 1 (2014): P1.11.

\bibitem[MFdM15]{MM15}
Jaume Mart\'{i}-Farr\'{e} and Anna de~Mier.
\newblock ``Completion and decomposition of a clutter into representable  matroids.''
\newblock {\em Linear Algebra Appl.} 472 (2015): 31--47.

\bibitem[MFdM17]{MM17}
Jaume Mart\'{i}-Farr\'{e} and Anna de~Mier.
\newblock ``Transformation and decomposition of clutters into matroids.''
\newblock {\em Adv. Math.} 312 (2017): 286--314.


\bibitem[Mn{\"e}85]{mnev1985manifolds}
Nikolai~E. Mn{\"e}v.
\newblock ``On manifolds of combinatorial types of projective configurations and convex polyhedra.''
\newblock {\em Soviet Math. Doklady} 32 (1985): 335--337.


\bibitem[Mn{\"e}88]{mnev1988universality}
Nikolai~E. Mn{\"e}v.
\newblock ``The universality theorems on the classification problem of
  configuration varieties and convex polytopes varieties.''
\newblock {\em Lecture Notes in Math.} 1346 {\em Topology and geometry—Rohlin seminar} (1988): 527--543.


\bibitem[MR18]{Fatemeh2}
Fatemeh Mohammadi and Johannes Rauh.
\newblock ``Prime splittings of determinantal ideals.''
\newblock {\em Comm. Algebra} 46 no. 5 (2018): 2278--2296.


\bibitem[Oh09]{Oh2009Combinatorics}
Suho Oh.
\newblock ``Combinatorics of Positroids.''
\newblock {\em Discrete Math. Theor. Comput. Sci.} DMTCS Proceedings vol. AK, 21st International Conference on Formal Power Series and Algebraic Combinatorics (2009).


\bibitem[Oxl11]{Oxley}
James Oxley.
\newblock ``Matroid Theory.''
\newblock Second edition. {\em Oxf. Grad. Texts Math} (2011).


\bibitem[PS19]{pfister2019primary}
Gerhard Pfister and Andreas Steenpass.
\newblock ``On the primary decomposition of some determinantal hyperedge ideal.''
\newblock {\em J. Symbolic Comput.} 103 (2019): 14--21.


\bibitem[PW70]{piff1970vector}
Mike~J. Piff and Dominic~J.A. Welsh.
\newblock ``On the vector representation of matroids.''
\newblock {\em J. Lond. Math. Soc. (2)} 2 (1970): 284--288.


\bibitem[Rau13]{Rauh}
Johannes Rauh.
\newblock ``Generalized binomial edge ideals.''
\newblock {\em Adv. in Appl. Math.} 50 no. 3 (2013): 409--414.


\bibitem[RG99]{richter1999universality}
Jürgen Richter-Gebert.
\newblock ``The universality theorems for oriented matroids and polytopes.
\newblock {\em Contemp. Math.} 223 (1999): 269--292.


\bibitem[RG11]{richter2011perspectives}
Jürgen Richter-Gebert.
\newblock ``Perspectives on Projective Geometry: A Guided Tour Through Real and Complex Geometry.''
\newblock Springer Berlin Heidelberg, 2011.


\bibitem[Ryb11]{rybnikov2011fundamental}
Grigori~L. Rybnikov.
\newblock ``On the fundamental group of the complement of a complex hyperplane arrangement.''
\newblock {\em Funct. Anal. Appl.} 45 no. 2 (2011): 137--148.


\bibitem[SA15]{Steudel-Ay}
Bastian Steudel and Nihat Ay.
\newblock ``Information-theoretic inference of common ancestors.''
\newblock {\em Entropy} 17 no. 4 (2015): 2304--2327.


\bibitem[SJS17]{sitharam2017handbook}
Meera Sitharam, Audrey~St. John, and Jessica Sidman.
\newblock ``Handbook of Geometric Constraint Systems Principles (1st edition).''
\newblock {\em Chapman \& Hall/CRC Math.} (2017).


\bibitem[ST13]{SwansonTaylor11:Minimial_Primes_of_CI_Ideals}
Irena Swanson and Amelia Taylor.
\newblock ``Minimal primes of ideals arising from conditional independence statements.''
\newblock {\em J. Algebra} 392 (2013): 299--314.


\bibitem[Stu89]{sturmfels1989matroid}
Bernd Sturmfels.
\newblock ``On the matroid stratification of {G}rassmann varieties,
  specialization of coordinates, and a problem of {N}. {W}hite.''
\newblock {\em Adv. Math.} 75 no. 2 (1989): 202--211.


\bibitem[Stu90]{sturmfels1990grobner}
Bernd Sturmfels.
\newblock ``Gr{\"o}bner bases and {S}tanley decompositions of determinantal
  rings.''
\newblock {\em Math. Z.} 205 no. 1 (1990): 137--144.


\bibitem[Stu02]{Sturmfels02:Solving_polynomial_equations}
Bernd Sturmfels.
\newblock ``Solving systems of polynomial equations.''
\newblock {\em CBMS Reg. Conf. Ser. Math.} 97 (2002).


\bibitem[Stu05]{Studeny05:Probabilistic_CI_structures}
Milan Studen\'y.
\newblock ``Probabilistic conditional independence structures.''
\newblock {\em Information Science and Statistics} Springer, London, (2005).

\bibitem[Stu08]{sturmfels2008algorithms}
Bernd Sturmfels.
\newblock ``Algorithms in invariant theory.''
\newblock {\em Texts Monogr. Symbol. Comput.} (2008).


\bibitem[STW21]{sidman2019geometric}
Jessica Sidman, Will Traves, and Ashley Wheeler.
\newblock ``Geometric equations for matroid varieties.''
\newblock {\em J. Combin. Theory Ser. A} 178 (2021): 105360.


\bibitem[Sul18]{Sullivant}
Seth Sullivant.
\newblock ``Algebraic Statistics.''
\newblock {\em Grad. Stud. Math.} (2018).


\end{thebibliography}
\end{document}